\documentclass[english,11pt, a4paper]{article}
\usepackage{graphicx} 
\usepackage[T1]{fontenc}
\usepackage{bbm}

\usepackage{amsthm,amsmath,amsfonts,amssymb}
\usepackage{graphicx}
\usepackage{enumerate}
\usepackage{authblk}
\usepackage{cite}
\usepackage{mathtools, amsmath, amssymb, amsrefs, mathrsfs}
\usepackage{url}
\usepackage[inline, shortlabels]{enumitem}
\numberwithin{equation}{section}

\let\OLDthebibliography\thebibliography
\renewcommand\thebibliography[1]{
  \OLDthebibliography{#1}
  \setlength{\parskip}{0pt}
  \setlength{\itemsep}{0pt plus 0.3ex}
}

\DeclarePairedDelimiter\floor{\lfloor}{\rfloor}

\newcommand{\R}{\mathbb{R}}
\newcommand{\N}{\mathbb{N}}
\newcommand{\C}{\mathbb{C}}
\newcommand{\Z}{\mathbb{Z}}

\newcommand{\D}{\mathbb{D}}
\newcommand{\T}{\mathbb{T}}
\newcommand{\E}{\mathbb{E}}

\renewcommand{\d}{\mathrm{d}}

\renewcommand{\i}{\mathrm{i}}

\renewcommand{\Re}{\operatorname{Re}}
\renewcommand{\Im}{\operatorname{Im}}

\renewcommand{\P}{\mathbb{P}}

\def\t{\theta}

\let \ge \geqslant
\let \geq \geqslant
\let \le \leqslant
\let \leq \leqslant

\newcommand{\x}{\mathbf{x}}
\newcommand{\y}{\mathbf{y}}
\newcommand{\dd}{\mathbf{d}}
\newcommand{\g}{\mathbf{g}}
\newcommand{\f}{\mathbf{f}}
\newcommand{\h}{\mathbf{h}}

\newtheorem{thm}{Theorem}[section]
\newtheorem{prop}[thm]{Proposition}
\newtheorem{lemma}[thm]{Lemma}

\theoremstyle{remark}
\newtheorem*{remark}{Remark}

\begin{document}

\title{Planar Coulomb gas on a Jordan arc at any temperature}
\author{K. Courteaut\footnote{Courant Institute of Mathematical Sciences, New York University, kc5733@nyu.edu}, K. Johansson\footnote{Department of Mathematics, KTH Royal Institute of Technology, kurtj@kth.se}, and F. Viklund\footnote{Department of Mathematics, KTH Royal Institute of Technology, frejo@kth.se}}
\date{}

\maketitle
\begin{abstract}
We study a planar Coulomb gas confined to a sufficiently smooth Jordan arc $\gamma$ in the complex plane, at inverse temperature $\beta > 0$. Let
\[\bar{Z}_{n}^\beta(\gamma)=\frac{Z_n^\beta(\gamma)}{((2\text{cap}(\gamma))^{\beta n^2/2}e^{(1-\beta/2)n(H(\nu_{e,\gamma})-H(\nu_{e,I}))}},\]
 be the  normalized partition function, where $\text{cap}(\gamma)$ is the capacity of $\gamma$, $H(\nu_{e,\gamma})$ the entropy of the logarithmic equilibrium measure on $\gamma$, and $I=[-1,1]$. We compute the free energy as the number of particles tends to infinity, including the constant term:
\[
\lim_{n\to \infty}\log \frac{\bar{Z}_{n}^\beta(\gamma)}{\bar{Z}_{n}^\beta(I)} = \frac{1}{24}J^A(\gamma) +  \frac{1}{8} \left( \sqrt{ \frac{\beta}{2} } - \sqrt{ \frac{2}{\beta}} \right)^2  J^{F}(\gamma).
\]
Here $\bar{Z}_{n}^\beta(I)$ is an explicit Selberg integral, $J^A(\gamma)$ is half the Loewner energy of a certain Jordan curve associated to $\gamma$ plus a covariance term, and $J^F(\gamma)$ is the Fekete energy, related to the zero temperature limit of the model.

We also prove an asymptotic formula for the Laplace transform of linear statistics for sufficiently regular test functions. As a consequence, the centered empirical measure converges to a Gaussian field with explicit asymptotic mean, and asymptotic variance given by the Dirichlet energy of the bounded harmonic extension of the test function outside of the arc.

A key tool in our analysis is the arc-Grunsky operator $B$ associated to $\gamma$, reminiscent of but different from the classical Grunsky operator. We derive several basic properties of the arc-Grunsky operator, including an estimate analogous to the strengthened Grunsky inequality and the relation to the Dirichlet integral. In our proofs $J^A(\gamma)$ arises from the Fredholm determinant $\det(I+B)$ and a substantial part of the paper is devoted to relating this to the Loewner energy. 
\end{abstract}
\tableofcontents

\section{Introduction}\label{sec:intro}

Let $\gamma$ be a rectifiable Jordan arc connecting two distinct points $a$ and $b$ in the complex plane. Consider a (logarithmic) Coulomb gas of $n$ charged particles confined to $\gamma$, interacting via the 2D electrostatic potential energy. 
Their positions $z_1,\dots,z_n$ have a joint probability density given by
\begin{equation}\label{fn}
\frac{1}{Z_{n}^\beta(\gamma) n!} \exp \Big(-\beta \sum_{1\leq k < \ell \leq n} \log |z_k -z_\ell|^{-1} \Big) 
\end{equation}
with respect to the product arc-length measure on $\gamma^n = \gamma\times \cdots \times \gamma$, and where $\beta > 0$ is the \emph{inverse temperature}.
The \emph{partition function} is
\begin{equation}\label{Zn}
Z_{n}^\beta(\gamma)=\frac 1{n!}\int_{\gamma^n}\exp \big(-\beta\sum_{1\le k < \ell \le n}\log|z_k-z_\ell|^{-1}) \prod_{k=1}^n|dz_k|.
\end{equation}
By definition, $\log Z_{n}^\beta(\gamma)$ is the \emph{free energy}. When $\beta = 2$ we write simply $Z_{n}(\gamma)$. We regard $Z_n^\beta(\gamma)$ as a functional of the arc and we are interested in its dependence on $\gamma$ as $n \to \infty$.

Given a function $g:\gamma\to\C$, consider the weighted partition function, 
\begin{equation}\label{Dn}
Z_{n}^\beta(\gamma)[e^g]=\frac 1{n!}\int_{\gamma^n} \exp\big(-\beta \sum_{1\le k < \ell \le n}\log|z_k-z_\ell|^{-1}+\sum_{k=1}^n g(z_k)\big) \prod_{k=1}^n|dz_k|,
\end{equation}
so that $Z_{n}^\beta(\gamma)=Z_{n}^\beta(\gamma)[1]$ and the ratio \begin{equation}\label{Pn}
    \frac{Z_{n}^\beta(\gamma)[e^g]}{Z_{n}^\beta(\gamma)} = \E_{n, \gamma}^\beta \left[\exp( \sum_{k=1}^n g(z_k)) \right],
    \end{equation}
is the Laplace functional of the empirical measure of the point process \eqref{fn} on the arc.

Our primary objective in this paper is to compute the asymptotics of the free energy and the Laplace functional, up to constant order, as $n\to\infty$. In particular, we provide a description of how the expansion depends on the geometry of the arc. Our approach relies on a detailed analysis of the arc-Grunsky operator, which we introduce here, and its connection to geometric and potential-theoretic quantities, notably the Loewner energy of a Jordan curve associated to the arc. These objects and their interplay seem interesting to study further. See also \cites{Joh1, Joh2, CouJoh, JohVik}.

\subsection{Main results}
\subsubsection*{Set-up and first statement}
Let $\gamma$ be a rectifiable Jordan arc with endpoints $-1$ and $1$. We write $I=[-1,1]$.
Denote by $\nu_e=\nu_{e,\gamma}$ the electrostatic equilibrium measure on $\gamma$, equivalently, $\nu_e$ is harmonic measure in $\mathbb{C} \smallsetminus \gamma$ seen from $\infty$, see Section~\ref{sec:prel}. 

 Let $\textrm{cap}(\gamma)$ be the logarithmic capacity of $\gamma$, and let
 \[
 H(\nu_{e,\gamma})=-\int_\gamma\rho_{e,\gamma}(z)\log\rho_{e,\gamma}(z)|dz|,
 \]
be the entropy of the equilibrium measure  $d\nu_{e,\gamma}(z)=\rho_{e,\gamma}(z)|dz|$.
 Define $d_0=d_0(\gamma)$ by,
\begin{equation}\label{d0ent}
\frac{d_0}{2}=-(H(\nu_{e,\gamma})-H(\nu_{e,I})).
\end{equation}
It is convenient to normalize the partition function and consider
\begin{equation}\label{Znnorm}
\bar{Z}_{n}^\beta(\gamma)=\frac{Z_n^\beta(\gamma)}{(2\text{cap}(\gamma))^{\beta n^2/2}e^{(\beta/2-1)d_0(\gamma)n/2}},
\end{equation}
which removes the leading order asymptotics. Note that $d_0(I)=0$.

In order to state our results, we will associate a Jordan curve $\eta=\eta(\gamma)$ to the arc $\gamma$. We do this by ``opening'' $\gamma$ using the map $s(z)=z + \sqrt{z^2-1}$, where the branch cut is taken along $\gamma$ and $s(z) = 2z + o(1)$ as $z \to \infty$, an idea going back to Widom \cite{Wid69}, p. 206. Then $s$ maps the complement of $\gamma$ conformally onto the domain $D^*$ exterior to a unique Jordan curve $\eta=\eta(\gamma)$ passing through $-1$ and $1$. Note that $z \mapsto (z+1/z)/2$ maps $\eta$ onto $\gamma$ and that $\eta = 1/\eta$ as sets. Let $D$ be the bounded component of $\C \smallsetminus \eta$. Write $f: \mathbb{D} = \{|z|<1\} \to D$ and $g:\mathbb{D}^* = \{|z|>1\} \to D^*$ for the conformal maps normalized to fix $0$ and $\infty$, respectively, and satisfy $f'(0)>0$ and $g'(\infty) = \lim_{z \to \infty} g'(z) > 0$. 

The \emph{Loewner energy} of $\eta$ is defined by
\begin{align}\label{def:Loewner-energy}
I^L(\eta) = \mathcal{D}_{\mathbb{D}}(\log|f'|) + \mathcal{D}_{\mathbb{D}^*}(\log|g'|) + 4 \log |f'(0)/g'(\infty)|,
\end{align}
where \[\mathcal{D}_D(u)=\frac{1}{\pi}\int_D|\nabla u|^2 d^2z\] is the Dirichlet integral, see \cite{Wan}.
While formally defined for any Jordan curve, the Loewner energy (also known as the universal Liouville action) is finite if and only if the curve is a Weil--Petersson quasicircle \cite{TakTeo}, and it is minimized if and only if the curve is a circle. It has recently attracted significant attention in part due to its links to random conformal geometry and Teichm\"uller theory, and there are many different characterizations of Weil--Petersson quasicircles, see \cites{Wan, WangSurvey, TakTeo, Bishop, VW-GAFA, VW-PLMS, Joh2, JohVik}. 
Using the Loewner energy of the opened Jordan curve $\eta=\eta(\gamma)$, we define the following quantity associated to the arc $\gamma$, 
\begin{align}\label{J-arc-Loewner-energy}
J^A(\gamma) := \frac{1}{2}I^L(\eta) - 3 \log|h'(1)  h'(-1)|,
\end{align}
where $h = g^{-1}: D^* \to \mathbb{D}^*$ is a uniformizing conformal map fixing $\infty$. We show that $J^A(\gamma)$ is non-negative and finite if $\gamma$ is sufficiently regular, see Section~\ref{sect:Loewner-energies}.

The following is our first main result. 
\begin{thm}\label{thm:main-arc-loewner}Let $\beta=2$ and suppose that $\gamma$ is an analytic Jordan arc with endpoints $-1$ and $1$. Then if $I=[-1,1]$,
\[
\lim_{n \to \infty} \log \frac{\bar{Z}_{n}(\gamma)}{\bar{Z}_{n}(I)} = \frac{1}{24}J^A(\gamma),
\]
where $J^A(\gamma)$ is as in \eqref{J-arc-Loewner-energy}, and, as $n \to \infty$, 
\begin{equation}\label{Znunitint}
Z_n(I)=2^{1/12}e^{3\zeta'(-1)}n^{-1/4}(2\pi)^n2^{-n^2+o(1)}.
\end{equation}
\end{thm}
This theorem will follow from a more general result stated below.

In our argument, $J^A(\gamma)$ first emerges as minus the logarithm of a Fredholm determinant, see Section~\ref{sect:arc-grunsky-intro}. Proving the identity involving the Loewner energy of the opened curve requires additional work, see Theorem~\ref{thm:Grunsky-Loewner2-intro}. The optimal regularity for Theorem~\ref{thm:main-arc-loewner} to hold remains an interesting open problem.

\begin{remark}
Let $\alpha$ be the  hyperbolic geodesic in the complement of $\gamma$ connecting $1$ and $-1$. Then $\gamma \cup \alpha$ is a Jordan curve which in general is different from the opened curve $\eta$. By definition $I^A(\gamma) := I^L(\gamma \cup \alpha)$ is the \emph{arc-Loewner energy} of $\gamma$, see \cite{Wan}. With the help of a result of Krusell we will explicitly relate $J^A$ and $I^A$ in Section~\ref{sect:Loewner-energies}.
\end{remark}
\subsubsection*{Statement for general $\beta$ and linear statistics}
Theorem~\ref{thm:main-arc-loewner} is a special case of a more general result which holds for any $\beta > 0$. Before stating it, we note that the partition function can be evaluated explicitly as a Selberg integral in the special case when $\gamma = I=[-1,1]$, 
 \begin{equation}\label{SelZ}
 \bar{Z}_{n}^\beta(I) =\frac{2^{\beta n^2/2+(1-\beta/2)n}\Gamma(1+\beta n/2)}{n!\,\Gamma(1+\beta/2)^n} \prod_{j=0}^{n-1} \frac{(\Gamma(1+\beta j/2))^3}{\Gamma(2+(n+j-1)\beta/2)}.
 \end{equation}
 See \cite{Forrester-Selberg}.  When $\beta=2$, this formula can be further simplified using the Vandermonde determinant, Andréief's formula, and the orthogonality of Legendre polynomials on $[-1,1]$. This gives
\begin{align}\label{eq:Zn-11}
Z_n(I) &= \prod_{j=0}^{n-1} \Big(j+\frac{1}{2}\Big)^{-1} \Big( \frac{2^j(j!)^2}{(2j)!} \Big)^2 \nonumber \\
&= 2^{1/12}e^{3\zeta'(-1)}n^{-1/4}(2\pi)^n2^{-n^2+o(1)},
\end{align}
as $n\to \infty$ which is the form we used in Theorem~\ref{thm:main-arc-loewner}, see \cite{Wid}.

We introduce some further notation in order to state our result for general $\beta>0$.
Let $z_e(t)$, $t\in I$, be the parametrization of $\gamma$ that maps the equilibrium measure of the interval $I$ to that of $\gamma$.  More precisely,  
\begin{equation}\label{eqpar}
\nu_{e,\gamma}(z_e[-1,t]) = \nu_{e,I}([-1,t])=\frac{1}{\pi} \int_{-1}^t \frac{\d s}{\sqrt{1-s^2}}. 
\end{equation}
 Note that $1/(\pi\sqrt{1-s^2})$ is the density of the equilibrium measure on $I$. We refer to $z_e$ as the \textit{equilibrium parametrization} of $\gamma$ by $I$. This parametrization plays a central role in the paper. 
 A computation shows that
 \[
 d_0 = -\frac{2}{\pi}\int_{-1}^1 \frac{\log|z_e'(t)|}{\sqrt{1-t^2}} dt,
 \]
 where $d_0$ is defined by \eqref{d0ent}.
Also, write $\tau_e : \gamma \to I$ for the inverse of $z_e$.  
If $\gamma$ is sufficiently regular, $\tau_e$ is a real-valued, continuously differentiable function on $\gamma$ with $\tau'_e > 0$. 
Let 
\[ m(z) = - \frac{1}{2}\log\left|\frac{1-z^2}{1-\tau_e(z)^2}\right| .
\]

We define the \emph{Fekete energy} of $\gamma$ by 
\begin{equation}\label{def:Fekete-energy}
J^F(\gamma) = \mathcal{D}_{\mathbb{C}\smallsetminus \gamma}\left(\log|\tau'_e| - m  \right) + 3\log| h'(-1)h'(1)| - 2(d_0 + \log(2\textrm{cap}(\gamma))),
\end{equation}
See Section~\ref{sec:mainthm}.
When $u$ is defined only on $\gamma$, we write $\mathcal{D}_{\C \smallsetminus \gamma}(u)$ for the Dirichlet integral of the bounded harmonic extension of $u$ to $\C \smallsetminus \gamma$.

We have the following generalization of Theorem~\ref{thm:main-arc-loewner}.
\begin{thm}\label{thm:main-all-beta}Let $\beta > 0$
and suppose that $\gamma$ is an analytic Jordan arc with endpoints $-1$ and $1$. Then if $I = [-1,1]$,
\[
\lim_{n\to \infty}\log \frac{ \bar{Z}_{n}^\beta(\gamma)}{\bar{Z}_{n}^\beta(I)} = \frac{1}{24}J^A(\gamma) +  \frac{1}{8} \left( \sqrt{ \frac{\beta}{2} } - \sqrt{ \frac{2}{\beta}} \right)^2  J^{F}(\gamma),
\]
where $\bar{{Z}}_{n}^\beta(I)$ is as in \eqref{SelZ}, $J^A(\gamma)$ is as in \eqref{J-arc-Loewner-energy}, and $J^F(\gamma)$ is as in \eqref{def:Fekete-energy}.
\end{thm}
The proof of Theorem~\ref{thm:main-all-beta} will be completed at the end of Section~\ref{sec:mainthm}.

Finally, we consider linear statistics of the particles on $\gamma$, for both real- and complex-valued test functions. We state here the real-valued case; the more general complex-valued case is given in Theorem \ref{thm:relSz}, which is stated using the Chebyshev coefficients of the test function and the arc-Grunsky operator introduced just below in Section~\ref{sect:arc-grunsky-intro}.

For any real $u \in C^2(\gamma)$, set 
\begin{align*}
A[u] &= \mathcal{D}_{\mathbb{C} \smallsetminus \gamma}(u), \\ M[u]& = \mathcal{D}_{\mathbb{C} \smallsetminus \gamma}(\log|\tau'_e| - m,u) + u(-1) + u(1) -2 \int_\gamma u(z) d\nu_e(z), 
\end{align*}
where $\mathcal{D}_{\mathbb{C} \smallsetminus \gamma}(u,v)$ denotes the Dirichlet inner product of the bounded harmonic extension of $u$ and $v$ to $\C \smallsetminus \gamma$.

\begin{thm}\label{thm:main-linear-statistics}
Let $\alpha, \beta > 0$ and suppose that $\gamma$ is a $C^{19+\alpha}$ Jordan arc with endpoints $-1$ and $1$. Let $u$ be a $C^{4+\alpha}$ real-valued function on $\gamma$. Then,
\[
\lim_{n \to \infty}\mathbb{E}_{n,\gamma}^\beta\left[\exp \left(\sum_{j=1}^n u(z_j) - n \int_\gamma u(z) d\nu_e(z) \right) \right] = \exp\left(\frac{1}{4\beta} A[u] + \frac{1}{2\beta}(\frac{\beta}{2}-1)M[u] \right). 
\]
\end{thm}
Theorem~\ref{thm:main-linear-statistics} is proved in Section~\ref{sec:mainthm}. 

We can interpret Theorem~\ref{thm:main-linear-statistics} as stating that the fluctuations of the empirical measure of \eqref{fn} converge in distribution to a Gaussian field on the arc, whose covariance is given by the Dirichlet inner product of the harmonic extensions of the test functions. The term $M[u]$ is a correction to the asymptotic mean that appears when $\beta\neq 2$. In the special case when $\gamma=[-1,1]$,
\[
\frac 1{2\beta}\bigg(\frac{\beta}2-1\bigg)M[u]=-\bigg(\frac 2{\beta}-1\bigg)\int_{-1}^1u(t)\bigg[\frac 14\delta(t-1)+\frac 14\delta(t+1)-\frac 1{2\pi\sqrt{1-t^2}}\bigg]\,dt.
\]
This is, up to a sign, exactly the same asymptotic correction to the mean that we get for the Gaussian $\beta$-ensemble scaled so that the asymptotic eigenvalue distribution is supported on the interval $[-1,1]$, see \cite{Joh98}.

\subsection{The arc-Grunsky operator}\label{sect:arc-grunsky-intro}
A key tool in our analysis is the \emph{arc-Grunsky operator}, which we define using the equilibrium parametrization $z_e$. Recall that the Chebyshev polynomials are the orthogonal polynomials for the equilibrium measure on $[-1,1]$.
\begin{lemma}[The arc-Grunsky expansion]\label{lem:arcGrunsky}
Let $\alpha>0$ and let $\gamma$ be a $C^{5+\alpha}$ Jordan arc with endpoints $-1$ to $1$. Then there exist real-valued coefficients $a_{kl}$, $k,l\geq 1$, such that uniformly for $s,t \in I$,
\begin{equation}\label{Grunskyexp}
\log \Big| \frac{z_e(s)-z_e(t)}{s-t} \Big| = \log(2\mathrm{cap}(\gamma)) -2 \sum_{k,l\geq 1} a_{kl} T_k(s)T_l(t),
\end{equation}
where $T_k$ is the $k$:th Chebyshev polynomial of the first kind.
\end{lemma}
The proof will be given in Section \ref{sec:Grunsky}. We call $a_{kl}$, $k,l \geq 1$, the \textit{arc-Grunsky coefficients}, in analogy with the classical Grunsky coefficients which can also be defined from $\gamma$ using the normalized conformal map from the exterior of the unit disk. The arc-Grunsky coefficients are however different from the Grunsky coefficients and we do not know any direct relation between them. 

Acting by matrix multiplication, the semi-infinite matrix $B = (b_{kl})_{k,l\geq1}$, with $b_{kl} = \sqrt{kl}a_{kl}$ induces an operator $B: \ell^2(\Z_+) \rightarrow \ell^2(\Z_+)$, the \textit{arc-Grunsky operator}. It follows from the definition that $a_{kl} = a_{lk}$, and hence $B$ is a real, symmetric operator. Moreover, we prove in Section \ref{sec:Grunsky} that it satisfies an analog of the strengthened version of the Grunsky inequality which in the classical setting holds for (and characterizes) quasicircles. 
\begin{prop}[Strengthened arc-Grunsky inequality]\label{prop:stGrineq-intro}
Let $\alpha>0$ and let $\gamma$ be a $C^{9+\alpha}$ Jordan arc with endpoints $-1$ and $1$. Let $B$ be the arc-Grunsky operator associated to $\gamma$. Then there exists a constant $0 \leq \kappa<1$ so that
\begin{equation}\label{stGrineq-intro}
B \geq -\kappa I,
\end{equation}
that is, all eigenvalues of $B$ are greater or equal to $-\kappa$.
\end{prop}
One may ask about the optimal regularity for Proposition~\ref{prop:stGrineq-intro}. The first step to prove Proposition~\ref{prop:stGrineq-intro} is the analog of the classical Grunsky inequality, i.e., \eqref{stGrineq-intro} with $\kappa=1$ which holds under weaker regularity assumptions. Our argument is based on potential theory and also gives a new proof of the classical Grunsky inequality, see Proposition~\ref{prop:Grineq}.

\begin{remark}We only have a one-sided estimate in \eqref{stGrineq-intro}, as opposed to the classical Grunsky inequality which gives both an upper and lower bound on the eigenvalues. We do not know whether a two-sided estimate  holds in the present situation. For the ordinary Grunsky operator the one-sided inequality implies the two-sided inequality since the eigenvalues come in pairs with opposite signs.
\end{remark}

The previous proposition ensures that $I+B$ is invertible if the arc $\gamma$ is sufficiently smooth. This fact is very important in our proofs, and a necessary condition for the asymptotic formulas that we obtain to be well-defined. Moreover, sufficient regularity also ensures that $B$ is trace-class in which case the Fredholm determinant $\det(I+B)$ is defined. In a similar manner as in \cites{Joh2, CouJoh, JohVik}, the $\beta$-independent constant term in the free energy expansion is expressed using this Fredholm determinant.  We have the following identity.
\begin{thm}\label{thm:Grunsky-Loewner2-intro}
    Let $\gamma$ be an analytic Jordan arc with endpoints $-1$ and $1$ and let $B$ be the arc-Grunsky operator associated to $\gamma$. Then,
\[
-\log \det(I+B) = \frac{1}{12}J^A(\gamma),
\]
where $J^A(\gamma)$ is as in \eqref{J-arc-Loewner-energy}.
 \end{thm}
 Theorem~\ref{thm:Grunsky-Loewner2-intro} is proved in Section~\ref{sec:var}. Given the theorem, the asymptotics in Theorem~\ref{thm:main-all-beta} (and Theorem~\ref{thm:main-arc-loewner}) could be expressed using only the equilibrium parametrization, $z_e$. 
 
 Theorem~\ref{thm:Grunsky-Loewner2-intro} is parallel to the following fact about the classical Grunsky operator $G$ associated to a given Jordan curve $\eta$: if $\eta$ is sufficiently regular (a Weil-Petersson quasicircle), then
\begin{equation}\label{Grunsky-Loewner-energy-identity}
-\log\det (I - GG^*) = \frac{1}{12}I^L(\eta),
\end{equation}
where $I^L$ is the Loewner energy. See \cites{TakTeo, Wan}. We do not use \eqref{Grunsky-Loewner-energy-identity} in this paper besides noting that Theorem~\ref{thm:Grunsky-Loewner2-intro} combined with \eqref{J-arc-Loewner-energy} and \eqref{Grunsky-Loewner-energy-identity} \emph{a posteriori} provides a relation between the two determinants.

We prove Theorem~\ref{thm:Grunsky-Loewner2-intro} by computing appropriate variations of both sides and integrating. This is also how \eqref{Grunsky-Loewner-energy-identity} is proved, although our argument is somewhat different and based on Douglas' formula. With some simplifying modifications it gives an alternative proof of \eqref{Grunsky-Loewner-energy-identity}. The analyticity of $\gamma$ is used in the proof of Theorem~\ref{thm:Grunsky-Loewner2-intro}, which is also why it is assumed in Theorem~\ref{thm:main-arc-loewner} and Theorem~\ref{thm:main-all-beta}. With additional work this assumption could be relaxed, but the optimal condition remains unclear.

\subsection{Example: the circular arc}

The only case apart from an interval where everything can be computed explicitly is when $\gamma$ is a circular arc. Let $C_{\alpha}$, $0<\alpha<\pi$, be the circular arc on the unit circle from $e^{\rm{i}\alpha}$ to $e^{\rm{i}(2\pi-\alpha)}$, and let $\gamma_\alpha$ be the same circular arc, but with endpoints at $-1$ and $1$, center
$\rm{i}\cot\alpha$, and radius $(\sin\alpha)^{-1}$. They are related by the map $v(\zeta)=\cos\alpha+\rm{i}\zeta\sin\alpha$ which maps $\gamma_\alpha$ to $C_\alpha$.
Setting 
\[
\sqrt{z^2-1}=e^{\i\alpha/2}\bigg(e^{\i\alpha}\frac{z+1}{z-1}\Bigg)^{1/2}(z-1)
\]
with a branch cut along the positive real axis, we get $\sqrt{z^2-1}$ with a branch cut along $\gamma_\alpha$, and the exterior map is given by
\begin{equation}\label{psicircarc}
\psi(z)=\sin\frac{\alpha}2(z+\sqrt{z^2-1})-\i\cos\frac{\alpha}2.
\end{equation}
A computation using the formulas in Section \ref{sec:prel} gives the equilibrium parametrization
\[
\tilde{z}_e(t)=2t^2\cos^2(\alpha/2)-1+2\i t\cos(\alpha/2)\sqrt{1-t^2\cos^2(\alpha/2)},
\]
$-1\le t\le 1$, for $C_\alpha$ so $z_e(t)=v^{-1}(\tilde{z}_e(t))$. From this it can be seen that the arc-Grunsky coefficients for $C_\alpha$ satisfy $a_{2k,2l}=a_{2k,2l-1}=a_{2k-1,2l}=0$, whereas $a_{2k-1,2l-1}\neq 0$ in general and does not appear to be given by a simple formula.

Opening up $\gamma_\alpha$ we get a circle and it follows from \eqref{psicircarc} that
\[
h(w)=w\sin(\alpha/2)-\i\cos(\alpha/2).
\]
Since the Loewner energy of a circle is zero, we see from \eqref{J-arc-Loewner-energy}
that
\[
J^A(\gamma_\alpha)=-6\log\sin\frac{\alpha}2.
\]
Note that $\text{cap}(C_\alpha)[u]=\cos(\alpha/2)$. Let $u(\theta)$ be a $C^{4+\alpha}$ real-valued function on $C_\alpha$. It follows from Theorem \ref{thm:main-arc-loewner}, \eqref{Znunitint} and Theorem \ref{thm:main-linear-statistics}  that
\[
Z_n(C_\alpha)[e^u]=2^{1/12}e^{3\zeta'(-1)}n^{-1/4}(2\pi)^n2^{-n^2}\bigg(\cos\frac{\alpha}2\bigg)^{n^2}\bigg(\sin\frac{\alpha}2\bigg)^{-1/4} e^{\frac 18A[u]+o(1)}
\]
as $n\to\infty$. This formula was first proved in \cite{Wid} using a special identity and orthogonal polynomials, under the assumption that $u(\theta)=u(2\pi-\theta)$, and with an expression for $A[u]$ in terms of appropriate Fourier coefficients. This symmetry condition on $u$ was removed in \cite{DuiKoz} using a completely different approach, again with an expression for $A[u]$ in terms of appropriate Fourier coefficients. Relating the expression for $A[u]$ in \cite{DuiKoz}, Formula (22), to the form given here as a Dirichlet energy is a non-trivial computation.

\subsection{Discussion }

Planar Coulomb gases confined to one-dimensional sets have been studied primarily in the case of the unit circle and the real line, in part due to their connection to random matrix theory. For example, when $\gamma$ is the unit circle and $\beta=2$, \eqref{fn} gives the joint density of the eigenvalues of a Haar (uniform) $n\times n$ unitary matrix. For the real line, a confining potential has to be added into the Gibbs measure \eqref{fn} for integrability; for instance, if it is quadratic we recover the Gaussian $\beta$-Ensemble. There is a very extensive literature on CLTs of linear statistics and asymptotics of partition functions for Coulomb gases on the line in an external potential, confined to an interval or to the unit circle. We will not attempt to summarize these papers here, and refer instead to the recent book \cite{Serfaty2024} and the references therein.

More general curves and arcs in the plane have received much less attention. 
The case of a Jordan curve in the complex plane has been studied in \cites{Joh1, Joh2} for $\beta=2$ and \cites{CouJoh, WieZab} for general $\beta>0$. In \cite{CouJoh}, it was found that,
assuming sufficient regularity of $\eta$, the free energy asymptotics has the general form,
\begin{equation}\label{eq:expansion-Jordan-curve}
\lim_{n \to \infty}\log \frac{Z_n^\beta(\eta)/\textrm{cap}(\eta)^{\beta n^2/2 + (1-\beta/2)n}}{Z_n^\beta(\mathbb{T})} = \frac{1}{24}I^L(\eta) + \frac{1}{8}\left( \sqrt{\frac{\beta}{2}} - \sqrt{ \frac{2}{\beta}} \right)^2 I^F(\eta). 
\end{equation}
The geometry of $\eta$ is encoded by the Loewner energy $I^L(\eta)$ defined in \eqref{def:Loewner-energy} and the Fekete (-Pommerenke) energy, which for a Jordan curve is defined by \[I^F(\eta):=\mathcal{D}_{\mathbb{C}\smallsetminus \eta}\left(\log|h'|\right),\]
where $h$ is the conformal map from the exterior of the curve $\eta$ onto the exterior disc. Both quantities are minimized for a circle and finite if and only if the Jordan curve belongs to the class of Weil-Petersson quasicircles. The Fekete energy is so named because it appears in the expansion of the discriminant of the Jordan curve, which is naturally interpreted as the model we consider here at $\beta = +\infty$. We refer to the discussion and references in \cite{JohVik}, where a 2D gas confined to a Jordan domain at $\beta=2$ was considered. In that case an exact formula holds, implying an expansion of similar form, but with constant term $-I^L/12$.  

Comparing \eqref{eq:expansion-Jordan-curve} with Theorem~\ref{thm:main-all-beta}, we see that that the expansions share an identical structure, even though the specific quantities involved differ. In the case of a Jordan curve, these quantities are defined directly in terms of the conformal map that sends the exterior of the unit disk to the exterior of the curve. For the arc, however, the corresponding quantities are connected to the arc's equilibrium measure and are only indirectly related to the conformal maps. They also involve covariance terms and interaction between the endpoints of the arc. Of course, in the case of a Jordan curve, the exterior conformal map also provides the equilibrium parametrization, and it is this aspect that we generalize to the setting of the arc.

In \cite{CouJoh}, it was also established that for sufficiently regular Jordan curves $\eta$ and sufficiently smooth test functions $u$ we have the following limit:
\begin{align*}
&\lim_{n\to\infty} \E_{\eta^n}^\beta \left[ \exp \left( \sum_{\mu=1}^n u(z_\mu) - n \int_\eta u(z)\, d\nu_{eq}(z) \right) \right] \\
&= \exp \left(\frac{2}{\beta}  \mathbf{u}^t (I+K)^{-1}\mathbf{u}+2\left(1-\frac{2}{\beta}\right) \mathbf{d}^t(I+K)^{-1}\mathbf{u} + o(1)\right).
\end{align*}
Here $\nu_{eq}$ is the equilibrium measure of $\eta$, $\mathbf{u}$ is an infinite-dimensional vector containing the (rescaled) Fourier coefficients of $u$ moved to the unit circle using the exterior conformal map, $\mathbf{d}$ are (rescaled) Fourier coefficients of $\log|h'|$, where $h$ is the exterior uniformizing conformal map, and $K$ is an infinite-dimensional matrix constructed from the real and imaginary parts of $B^\eta$, the classical Grunsky operator associated with $\eta$ (see (1.5) and the following equations in \cite{CouJoh} for the precise definitions). Again we see very strong similarities with the results of this paper. By comparing with Theorem~\ref{thm:main-linear-statistics} (see in particular the formulation in Theorem \ref{thm:relSz} below), we can see that the arc-Grunsky operator indeed plays the role of the classical Grunsky operator, or more precisely $K$, in the setting of open arcs. The main difference lies in the appearance of the column vector $\f$ in \eqref{Ag} which takes into account the endpoints of the arc. Such a vector does not appear in the equation above since the curve $\eta$ is a closed curve with no endpoints. It would be interesting to further study the arc-Grunsky operator and its relation to the geometry of the arc.

Let us finally note that in the case of a Jordan curve there is a direct relation between the Grunsky operator and the Neumann-Poincaré operator acting on real-valued functions on the curve. This connection does not appear to generalize to the case of the arc-Grunsky operator associated with a Jordan arc. The Neumann-Poincaré operator is given by 
\[
K_{\text{NP}}(u)(z)=\text{Re}\bigg(\frac 1{\pi\i}\text{p.v}\int_{\gamma}\frac{u(\zeta)}{\zeta-z}\,d\zeta\bigg).
\]
Consider the case when $\gamma$ is a circular arc. If we parametrize using the equilibrium parametrization and use the fact that $z_e\big(-(\cos(\alpha/2))^{-1}\cos(\theta/2)\big)=e^{-\i\theta}$, we see that
\[
K_{\text{NP}}(u)(z) \equiv \frac 1{2\pi}\int_{-1}^1u(z_e(t))\,dt,
\]
and consequently there appears to be no relation to the non-trivial arc-Grunsky operator.

\subsection*{Acknowledgements}  K. C. was supported by the Knut and Alice Wallenberg Foundation (KAW). K. J. was supported by the Knut and Alice Wallenberg Foundation (KAW), and the Swedish Research Council (VR). F. V. was supported by the Knut and Alice Wallenberg Foundation Wallenberg Scholar grant KAW 2023.0317 and the G\"oran Gustafsson Foundation for Research in Natural Sciences and Medicine. We are very grateful to Ellen Krusell for allowing us to include Proposition~\ref{prop:ellen}. We thank Yilin Wang for many discussions and useful suggestions. Thanks also to Kari Astala for correspondence regarding Section~\ref{sect:interpolation-of-arcs}, to Steffen Rohde and Paul Wiegmann for discussions, and to Olof Rubin and Anton Zabrodin for pointing out two mistakes in an earlier version of the paper. AI tools were used to proofread a final version of the paper, leading to several minor corrections.  The authors finally thank the referee for their careful reading and useful comments.
\section{Preliminaries}\label{sec:prel}

\begin{table}[h]
    \centering
    \begin{tabular}{|c|c|}
    \hline
    Notation & Meaning \\
    \hline
    $\gamma$ & Jordan arc between $-1$ and $1$, the two sides are denoted $\gamma_{\pm}$ \\
    $\Omega^*$ & $\C \smallsetminus \gamma$ \\
        $ \psi$ & Exterior conformal map $\Omega^*\to \mathbb{D}^*$  \\
        $ \psi_{\pm}(z)$ & $\psi(z_\pm)$, where $z_\pm$ denotes the two prime ends corresponding to $\gamma_\pm$\\
    $\eta$ & Opened Jordan curve corresponding to Jordan arc $\gamma$ \\
        $D, D^*$ & Bounded and unbounded component of $\mathbb{C} \smallsetminus \eta$ \\
          $j$ &  Inversion map $z \mapsto 1/z$ \\

  $f$ &  Conformal map $f:\D \to D, f(0)=0, f'(0)>0$ \\
    $g$ & Conformal map $g: \D^* \to D^*, g(\infty) = \infty, g'(\infty) > 0$  \\
  $h = h_{e}$ & $h:=g^{-1}:  D^*  \to \mathbb{D}^*$ \\
    $h_{i} = j \circ h_{e} \circ j$ &  $h_i:=f^{-1}: D  \to \mathbb{D}$ \\

        \hline
    \end{tabular}
    \caption{Table summarizing notation}
    \label{tab:sample_two_columns}
\end{table}

\subsection{Equilibrium measure}
Let $K$ be a compact set in the complex plane $\mathbb{C}$. The \emph{logarithmic energy} of a probability measure $\mu$ on $K$ is defined by
\begin{equation}\label{Imu}
I_K[\mu]=\int_K\int_K\log|z-w|^{-1}\,d\mu(z)d\mu(w).
\end{equation}
There is a unique measure $\nu_e$, the \emph{equilibrium measure} on $K$, such that
\[
I_K[\nu_e]=\inf_{\mu}I_\gamma[\mu],
\]
where the infimum is over all probability measures on $K$. It is well-known that $\nu_e$ is harmonic measure seen from $\infty$, i.e., the hitting distribution of a planar Brownian motion started from $\infty$.
The \textit{capacity} of $K$ is given by
\[ \mathrm{cap}(K) = e^{- I_K[\nu_e]}. \]
If $I = [-1,1]$, then $\mathrm{cap}(I) = 1/2$.

\subsection{Equilibrium parametrization}
Recall that $z_e(t)$, $t\in[-1,1]$, is the equilibrium parametrization of the Jordan arc $\gamma$. We will need some regularity properties of $z_e$, which we will derive from those of its inverse $\tau_e \coloneqq z_e^{-1}$. Accordingly, we begin by finding a suitable expression for $\tau_e$.

Let $\nu_+(\zeta)$ and $\nu_-(\zeta)$ denote the normals of $\gamma$ at $\zeta$ taken from the $\pm$ sides (chosen arbitrarily, but fixed once chosen). Since $\gamma$ is continuously differentiable, the equilibrium measure $\nu_e$ of $\gamma$ is absolutely continuous with respect to arc length and satisfies
\[ \d \nu_e(\zeta) = -\frac{1}{2\pi} \Big( \frac{\partial U}{\partial \nu_+}(\zeta) + \frac{\partial U}{\partial \nu_-}(\zeta) \Big) |\d \zeta | \]
where $U$ is the equilibrium potential,
\[ U(z) = \int_\gamma \log|z-\zeta|^{-1} \d \nu_e(\zeta). \]
The Green's function with singularity at $\infty$ for $\Omega^* \coloneqq \C \smallsetminus \gamma$ can be defined as
\[ G_{\Omega^*}(z,\infty) = I_\gamma[\nu_e] -U(z), \quad z\in \C\smallsetminus \gamma. \]
It is also equal to 
\[ G_{\Omega^*}(z,\infty) = \log |\psi(z)| \]
where $\psi$ is the unique conformal map from $\Omega^*$ onto $\D^*$ fixing $\infty$ and with positive derivative there (meaning, $\lim_{z\to \infty} \psi(z)/z >0$). Thus
\begin{align*}
\d \nu_e(\zeta) &= \frac{1}{2\pi} \Big( \frac{\partial }{\partial \nu_+(\zeta)}+ \frac{\partial }{\partial \nu_-(\zeta)} \Big) G_{\Omega^*}(\zeta,\infty) |\d \zeta| \\
&= \frac{1}{2\pi} \Big( \frac{\partial }{\partial \nu_+(\zeta)} + \frac{\partial }{\partial \nu_-(\zeta)} \Big) \log |\psi(\zeta)| |\d \zeta|.
\end{align*}
If $\tau(z) = \gamma'(s)/|\gamma'(s)|$ is the unit tangent at $\gamma(s)=z$ then $\nu_+ = \i \tau$, $\nu_- = -\i \tau$, and we see that
\begin{equation}\label{eq:nueq}
\d \nu_e(\zeta) = \frac{1}{2\pi}\Im \left( \tau \cdot \Big( -\frac{\psi_+'}{\psi_+}+\frac{\psi_-'}{\psi_-} \Big)  \right)|\d \zeta|
\end{equation}
where $\psi_+/\psi_-$ is the limit of $\psi$ from the left/right as $\gamma$ is traversed from $-1$ to $1$. This gives
\[ \int_a^w \d \nu_e(\zeta) = \frac{1}{2\pi} \Im \int_a^w \Big( -\frac{\psi_+'}{\psi_+}(\zeta)+\frac{\psi_-'}{\psi_-}(\zeta) \Big) \d\zeta = \frac{1}{2\pi}\big( \arg (\psi_-(w))- \arg(\psi_+(w)) \big). \]
so by definition of the equilibrium parametrization $z_e$ in \eqref{eqpar} and because $|\psi_+| = | \psi_-| = 1$,
\[ \frac{1}{\pi}\int_{-1}^t \frac{\d x}{\sqrt{1-x^2}} = \frac{1}{2\pi \i} \log \Big( \frac{\psi_-(z_e(t))}{\psi_+(z_e(t))} \Big) \]
where the imaginary part of the logarithm takes values in $[0,2\pi)$. We can now take the inverse $\tau_e$ of $z_e$ since in the definition \eqref{eqpar} the right-hand side is strictly increasing in $t$.
Consequently
\begin{equation}\label{eq:te(z)}
\tau_e(\zeta
) = -\cos \Big( \frac{1}{2\i} \log  \frac{\psi_-(\zeta)}{\psi_+(\zeta)} \Big), \quad \zeta \in \gamma. 
\end{equation}

\subsection{Jordan curves and Loewner energies associated to a Jordan arc}\label{sect:Loewner-energies}
We will associate two distinct Jordan curves to the Jordan arc $\gamma$.

We first consider the Jordan curve $\eta$ obtained by opening $\gamma$.
Let \[u(w) = \frac{1}{2}(w+\frac{1}{w})\] with inverse given by \[u^{-1}(z)=:s(z) = z+ \sqrt{z^2-1},\] where $z\mapsto \sqrt{z^2-1}$ has a branch cut along $\gamma$ and behaves like $z$ near infinity. Write $s(z_\pm)$ for the (non-tangential) limits of $s$ at the two prime ends $z_\pm$ at $z$ (corresponding to the two sides $\gamma_\pm$). Set
\[ q_+(z) = s(z_+), \quad q_-(z) = s(z_-), \quad z\in \gamma. \]
These two maps are used to ``open'' the arc $\gamma$: we concatenate the two arcs to form a Jordan curve, $\eta = \eta_- - \eta_+$ with $ \eta_- = q_-(\gamma)$ and $\eta_+=q_+(\gamma)$. Then $\eta$ is a Jordan curve such that if $j(z) = 1/z$, then (as sets) $j(\eta) = \eta$. Write $D^*$ for the unbounded component of $\C \smallsetminus \eta$. Then $s: \C \smallsetminus \gamma \to D^*$ is a conformal map fixing $\infty$ with derivative $2$ at $\infty$.

Widom proved in \cite{Wid69}, Lemma 11.1, that if the arc $\gamma \in C^{m+\alpha}$ then the opened Jordan curve $\eta \in C^{m+\alpha}$.

Let
\(h: D^* \to \D^* \)
be the unique conformal map from the exterior of $\eta$ onto the exterior of the unit disc, which maps $\infty$ to $\infty$ and has positive derivative at $\infty$.

\begin{lemma}\label{lem:eqpar}
Assume that $\gamma$ is a $C^{m+\alpha}$ Jordan arc with $m\geq 1$, $\alpha \in (0,1)$. Then the equilibrium parametrization $z_e(t)$, $-1 \le t \le 1$, is a $C^{(m+1+\alpha)/2}$ function and there exists a constant $c>0$ such that $|z_e'(t)|\geq c$ for all $t\in [-1,1]$.
\end{lemma}
\begin{proof} 
We assume without loss of generality that $\gamma$ has endpoints $-1$ and $1$ and is oriented from $-1$ to $1$. We consider the opened Jordan curve $\eta= \eta(\gamma)$. By Widom's result  (\cite{Wid69}, Lemma 11.1), we know that $\eta \in C^{m+\alpha}$.


Let $\psi$ and $h$ denote the conformal maps from the exterior of $\gamma$ and $\eta$ onto $\D^*$ fixing $\infty$ and with positive derivative there. Denote by $\phi$ and $g$ their inverses. By Kellogg's theorem (see Proposition 4.3 in \cite{GarMar}) the extension of $g$ to $\partial \D$ is $C^{m+\alpha}$ and $|g'|>0$ on $\partial \D$. Thus $h$ is $C^{m+\alpha}$ on $\eta$ so both $\psi_+ = h \circ q_+$ and $\psi_- = h \circ q_-$, the non-tangential limits of $\psi$ from the left and the right, are $C^{m+\alpha}$ away from the endpoints. By \eqref{eq:te(z)} this is also true for $\tau_e$. 

We now discuss the regularity at the endpoints and we focus on $-1$; the argument for the other endpoint is the same. Let $\sigma(t)$ be the arc length parametrization of $\gamma$ traced from $-1$ to $1$. By assumption, $\sigma \in C^{m + \alpha}$ on the closed interval $[0,L]$, where $L = |\gamma|$.
 By \eqref{eq:nueq}, we have
\[ 
\d \nu_e(z) = \frac{1}{2\pi} \Im \left( \tau(z) \left( -\frac{\psi_+'(z)}{\psi_+(z)}+\frac{\psi_-'(z)}{\psi_-(z)} \right) \right) |\d z|.  
\]
If we write $\psi_+(z) =  e^{\i x}$, then $\phi'(e^{\i x})= 1/\psi_+'(z)$ and so
\[ 
\tau(z)\frac{\psi_+'(z)}{\psi_+(z)} = \frac{-\i e^{\i x}\phi'(e^{\i x})}{|\phi'(e^{\i x})|}\frac{\psi_+'(z)}{e^{\i x}} = -\i |\psi_+'(z)|.
\]
Similarly, 
 \[\tau(z)\frac{\psi_-'(z)}{\psi_-(z)} =\i |\psi_-'(z)|. \]
 It follows that,
\[ \d \nu_e(z) = \frac{1}{2\pi}(|\psi_+'(z)|+|\psi_-'(z)|)|\d z|.  \]
On the other hand, by \eqref{eq:te(z)}, we see that for $z \in \gamma$,
\begin{align*}
\frac{\tau_e'(z)}{\sqrt{1-\tau_e(z)^2}} &= \frac{1}{2} ( |\psi_+'(z)|+|\psi_-'(z)|).
\end{align*}
Define the map,
\[
T:[0,L] \to [-1,1], \qquad T(t)=\tau_e \circ \sigma(t).
\]
We want to prove that $T$ is a diffeomorphism in $C^{(m+1+\alpha)/2}$. Given this, it follows immediately that $z_e(t) = \sigma \circ T^{-1}(t)$ has the desired regularity.
For this, note first that
\begin{align}\label{sept10.1}
\frac{T'(t)}
{\sqrt{1-T(t)^2}}=\frac{1}{2}\left(|\psi_+'(\sigma(t))|+|\psi_-'(\sigma(t))| \right).
\end{align}
We analyze the right-hand side. Recall that $s(z) = z+ \sqrt{z^2-1}$ with branch cut along $\gamma$. We can write $\psi = h \circ s$ and $\psi_\pm(\sigma(t)) = h \circ q_\pm (\sigma(t))$.
 Set for $t$ small,
\[
Q(t)=\frac{1}{t}(\sigma(t)^2-1).
\]
Then $Q$ extends to be $C^{m-1 + \alpha}$ on an interval $[0, \delta]$ and we have $Q(0) = -2\sigma'(0) \neq 0$. Let $q$ be a choice of branch for $Q^{1/2}$, and we see that $q\in C^{m-1 + \alpha}$ as well. 
By the proof of \cite{Wid69}, Lemma~11.1, if $\delta$ is small enough,
\[
\omega(r) := \sigma(r^2) + r q(r^2), \qquad r \in (-\delta, \delta), 
\]
is a $C^{m + \alpha}$ parametrization of the opened Jordan curve $\eta$ near $-1$. Moreover,
\[
u\circ \omega(r) = \sigma(r^2),
\]
where we recall that $u(w) = (w+1/w)/2$. Then, in a neighborhood of $0$,
\[
\psi_\pm \circ \sigma(r^2) = h \circ \omega(r),
\]
and since  $|\sigma'| \equiv 1$,
\[
|\psi_{\pm}' (\sigma(r^2))|=\frac{|h'(\omega(\pm r)) \omega'(\pm r)|}{2r}.
\]
If we set, 
\[
H(r) = |h'(\omega(r))\omega'(r)|,
\]
then $H$ is $C^{m-1+\alpha}$ in a neighborhood of $0$ and for $t > 0$ small,
\[
|\psi'_\pm(\sigma(t))| = \frac{H(\pm \sqrt{t})}{2\sqrt{t}}. 
\]
Define $G(t) = H(\sqrt{t}) + H(-\sqrt{t})$. Then using the regularity of $H$, we see that $G(t) \in C^{(m-1+\alpha)/2}$ in a neighborhood of $0$ and we have,
\[
    |\psi'_+(\sigma(t))|  + |\psi'_-(\sigma(t))| =  \frac{G(t)}{2\sqrt{t}}.
\]


Combining this with \eqref{sept10.1} we obtain an ODE,
\begin{align}\label{ODE}
\frac{T'(t)}{\sqrt{1-T(t)^2}}=\frac{G(t)}{4\sqrt{t}}.
\end{align}
Set
\[
\theta(t)= \arccos(-T(t)),
\]
so that,
\[
\theta'(t) =\frac{T'(t)}{\sqrt{1-T(t)^2}}=\frac{G(t)}{4\sqrt{t}}.
\]
We can write\[
\theta(t) =\sqrt{t}U(t),
\]
where
\[
U(t)=\frac{1}{4}\int_0^1 \frac{G(tv)}{\sqrt{v}} \,\d v, 
\]
and $U \in C^{(m-1+\alpha)/2}$ with $U(0) >0$.
We get
\[
T(t) = -\cos(\sqrt{t}\,U(t)).
\]
Plugging this into \eqref{ODE}, we see that,
\[
T'(t) = \frac{G(t)}{4\sqrt{t}} \sin(\sqrt{t}U(t)).
\]
It follows that
$T'\in C^{(m-1+\alpha)/2}$ and so $T\in C^{(m+1+\alpha)/2}$. We also see that $
T'(0)>0$. This completes the proof.
\end{proof}

Define the following quantity associated to the arc $\gamma$ 
\[
J^A(\gamma) := \frac{1}{2}I^L(\eta) - 3 \log| h'(1)  h'(-1)|,
\]
and $I^L(\eta)$ is the Loewner energy of the Jordan curve $\eta$ as in \eqref{def:Loewner-energy}.
The reason for introducing this particular quantity is that it appears in Theorems~\ref{thm:main-arc-loewner} and \ref{thm:main-all-beta}. Of course, some regularity is required for $J^A(\gamma)$ to be well-defined and finite. We will not consider the optimal regularity here and just note that $\gamma \in C^{3}$ certainly suffices. We show in Theorem~\ref{thm:Grunsky-Loewner2-intro} that if $\gamma$ is analytic, then
\begin{equation}\label{J-det}
J^A(\gamma)=-12\log \det(I+B),
\end{equation}
where $B$ is the arc-Grunsky operator for $\gamma$, see Section~\ref{sec:Grunsky} and Theorem~\ref{thm:main}. It is not obvious that $J^A(\gamma) \ge 0$ but we will prove this below.

There is another natural way to associate a Jordan curve and Loewner energy to $\gamma$. Let $\alpha$ be the hyperbolic geodesic in $\mathbb{C} \smallsetminus \gamma$ from $1$ to $-1$. (That is, $\alpha$ is the image of the imaginary axis under a conformal map taking $\mathbb{H} = \{z: \Im z >0\}$ onto $\C \smallsetminus \gamma$ with $0$ mapped to $1$ and $\infty$ mapped to $-1$.) Then we complete $\gamma$ using $\alpha$ to a Jordan curve $\gamma \cup \alpha$. Following Wang \cite{Wan} the \emph{arc-Loewner energy} of $\gamma$ is then defined by
\[
I^A(\gamma) := I^L(\gamma \cup \alpha).
\]
This definition is very natural from the point of view of the Loewner equation: the geodesic corresponds to the Loewner driving function being constant and the curve $\gamma \cup \alpha$ has the minimal Loewner energy among all Jordan curves containing $\gamma$ as a subarc, see \cite{Wan}. A result of Krusell relates the arc-Loewner energy to the Loewner energy of $\eta=\eta(\gamma)$.
\begin{prop}\label{prop:ellen}
    Given a Jordan arc $\gamma$ between $-1$ and $1$, let $\eta=\eta(\gamma)$ be the opened 
Jordan curve as above. Suppose that $\eta$ is smooth. Then,
    \begin{align}\label{eq:arc-energy-identity}
I^A(\gamma) = \frac{1}{2}I^L(\eta)+3\log|F'(1)F'(-1)|,\end{align}
    where $F$ is a conformal map from $\D$ onto $ D$, the bounded component of $\C\smallsetminus\eta$, fixing $-1$ and $1$. 
\end{prop}
See Appendix~\ref{app:Ellen} for the proof.
We can use this result to prove the following.
\begin{prop}
Suppose $\eta = \eta(\gamma)$ is smooth. Then,
\[J^A(\gamma)   \ge I^A(\gamma) \ge 0\]
and $J^A(\gamma) = 0$ if and only if $\gamma = [-1,1]$. 

Moreover, $J^A(\gamma) = I^A(\gamma)$ if and only if $\omega(\gamma_+) = \omega(\gamma_-)$, where $\omega$ denotes harmonic measure seen from $\infty$ in $\C \smallsetminus \gamma$ and $\gamma_\pm$ are the two sides of $\gamma$.
\end{prop}
\begin{remark}
We do not know how to derive positivity from \eqref{J-det}; we only have information about one side of the spectrum of $B$ and this is not enough, see Section~\ref{sec:Grunsky}.
\end{remark}

\begin{proof}
By definition and the chain rule,
\[J^A(\gamma)=\frac{1}{2}I^L(\eta)-3\log|{h}'(1)h'(-1)|=\frac{1}{2}I^L(\eta)+3\log| g'(\zeta_+)   g'(\zeta_-)|,\]
where $\zeta_\pm =  {h}(\pm 1)$ and $ g =   h^{-1}$ so that $  g: \mathbb{D}^* \to D^*$ and fixes $\infty$ with positive derivative there.
Let $  f:\mathbb{D} \to D$ be the conformal map fixing $0$ with positive derivative there. Then $  f = j \circ   g \circ j$. Hence
\[
  g'(z) =  f'(1/z)\frac{1}{z^2}\frac{1}{  f(1/z)^2}  =   f'(1/z) \frac{  g(z)^2}{z^2},
\]
and since $  g(\zeta_{\pm}) = \pm 1$ we see that $| g'(\zeta_\pm)| = | f'(1/\zeta_\pm)|$. Therefore
\[
J^A(\gamma) = \frac{1}{2}I^L(  \eta) + 3\log |f'(1/\zeta_+) f'(1/\zeta_-)|
\]
Let $F$ be as in Proposition~\ref{prop:ellen}. Then we can write 
\[
F = f \circ M
\]
where $M(z) = \lambda (z-a)/(1-\overline{a}z)$ is a M\"obius transformation with $|\lambda|=1$ and $a = F^{-1}(0) \in \mathbb{D}$. Hence, 
\[\log|F'(1)F'(-1)| = \log|f'(1/\zeta_+)f'(1/\zeta_-)| + \log|M'(1)M'(-1)|.\]
On the other hand, 
\begin{equation}\label{M-negative}
\log |M'(1)M'(-1)| = 2 \log \frac{1-|a|^2}{|1-a^2|} \le 0.
\end{equation}
Using \eqref{eq:arc-energy-identity} and the fact that $I^A(\gamma) \ge 0$ we see therefore see that
\begin{align}\label{J-I}    
J^A(\gamma) = I^A(\gamma) - 3\log|M'(1)M'(-1)| \ge I^A(\gamma) \ge 0.
\end{align}

It is clear that $J^A([-1,1]) = 0$ since in this case $\eta = \mathbb{T}$, the unit circle. So suppose that $J^A(\gamma) = 0$. We want to show that this implies $\gamma = [-1,1]$. Using \eqref{J-I} we have that $I^A(\gamma)=3\log|M'(1)M'(-1)|$. But $I^A(\gamma) \ge 0$ so by \eqref{M-negative} we see that $I^A(\gamma) = 0$ and it follows that $\gamma$ is a circular arc from $-1$ to $1$. This in turn implies that $\eta$ is a circle passing through $-1$ and $1$. Write $\eta_+ = \eta \cap \{z : \Im z \ge 0 \}$ for the ``upper'' part of the circle $\eta$ and note that $T_+ = F^{-1}(\eta_+)$ is a half-circle with end points $-1$ and $1$. Next we use that $\log|M'(1)M'(-1)| = 0$ which by \eqref{M-negative} implies that $a=F^{-1}(0)$ is real. But then $\omega(a,T_+, \D) = 1/2$, where $\omega$ denotes harmonic measure. So by conformal invariance, $\omega(0, \eta_+, D)=1/2$ which is only possible if $\eta_+=T_+$ and $D = \D$ so that $\eta = \mathbb{T}$. We conclude that $\gamma = [-1,1]$. 

Finally, by \eqref{J-I}, $J^A(\gamma) = I^A(\gamma)$ if and only if $\log|M'(1)M'(-1)| = 0$ and arguing as above, this holds if and only if $a$ is real. But $a$ is real if and only if $\omega(0, \eta_+, D)=1/2$. Let $\eta_- = \eta \smallsetminus \eta_+$. Since $j(\eta) = \eta$ and $j(\eta_+) = \eta_-$, $\omega(\infty, \eta_-, D^*)=\omega(0, \eta_+, D)=1/2$. Since $w(z) = (z+1/z)/2$ maps $D^*$ conformally onto $\C \smallsetminus \gamma$ with $\eta_\pm$ mapped onto $\gamma_\pm$ fixing $\infty$, the last claim follows by conformal invariance. 
\end{proof}

\subsection{Interpolation of curves}\label{sect:interpolation-of-curves}\label{sect:interpolation-of-arcs}
We now discuss how to deform the reference arc $[-1,1]$ into a prescribed analytic Jordan arc so that all arcs in the interpolating family are analytic. We do this by embedding into a holomorphic motion. We also study the induced deformation of the Jordan curves obtained by opening the arc with the map $z \mapsto z + \sqrt{z^2-1}$. We will later use these deformations to integrate the variational formulas to prove Theorem~\ref{thm:Grunsky-Loewner2-intro} in Section~\ref{sec:var}.

Let $\gamma$ be an analytic Jordan arc in $\mathbb{C}$ with end points $-1$ and $1$. Let $\gamma_0 = [-1,1]$. We construct a family of analytic Jordan arcs $(\gamma_\alpha)$, which interpolates between $\gamma_0$ and $\gamma$. By definition, there exists an open set $N$ with $\gamma_0 \subset N$ and a conformal map $f$ defined on $N$ such that $f(\gamma_0) = \gamma$. We may assume that $N$ is simply connected and that $\partial N$ is smooth. Note that $f$ fixes $-1,1$. Using for instance the Beurling-Ahlfors extension (see Section~2.4 of \cite{AJKS}), there exists  a quasiconformal homeomorphism $F:\mathbb{C} \to \mathbb{C}$ agreeing with $f$ on $N$ and conformal with $F(z) = cz+O(1)$ in a neighborhood of $\infty$, for a real $c$. In particular, $F$ fixes $\infty$. Let $\mu := \partial_{\bar z} F/\partial_z F$ be the corresponding Beltrami coefficient. Then  $\mu \equiv 0$ in $N$ and $\|\mu\|_\infty = \kappa < 1$. Let $\alpha \in \mathbb{D}$ and consider the unique quasiconformal homeomorphism $F_\alpha: \mathbb{C} \to \mathbb{C}$ which fixes $-1,1,\infty$ and satisfies the Beltrami equation,
\[
\partial_{\bar z} F_\alpha = \mu_\alpha \partial_{z}F_\alpha, \qquad \mu_\alpha = \alpha \mu.
\]
See Theorem~5.3.4 of \cite{AIM}.
In fact, if $r>1$ satisfies $r< 1/\kappa$, we may take $\alpha \in r\mathbb{D}$.
Since $\mu_\alpha \equiv 0$ in $N$, $F_\alpha$ is conformal on $N$. Define 
\begin{align}\label{def:gamma-alpha}
\gamma_\alpha := F_\alpha(\gamma_0),
\end{align}
which provides a \emph{holomorphic motion} of $[-1,1]$.

\begin{lemma}
For all $\alpha \in r\mathbb{D}$, $\gamma_\alpha=F_\alpha([-1,1])$ is an analytic Jordan arc with end points $-1,1$, $\gamma_0 = [-1,1]$ and $\gamma_1 = \gamma$. 
For each $z \in \mathbb{C}$, $ \alpha \mapsto F_\alpha(z)$ is holomorphic in $r\mathbb{D}$. For each $\alpha \in r \mathbb{D}$, $z \mapsto \partial_\alpha F_\alpha(z)$ is holomorphic in $N$. 
\end{lemma}
\begin{proof}
The first part follows immediately from the construction of $F_\alpha$ and \eqref{def:gamma-alpha}.
Holomorphicity of $\alpha \mapsto F_\alpha(z)$ follows from \cite[Corollary 5.7.5]{AIM}. For the last statement, we can use \cite[Theorem V.5]{Ahlfors} which implies that
\[
F_{\alpha+h}(z) = F_{\alpha}(z) + h\dot{F}_{\alpha}(z) + o(|h|), 
\]
where
\[
\dot{F}_{\alpha}(z) = -\frac{1}{\pi}\int_{\mathbb{C}} \mu(w) R(F_{\alpha}(w), F_{\alpha}(z))(\partial_w F_{\alpha}(w))^2d^2w
\]
and $R$ is a rational function. Since $\mu \equiv 0$ in $N$ (and has compact support) and $F_\alpha(z)$ is holomorphic in $N$ it follows that $\partial_\alpha F_\alpha(z)$ exists and is holomorphic as a function of $z$ in $N$, as claimed. 
\end{proof}

We now study the family of Jordan curves $\eta_\alpha$ obtained by opening the Jordan arcs $\gamma_\alpha$. For each $\alpha$, let $s_\alpha(z) = z + \sqrt{z^2-1}$, where the branch cut of the square root is taken along $\gamma_\alpha$ and $s_\alpha(z) = 2z + O(1/z)$ near $\infty$. Then $s_\alpha$ fixes $-1,1, \infty$ and is a conformal map of $\hat{\mathbb{C}}\smallsetminus \gamma_\alpha$ onto a simply connected domain $D_\alpha^*$ bounded by a Jordan curve $\eta_\alpha$. 
 \begin{lemma}
For all $\alpha \in r\mathbb{D}$, $\eta_\alpha$ obtained by opening $\gamma_\alpha$ using the map $z \mapsto z+ \sqrt{z^2-1}$ is an analytic Jordan curve passing through $-1,1$ and $j(\eta_\alpha)  = \eta_\alpha$. Moreover, if $\alpha=a+ib$ and $f_\alpha: \mathbb{D} \to D_\alpha$ is the normalized conformal map, then for $z$ in a neighborhood of $\mathbb{T}$, $(a,b,z) \mapsto f_\alpha(z)$ is jointly smooth (actually (real-) analytic). The analogous statement holds for the exterior conformal map $g_\alpha: \mathbb{D}^* \to D^*_\alpha$.
\end{lemma}
\begin{proof}
Recall that we write $u(z) = \frac{1}{2}(z+\frac{1}{z})$ which is a conformal map of $\mathbb{D}^*$ onto $\hat{\mathbb{C}}\smallsetminus [-1,1]$. Let $F_\alpha$ be as above. Then $G_\alpha:=s_\alpha \circ F_\alpha \circ u$ is a quasiconformal map $\mathbb{D}^* \to D^*_\alpha$ and since $s_\alpha$ is conformal, the Beltrami coefficient is given by,
\begin{align}\label{eq:Beltrami-for-G}
\mu_{G_\alpha} = \left(\frac{\overline{u'}}{u'}\right) \cdot \mu_{F_\alpha}\circ u = \alpha \left(\frac{\overline{u'}}{u'}\right) \cdot \mu_{F}\circ u 
\end{align}
Recall that we write $j(z) = 1/z$. Then by our construction, $j \circ G_\alpha = G_\alpha \circ j$ on $\mathbb{T}$. Consider for $z \in \mathbb{D}$, $H_\alpha = j \circ G_\alpha \circ j$. Then $H_\alpha$ fixes $-1,1,0$ and is a quasiconformal map $\mathbb{D} \to D_\alpha$, the inner domain of $\eta_\alpha$. We have
\[
\mu_{H_\alpha} = \left(\frac{\overline{j'}}{j'}\right) \cdot \mu_{G_\alpha} \circ j.
\]
Define \[\widetilde F_\alpha = H_\alpha 1_{\mathbb{D}} + G_\alpha 1_{\mathbb{D}^*}\] and extend the definition to $\mathbb{T}$ by continuity which is possible since $G_\alpha = j \circ G_\alpha \circ j$ on $\mathbb{T}$. Then for each $\alpha$, $\widetilde F_\alpha$ is a homeomorphism of the plane fixing $-1,1,\infty$ (and $0$), and $\widetilde F_\alpha$ is quasiconformal away from $\eta_\alpha$. Since $\gamma_\alpha$ is smooth, it is easy to see that $\eta_\alpha$ is a $C^1$ curve, so it is in particular quasiconformally removable, see \cite{Jones-Smirnov}. It follows that $\widetilde F_\alpha $ is a quasiconformal homeomorphism of the whole plane and using \eqref{eq:Beltrami-for-G} we have that the Beltrami coefficient of $\widetilde F_\alpha$ is identically $0$ in a neighborhood of $\mathbb{T}$. Moreover, $\alpha \mapsto \widetilde F_\alpha$ is holomorphic by Corollary 5.7.5 of \cite{AIM} and we see that $\eta_\alpha$ is embedded in a holomorphic motion of (a neighborhood of) $\mathbb{T}$. 

It remains to prove the statement about the conformal map $f_\alpha(z)$. This follows essentially from the argument in \cite{Rodin}, with small modifications that we now describe. We consider a small annular neighborhood $N$ of $\mathbb{T}$ on which the maps $\widetilde{F}_\alpha$ are conformal. Then $\widetilde F_\alpha$ induces a holomorphic motion of $N$. It is easy to see that we can find $\epsilon > 0$ so that $B_\epsilon=\{|z| \le \epsilon\}$ does not intersect $\tilde F_\alpha(N)$ for any $\alpha \in \mathbb{D}$. As in \cite{Rodin} (but using the extended $\lambda$-lemma, \cite[Theorem 12.3.2]{AIM}) we can extend $\tilde{F}_\alpha \mid_N$ to a holomorphic motion $\hat{F}_\alpha$ of $\hat{\mathbb{C}}$ which equals the identity on $B_\epsilon$. Let $\nu_\alpha$ be the Beltrami coefficient of $\hat{F}_\alpha$ restricted to $\mathbb{D}$. Then $\nu_\alpha=0$ in $B_\epsilon$ and $N \cap \mathbb{D}$. As in \cite{Rodin}, we find a quasiconformal map $W_0^{\nu_\alpha} : \mathbb{D} \to \mathbb{D}$ with Beltrami coefficient $\nu_\alpha$, normalized so that $\frac{d}{dz}W_0^{\nu_\alpha}(0) > 0$. We can extend $\nu_\alpha$ to $\hat{\mathbb{C}}$ by reflecting in $\mathbb{T}$ so that $\nu_\alpha(z) = \overline{\nu_\alpha(1/\bar{z})}z^2/\bar{z}^2$. We therefore see that $W_0^{\nu_\alpha}$ is the restriction to $\mathbb{D}$ of a quasiconformal map $w_\alpha: \hat{\mathbb{C}} \to \hat{\mathbb{C}}$ which preserves $\mathbb{T}$ and which is analytic in a neighborhood of $\mathbb{T}$. Write $\alpha = a+ib$. Then the Beltrami coefficient of $w_\alpha$ depends real-analytically on $a,b$, and so it follows that for $z$ fixed, $\alpha \mapsto w_\alpha(z)$ is real analytic in $a, b$ (see \cite{Rodin}). In particular, this holds for $z\in \mathbb{T}$. By the real analytic implicit function theorem (\cite[Theorem 2.3.5]{KP92}),   $\alpha \mapsto w_\alpha^{-1}(z)$ is therefore also real analytic in $a, b$ for $z \in \mathbb{T}$. By uniqueness, we can write the normalized conformal map as
\[
f_\alpha(z) = \hat{F}_\alpha \circ w_\alpha^{-1}(z), \quad z \in \overline{\mathbb{D}}.
\]
Indeed, the right-hand side maps $\mathbb{D}$ onto $D_\alpha$ conformally (and extends continuosly to $\mathbb{T}$) and has positive derivative at $0$; recall that $\hat{F}_\alpha'(0) =1$.
Since $\hat{F}_\alpha(z)$ is analytic for $\alpha \in \mathbb{D}$ and for $z \in N$, and $w_\alpha^{-1}(z)$ is real analytic in $a,b$ and analytic in a neighborhood of $\mathbb{T}$ it follows that for $z \in \mathbb{T}$ fixed, $\alpha \mapsto f_\alpha(z)$ is real analytic in $a,b$. Hartogs' theorem on separate analyticity then implies that $(a,b,z) \mapsto f_\alpha(z)$ is jointly analytic. 
 \end{proof}
\begin{remark}
 Given a general analytic Jordan curve $\eta$ separating $0$ from $\infty$, one may embed it into a holomorphic motion similarly as the arc by first extending the conformal map to a quasiconformal map, and deforming its Beltrami coefficient. It then follows directly that the extended conformal map depends analytically on $\alpha$, while one argues as above to see that the other conformal map depends real analytically on $a,b$.
    
\end{remark}

\section{Basic properties of the arc-Grunsky operator}\label{sec:Grunsky}

This section discusses the arc-Grunsky operator $B$. We will give sufficient conditions on the regularity of $\gamma$ for $B$ to be well-defined and trace class, so that the Fredholm determinant $\det(I+B)$ is well-defined, see Lemma~\ref{lem:arcGrunsky}  and Proposition~\ref{lem:trace-class}. We also prove analogs of the classical and strengthened Grunsky inequalities, see Proposition~\ref{prop:Grineq} and Proposition~\ref{prop:stGrineq}, as well as an important identity that relates the operator $(I+B)^{-1}$ to the Dirichlet energy of the harmonic extension of a given function on $\gamma$, see Proposition~\ref{prop:Dir}.  

\subsection{The arc-Grunsky coefficients and operator}

 Recall that $z_e(t)$ is the equilibrium parametrization of $\gamma$ as in Section~\ref{sec:prel}. Our first goal is to prove Lemma~\ref{lem:arcGrunsky}, that is, to show that if $\gamma$ is sufficiently smooth, then there exist real-valued coefficients $a_{kl}$, $k,l\geq 1$, such that uniformly for $s,t \in [-1,1]$,
\[\log \Big| \frac{z_e(s)-z_e(t)}{s-t} \Big| = \log(2\mathrm{cap}(\gamma)) -2 \sum_{k,l\geq 1} a_{kl} T_k(s)T_l(t),
\]
where $T_k$ is the $k$:th Chebyshev polynomial. Then by definition, $a_{kl}$ are the arc-Grunsky coefficients of $\gamma$. Given this
 we set $b_{kl} = \sqrt{kl}a_{kl}$ and let $B=(b_{kl})_{kl}$ be the corresponding operator, the arc-Grunsky operator, acting on $\ell(\mathbb{Z}_+)$ by matrix multiplication.

\begin{proof}[Proof of Lemma \ref{lem:arcGrunsky}]
We know from Lemma~\ref{lem:eqpar} that the equilibrium parametrization $z_e$ is a $C^{3+\alpha/2}$ function on $[-1,1]$. It follows that 
\begin{equation}\label{zcosexp}
\log \Big|\frac{z_e(\cos\theta)-z_e(\cos\phi)}{\cos\theta-\cos\phi} \Big| = -2 \sum_{k,l\geq 0} a_{kl} \cos(k\theta)\cos(l\phi)
\end{equation}
for some coefficients $a_{kl}\in\R$, $k,l\geq 0$, and the series converges uniformly for $\theta, \phi \in [0,2\pi]$. It follows from Frostman's theorem that
\[ \int_\gamma \log |z-\zeta| \d\nu_e(\zeta) = \log(\mathrm{cap}(\gamma)) \]
for quasi-every $z\in\gamma$ (i.e.\ everywhere on $\gamma$ except on a set of capacity zero), and hence by \eqref{eqpar}
\begin{equation}\label{Fr1}
\frac{1}{\pi}\int_{-1}^1 \log|z_e(s)-z_e(t)| \frac{\d s}{\sqrt{1-s^2}} = \log(\mathrm{cap}(\gamma)),
\end{equation}
q.e. on $[-1,1]$. Moreover, Frostman's theorem gives
\begin{equation}\label{Fr2}
\frac{1}{\pi} \int_{-1}^1 \log |s-t| \frac{\d s}{\sqrt{1-s^2}} = -\log 2,
\end{equation}
for $t\in[-1,1]$ since the capacity of $[-1,1]$ is $1/2$. The coefficients in \eqref{zcosexp} are given by
\begin{equation}\label{aklformula}
a_{kl} = -\frac{1}{2\pi^2} \int_0^\pi\int_0^\pi \log \big| \frac{z_e(\cos\theta)-z_e(\cos\phi)}{\cos\theta-\cos\phi} \big| \cos(k\theta)\cos(l\phi) \d\theta\d\phi.
\end{equation}
It follows from \eqref{Fr1} and \eqref{Fr2} that
\[ \frac{1}{\pi}\int_0^\pi \log \big| \frac{z_e(\cos\theta)-z_e(\cos\phi)}{\cos\theta-\cos\phi}\big| \d\theta = \log(2\mathrm{cap}(\gamma)), \]
for $\phi\in[0,\pi]$. From this identity we see that
\begin{align*}
a_{0l} &= -\frac{1}{2\pi^2} \int_0^\pi\int_0^\pi \log \big| \frac{z(\cos\theta)-z(\cos\phi)}{\cos\theta-\cos\phi} \big| \cos(l\phi) \d\theta\d\phi \\
&= -\frac{\log(2\mathrm{cap}(\gamma))}{2\pi} \int_0^\pi \cos(l\phi) \d\phi,
\end{align*}
which gives $a_{0l} = 0$ if $l\geq 1$. By symmetry, $a_{l0} = 0$ if $l\geq 1$. Furthermore, $a_{00} = -\frac{1}{2}\log(2\mathrm{cap}(\gamma))$. Inserting this into \eqref{zcosexp} gives \eqref{Grunskyexp}.\end{proof}
As one may expect, regularity properties of the arc $\gamma$ translate into fast decay of the arc-Grunsky coefficients $a_{kl}$ as $k$ and $l$ become large. This will be important in our proofs. We will repeatedly use the following lemma.
\begin{lemma}\label{lem:Grunskydecay}
Assume that $\gamma$ is a $C^{m+\alpha}$ Jordan arc with $m\geq 3$, $\alpha>0$. Given non-negative integers $p,q$ satisfying $p+q<m/2$, there exists a constant $A$ such that the arc-Grunsky coefficients satisfy
\begin{equation}\label{Grunskydecay}
|a_{kl}| \leq \frac{A}{k^{p+\alpha/4}l^{q+\alpha/4}},
\end{equation}
for all $k,l \geq 1$.

\end{lemma}

\begin{proof}
By Lemma~\ref{lem:eqpar} we have that the equilibrium parametrization $z_e \in C^{(m+1+\alpha)/2}$.
It follows from \eqref{Grunskyexp} that
\begin{equation*}
a_{kl}=-\frac 1{8\pi^2}\int_{-\pi}^\pi\int_{-\pi}^\pi\log\left|\frac{z_e(\cos\theta_1)-z_e(\cos\theta_2)}{\cos\theta_1-\cos\theta_2}\right|\cos k\theta_1\cos l\theta_2\,d\theta_1 d\theta_2,
\end{equation*}
so the Grunsky coefficients can be thought of as Fourier coefficients. Assume that $k\ge l$. An integration by parts argument shows that
\begin{equation*}
|a_{kl}|\le \frac A{k^{p+\alpha/2}l^q}\le \frac{A}{k^{p+\alpha/4}l^{q+\alpha/4}}
\end{equation*}
provided $p+q<(m+1)/2$ and $p+q$ odd. If $p+q$ is even we need $p+q < m/2$.
The case $k<l$ is analogous.
\end{proof}
\begin{prop}\label{lem:trace-class}
Suppose $\gamma$ is a $C^{9+ \alpha}$ Jordan arc, where $\alpha >0$. Then the arc-Grunsky operator $B$ is trace class. In particular, the Fredholm determinant $\det(I+B)$ is well defined. 
\end{prop}
\begin{proof}
    By Lemma~\ref{lem:Grunskydecay} with $p=2, q=2$, we have $|a_{k\ell}| \le A/(k^{2+\epsilon}\ell^{2+ \epsilon})$, for $\epsilon > 0$. Hence
    \[
    \sum_{k, \ell \ge 1}|b_{k \ell}| \le A\sum_{k, \ell \ge 1}\frac{1}{k^{3/2+\epsilon}\ell^{3/2+\epsilon}} <\infty,
    \]
    which implies the claim.
    \end{proof}
    The Fredholm determinant $\det(I+B)$ will be very important in our results.
\subsection{Deformed Pommerenke equation}

In the proof of Proposition~\ref{prop:Dir} below and in several other instances we will need to consider the solution $H_s$ to an integral equation closely related to the arc-Grunsky operator. Pommerenke introduced a similar equation in the context of the classical Grunsky operator when analyzing the distribution of Fekete points on a Jordan curve, see \cites{Pom67, pom69}. That equation also played an important role in \cites{Joh2, CouJoh, JohVik}. In the present setting we will need to consider a ``deformed'' version corresponding to an interpolation of partition functions, besides the change of domain. 

Let $s \in [0,1]$ be a parameter and $G$ is given. Consider the following equation,
\begin{equation}\label{Intekv1:prel}
G(\omega) = \frac{\beta s}{\pi} \Re  p.v.\int_0^\pi \frac{H_s(\theta)z_e'(\cos \theta)\sin \theta}{z_e(\cos\theta)-z_e(\cos \omega)} \d\theta + p.v.\frac{\beta(1-s)}{\pi} \int_0^\pi \frac{H_s(\theta)\sin\theta}{\cos\theta-\cos\omega}\d\theta
\end{equation}
on $[0,\pi]$. If $s=1$, this is the equation
\[ g(z_e(x)) = \frac{\beta}{\pi} \Re p.v. \int_{-1}^1 \frac{h(z_e(t))z_e'(t)}{z_e(t)-z_e(x)}\d t\]
or
\begin{equation}\label{Intekv2:prel}
g(z) = \frac{\beta}{\pi} \Re p.v. \int_\gamma \frac{h(\zeta)}{\zeta-z} \d\zeta,
\end{equation}
which, apart from a factor 2, is the same integral equation as the one considered by Pommerenke \cite{Pom} and also in \cites{CouJoh, JohVik} in the case when $\gamma$ is a closed Jordan curve. Here $h(z_e(\cos\theta)) = H_1(\theta)$.

We can rewrite \eqref{Intekv1:prel} further by using \eqref{Grunskyexp}, which gives
\begin{equation}\label{Grunskyexp2:prel}
\log \big| \frac{z_e(\cos\theta)-z_e(\cos\omega)}{\cos\theta-\cos\omega} \big| = \log(2\text{cap}(\gamma))-2\sum_{k,l\geq 1} a_{kl} \cos (k\theta)\cos(l\omega)
\end{equation}
since $T_k(\cos\theta) = \cos(k\theta)$. If we differentiate \eqref{Grunskyexp2:prel} with respect to $\theta$ we obtain
\begin{equation}\label{DGrexp:prel}
\frac{z_e'(\cos\theta)\sin\theta}{z_e(\cos\theta)-z_e(\cos\omega)}-\frac{\sin\theta}{\cos\theta-\cos\omega} = -2 \sum_{k,l\geq 1} ka_{kl}\sin(k\theta)\cos(l\omega).
\end{equation}
Inserting this into \eqref{Intekv1:prel} gives
\begin{equation}\label{Intekv3:prel}
G(\omega) = - \frac{2\beta s}{\pi} \sum_{k,l\geq 1} k a_{kl}\big( \int_0^\pi H_s(\theta) \sin(k\theta) \d\theta\big) \cos(l\omega) + \frac{\beta}{\pi} p.v. \int_0^\pi \frac{H_s(\theta)\sin\theta}{\cos\theta-\cos\omega} \d\theta.
\end{equation}
Let $h_k^s = \frac{2}{\pi} \int_0^\pi H_s(\theta) \sin(k\theta) \d\theta$ be the $k$th Fourier sine coefficient of $H_s$ so that
\[H_s(\theta) = \sum_{k\geq 1} h_k^s \sin(k\theta). \]
The conjugate function of $H_s(\theta)$ as a function of $[-\pi,\pi]$ is
\[ \widetilde{H}_s(\omega) = -\sum_{k\geq 1} h_k^s\cos(k\omega) = \frac{1}{2\pi} p.v. \int_{-\pi}^\pi \cot \big(\frac{\omega-\theta}{2}\big) H_s(\theta) \d\theta. \]
We see that
\begin{align}\label{Conjcomp}
&\frac{1}{\pi} p.v. \int_0^\pi \frac{H_s(\theta)\sin\theta}{\cos\theta-\cos\omega} \d\theta = -\frac{1}{2\pi} p.v. \int_0^\pi H_s(\theta) \big(\cot\big(\frac{\theta-\omega}{2}\big)+\cot\big(\frac{\theta+\omega}{2}\big) \big) \d\theta \nonumber\\
&= \frac{1}{2\pi} p.v. \int_{-\pi}^\pi H_s(\theta) \cot\big(\frac{\omega-\theta}{2}\big) \d\theta = \widetilde{H}_s(\omega).
\end{align}
Consequently, \eqref{Intekv3:prel} can be written
\begin{equation}\label{Intekv4:prel}
G(\omega) = -\beta s \sum_{k,l\geq 1} ka_{kl}h_k^s\cos(l\omega) + \beta\widetilde{H}_s(\omega).
\end{equation}
Next, consider the infinite column vectors
\begin{equation}\label{xtheta:prel}
\x_\theta = \big( \frac{1}{\sqrt{k}}\cos(k\theta) \big)_{k\geq 1},\quad \y_\theta = \big( \frac{1}{\sqrt{k}}\sin(k\theta) \big)_{k\geq 1}.
\end{equation}
Then
\[ G(\theta) = \x_\theta^t\g, \quad \widetilde{H}_s(\theta) = -\x_\theta^t\h^s,\]
where $\h^s = (\sqrt{k}h_k^s)_{k\geq 1}$.
We see that \eqref{Intekv4:prel} can be written 
\[ \x_\theta^t\g = -\beta s \x_\theta^t B \h^s -\beta \x_\theta^t\h^s\]
where $B$ is the arc-Grunsky operator and $\g$ is as in \eqref{gvec:prel}. Hence the integral equation can be written 
\begin{equation}\label{Intekv5:prel}
-\frac{1}{\beta} \g = (I+sB)\h^s.
\end{equation}
We will use the strengthened arc-Grunsky inequality proved in the next section to show that \eqref{Intekv5:prel} has a unique solution in $\ell^2(\mathbb{Z}_+)$ if $\gamma$ is sufficiently regular.

\subsection{The arc-Grunsky inequalities}
The classical Grunsky inequality can be interpreted as the contraction property of the corresponding operator acting on  $\ell^2(\mathbb{Z}_+)$. For the arc-Grunsky operator we only have a one-sided bound.
\begin{prop}\label{prop:Grineq}
Assume that $\gamma$ is $C^{5+\alpha}, \alpha > 0$, and let $B$ be the arc-Grunsky operator associated to $\gamma$. Then,
\begin{equation}\label{Grineq}
B \geq -\mathrm{I},
\end{equation}
i.e., all eigenvalues of $B$ are greater than or equal to $-1$.
\end{prop}
For the proof we will use a classical result from potential theory. Let $\sigma$ be a signed measure on the Jordan arc $\gamma$. The $1$-Riesz potential of $\sigma$ is defined by
\[ U_1(z;\sigma) = \int_\gamma \frac{\d\sigma(\omega)}{|z-\omega|},\quad z\in \C\smallsetminus\gamma. \]
We have the following result, see \cite{Lan}.
\begin{lemma}\label{thm:enpot}
If $\sigma(\gamma) = 0$ then
\begin{equation}\label{Isigma}
I_\gamma[\sigma] = \frac{1}{2\pi} \int_\C U_1(x+iy;\sigma)^2\d x\d y
\end{equation}
so $I_\gamma[\sigma]\geq 0$, with equality if and only if $\sigma \equiv 0$.
\end{lemma}
\begin{proof}[Proof of Proposition \ref{prop:Grineq}]
Let $f\in L^1(\nu_e)$ satisfy 
\begin{equation}\label{meanf}
\int_\gamma f\d\nu_e = 0,
\end{equation}
and let $\sigma$ be the signed measure $\d\sigma = f\d\nu_e$. It follows from \eqref{Isigma} that 
\[ \int_\gamma\int_\gamma f(z)f(w) \log |z-w|^{-1} \d\nu_e(z)\d\nu_e(w) \geq 0. \]
Inserting the equilibrium parametrization into this formula gives the inequality
\begin{equation}\label{fineq1}
\int_{-1}^1\int_{-1}^1 f(z_e(s))f(z_e(t)) \log |z_e(s)-z_e(t)| \frac{\d s}{\sqrt{1-s^2}}\frac{\d t}{\sqrt{1-t^2}} \leq 0.
\end{equation}
Choose
\[ f(z_e(\cos\theta)) = \sum_{m=1}^M mf_m\cos(m\theta). \]
Then,
\[ \int_\gamma f \d\nu_e = \frac{1}{\pi}\int_{-1}^1 f(z(t)) \frac{\d t}{\sqrt{1-t^2}} = \frac{1}{\pi} \int_0^\pi f(z_e(\cos\theta))\d\theta = 0, \]
so \eqref{meanf} is satisfied. From \eqref{fineq1} we get
\begin{equation}\label{fineq2}
\frac{1}{\pi^2}\int_0^\pi\int_0^\pi f(z_e(\cos\theta))f(z_e(\cos\phi)) \log |z_e(\cos\theta)-z_e(\cos\phi)|\d\theta\d\phi \leq 0.
\end{equation}
Let
\[ F(\omega) = \int_0^\omega f(z_e(\cos\theta))\d\theta = \sum_{m=1}^M f_m\sin(m\omega), \]
and note that $F(0) = F(\pi) =0$. An integration by parts shows that \eqref{fineq2} can be written as
\[ \frac{1}{\pi^2}\int_0^\pi\int_0^\pi F(\theta)F'(\phi) \frac{z_e'(\cos\theta)\sin\theta}{z_e(\cos\theta)-z_e(\cos\phi)}\d\theta\d\phi\leq 0. \]
Using \eqref{Grunskyexp}, 
\begin{equation}\label{Grunskyexp2:1}
\log \big| \frac{z_e(\cos\theta)-z_e(\cos\omega)}{\cos\theta-\cos\omega} \big| = -2\sum_{k,l\geq 1} a_{kl} \cos (k\theta)\cos(l\omega)
\end{equation}
since $T_k(\cos\theta) = \cos(k\theta)$. Differentiating \eqref{Grunskyexp2:1} with respect to $\theta$ gives
\begin{equation}\label{DGrexp:1}
\frac{z_e'(\cos\theta)\sin\theta}{z_e(\cos\theta)-z_e(\cos\omega)}-\frac{\sin\theta}{\cos\theta-\cos\omega} = -2 \sum_{k,l\geq 1} ka_{kl}\sin(k\theta)\cos(l\omega).
\end{equation}
Hence, we use \eqref{DGrexp:1} to obtain inequality
\begin{multline*}
    -\frac{2}{\pi^2} \sum_{k,l\geq 1} k a_{kl} \big( \int_0^\pi F(\theta)\sin(k\theta) \d\theta \big) \big( \int_0^\pi F'(\theta)\cos(l\theta)\d\theta \big) \\ \leq -\frac{1}{\pi^2} \int_0^\pi\int_0^\pi F(\theta)F'(\phi)\frac{\sin\theta}{\cos\theta-\cos\phi}\d\theta\d\phi.
\end{multline*}  
We recognize the Hilbert transform of $F$:
\begin{align}\label{Conjcomp:1}
\frac{1}{\pi} p.v. \int_0^\pi \frac{F(\theta)\sin\theta}{\cos\theta-\cos\omega} \d\theta & = -\frac{1}{2\pi} p.v. \int_0^\pi F(\theta) \big(\cot\big(\frac{\theta-\omega}{2}\big)+\cot\big(\frac{\theta+\omega}{2}\big) \big) \d\theta \nonumber\\
&= \frac{1}{2\pi} p.v. \int_{-\pi}^\pi F(\theta) \cot\big(\frac{\omega-\theta}{2}\big) \d\theta = :\widetilde{F}(\omega).
\end{align}

Using \eqref{Conjcomp:1}, it then follows that (see also Section~\ref{sec:interpolation-partition-function})
\[-\frac{1}{2}\sum_{k,l=1}^M kla_{kl} f_kf_l \leq \frac{1}{\pi} \int_0^\pi F'(\phi)\widetilde{F}(\phi) \d\phi = \frac{1}{2} \sum_{k=1}^M kf_k^2. \]
Let $x_k = \sqrt{k} f_k$. The last inequality can then be written
\[ \sum_{k,l=1}^M (\delta_{kl}+b_{kl})x_kx_l \geq 0. \]
Since $M$ and $(x_k)_{k=1}^M$ are arbitrary this proves that $B\geq -I$.
\end{proof}
\begin{remark}
The method to prove Proposition~\ref{prop:Grineq} can be used to give a new proof of the classical Grunsky inequality.
\end{remark}
Next, we turn to the proof of the strengthened arc-Grunsky inequality.

\begin{prop}\label{prop:stGrineq}
Assume that $\gamma$ is $C^{9+\alpha}, \alpha > 0$, and that $B$ is the arc-Grunsky operator associated to $\gamma$. Then there exists a constant $0 \leq \kappa<1$ so that
\begin{equation}\label{stGrineq}
B \geq -\kappa I,
\end{equation}
i.e., all eigenvalues of $B$ are greater or equal to $-\kappa$.
\end{prop}
\begin{proof}
Since we assume that $\gamma\in C^{9+\alpha}$, it follows from Lemma~\ref{lem:Grunskydecay} with $p=3$, $q=1$, that 
\begin{equation}\label{bklest2}
|b_{kl}| \leq \frac{A}{k^{5/2+\epsilon}l^{1/2+\epsilon}}
\end{equation}
where $\epsilon = \alpha/4$. Thus $\sum_{k,l\geq 1} b_{kl}^2<\infty$, that is, $B$ is a Hilbert-Schmidt operator on $\ell^2(\Z_+)$. Hence, it is a compact operator and has a discrete bounded spectrum $\{\lambda_j\}_{j=1}^\infty$ with $0$ as the only accumulation point. Thus, there is an eigenvalue $\lambda_\text{min}=\lambda_k$ such that $\lambda_j\ge \lambda_\text{min}$ for all $j$. Let $\mathbf{v}\in \ell^2(\Z_+)$ be the eigenvector with $\| \mathbf{v} \|_2 =1$ corresponding to $\lambda_\text{min}$. It follows from the arc-Grunsky inequality \eqref{Grineq} that $\lambda_{min} \geq -1$. Assume that $\lambda_{min} = -1$ so that
\begin{equation}\label{Bv}
(B+I)\mathbf{v} = 0.
\end{equation}
This implies that
\[ v_k = -\sum_{l\geq 1}  b_{kl}v_l,\]
and hence, by \eqref{bklest2} and the Cauchy-Schwarz inequality,
\begin{equation}\label{vkest}
|v_k| \leq \big( \sum_{l\geq 1} b_{kl}^2 \big)^{1/2} \| v \|_2 \leq \frac{A}{k^{5/2+\epsilon}}.
\end{equation}
Define the function $V$ by
\[ V(\theta) = \sum_{k\geq 1} v_k \sin(k\theta), \quad \theta\in[-\pi,\pi].  \]
Then the argument that leads from \eqref{Intekv4:prel} to \eqref{Intekv5:prel} above gives

\begin{equation}\label{veqn}
0 = -\sum_{k,l\geq 1} k a_{kl} v_k \cos(l\omega) + \widetilde{V}(\theta). \
\end{equation}
Define the function $v$ on $\gamma$ by
\[ v(z(\cos\theta)) = V(\theta). \]
Then, the argument leading from \eqref{Intekv1:prel} with $\beta= s = 1$ to \eqref{Intekv4:prel} shows that \eqref{veqn} gives
\[ \frac{1}{\pi} \Re p.v. \int_0^\pi \frac{V(\theta) z_e'(\cos\theta) \sin\theta}{z_e(\cos\theta)-z_e(\cos\omega)}\d\theta =0, \]
i.e.
\begin{equation}\label{veqn2}
\frac{1}{\pi} \Re p.v. \int_{-1}^1 \frac{v(z_e(t))z_e'(t)}{z_e(t)-z_e(s)}\d t =0
\end{equation}
for $s\in[-1,1]$. Write $u(t) = v(z_e(t))$ so that $u(\cos\theta) = V(\theta)$, and hence
\begin{equation}\label{uprimeexp}
-u'(\cos\theta)\sin\theta = V'(\theta) = \sum_{k\geq 1} kv_k \cos(k\theta),
\end{equation}
which can be written
\[ u'(t) = -\frac{1}{\sqrt{1-t^2}} \sum_{k\geq 1} k v_kT_k(t), \]
$t\in[-1,1]$. It follows from \eqref{vkest} that
\begin{equation}\label{uprime}
|u'(\theta)| \leq \frac{C}{\sqrt{1-t^2}},\quad t\in[-1,1].
\end{equation}
It follows from \eqref{veqn2} that
\begin{align*}
0 &= \Re p.v. \int_{-1}^1 u(t) \frac{\d}{\d t} \log (z_e(t)-z_e(s)) \d t\\
&= -\Re p.v. \int_{-1}^1 u'(t) \log (z_e(t)-z_e(s)) \d t \\
&= p.v. \int_{-1}^1 u'(t) \log |z_e(t)-z_e(s)|^{-1} \d t 
\end{align*}
for $s\in[-1,1]$, since $u(-1)=u(1)=0$. Thus
\begin{equation}\label{ueqn}
\int_{-1}^1\int_{-1}^1 u'(s)u'(t) \log|z_e(s)-z_e(t)|^{-1} \d s\d t=0.
\end{equation}
Define the function $f(z)$ on $\gamma$ by
\[ f(z(t)) = \frac{u'(t)}{|z_e'(t)|}, \quad t\in[-1,1].  \]
It follows from \eqref{uprime} that
\[ \int_\gamma |f(z)| |\d z| = \int_{-1}^1 |u'(t)| \d t \leq C \int_{-1}^1 \frac{\d t}{\sqrt{1-t^2}} = C\pi. \]
Consequently, $\d\sigma(z) = f(z)|\d z|$ gives a well-defined signed measure on $\gamma$, and
\[ \int_\gamma \d\sigma(z) = \int_\gamma f(z)| \d z| = \int_{-1}^1u'(t) \d t = u(1)-u(-1) =0. \]
It follows from \eqref{ueqn} and the definitions of $f$ and $\sigma$ that
\[ I[\sigma] = \int_\gamma\int_\gamma \log |z-w|^{-1} \d\sigma(z)\d\sigma(w) =0. \]
By Lemma~\ref{thm:enpot}, this implies that $\sigma=0$, i.e., $f=0$. But then $u'=0$ and by \eqref{uprimeexp}, $\mathbf{v} =0$, which contradicts $\lambda_{min} =-1$. Thus, $\kappa \coloneqq |\min(\lambda_{min},0)| < 1$ and $B\geq -\kappa I$.

\end{proof}

\subsection{The arc-Grunsky operator and the Dirichlet energy}
The following proposition is used in many places in the present paper. A similar result holds for the classical Grunsky operator, see \cites{CouJoh, JohVik}. Let $u:\gamma \rightarrow \C$ be a given sufficiently regular function. Let

\begin{equation}\label{gzt:prel}
u_k = \frac{2}{\pi}\int_{-1}^1 u(z_e(t))T_k(t) \frac{dt}{\sqrt{1-t^2}}, \quad k\geq 1,
\end{equation}

be Chebyshev coefficients of $u\circ z_e$ and write
\begin{equation}\label{gvec:prel}
{\mathbf u} = (\sqrt{k} u_k)_{k\geq1}.
\end{equation}
\begin{prop}\label{prop:Dir}
Assume that $u$ is a real-valued $C^{1+\alpha}$ function defined on a $C^{9+\alpha}$ Jordan arc $\gamma$. Then
\begin{equation}\label{Dir}
{\mathbf u}^t(I+B)^{-1}{\mathbf u} = \mathcal{D}_{\C \smallsetminus \gamma}(u),
\end{equation}
where $u$ is as in \eqref{gvec:prel}.
\end{prop}

\begin{proof}
The proof is analogous to that of Proposition 1.2 in \cite{CouJoh}. See also Section~2 of \cite{JohVik}. Consider the integral equation \eqref{Intekv2:prel} written using Chebyshev coefficients. Using the strengthened arc-Grunsky inequality Proposition~\ref{prop:stGrineq} we see that a unique solution exists: \[
(I+B)^{-1}{\mathbf u}=-\beta \h.
\]
Write $U(\theta)=u(z_e(\cos\theta))$ and $H(\theta)=H_1(\theta)=h(z_e(\cos\theta))$, where $h$ is the solution to \eqref{Intekv2:prel}. Recall the definitions \eqref{xtheta:prel}.
Then we can write $U(\theta)=\x_\theta^t{\mathbf u}$,
$H(\theta)=\y_\theta^t\h$.
If we let $J=(\frac 1k\delta_{jk})$, then
\[
\frac 2{\pi}\int_0^\pi\x_\theta\x_\theta^t\,d\theta=J
\]
and we see that
\begin{align*}
{\mathbf u}^t(I+B)^{-1}{\mathbf u}&=-\beta{\mathbf u}^t\h\\
& =-\beta{\mathbf u}^tJJ^{-1}\h\\
& =-\frac{2\beta}\pi\int_0^\pi {\mathbf u}^t\x_\theta\x_\theta^tJ^{-1}\h\,d\theta\\
&=-\frac{2\beta}\pi\int_0^\pi U(\theta)H'(\theta)\,d\theta.
\end{align*}
Next recall the definition  of the single-layer potential of the function $f$ on $\gamma$,
\[
\mathcal{S}(f)(z)=\frac 1{2\pi}\int_\gamma\log|\zeta-z|^{-1}f(\zeta)\,|d\zeta|.
\]
If the curve $\gamma$ is directed from $a$ to $b$ with complex tangent $\tau(z)$ at $z\in\gamma$, then $\nu_+(z)=\i\tau(z)$ and $\nu_-(z)=-\i\tau(z)$ are the two normal directions to $\gamma$ at $z$.
We then have the formula
\begin{equation}\label{fS}
f(z)=-\Big(\frac{\partial \mathcal{S}(f)}{\partial\nu_+}(z)+\frac{\partial \mathcal{S}(f)}{\partial\nu_-}(z)\Big)
\end{equation}
for $z\in\gamma$. As in the proof of Proposition 1.2 in \cite{CouJoh} it follows from \eqref{Intekv2:prel}, and an integration by parts argument using $h(\pm 1)=0$, that
\[
u(z)=2\beta \mathcal{S}(\frac{\partial h}{\partial \tau})(z),
\]
for $z\in\gamma$, and since the single-layer potential is continuous across $\gamma$ and harmonic in $\C\smallsetminus\gamma$, this gives a harmonic extension $u_e$ of $u$ to $\C\smallsetminus\gamma$.
It follows from \eqref{fS} that
\[
\frac{\partial u}{\partial\nu_+}(z)+\frac{\partial u}{\partial\nu_-}(z)=-2\beta\frac{\partial h}{\partial \tau}(z),
\]
for $z\in\gamma$. We see that $\zeta(\theta)=z_e(\cos\theta)$, $0\le\theta\le\pi$, parametrizes $-\gamma$, and consequently
\[
H'(\theta)=\frac{d}{d\theta}h(\zeta(\theta))=-|\zeta'(\theta)|\frac{\partial h}{\partial \tau}(\zeta(\theta)).
\]
Thus,
\begin{align*}
{\mathbf u}^t(I+B)^{-1}{\mathbf u}&=-\frac{2\beta}\pi\int_0^\pi U(\theta)H'(\theta)\,d\theta\\
& =\frac{2\beta}\pi\int_0^\pi u(\zeta(\theta))\frac{\partial h}{\partial \tau}(\zeta(\theta))|\zeta'(\theta)|\,d\theta\\
&=\frac{2\beta}\pi\int_\gamma u(\zeta)\frac{\partial h}{\partial \tau}(\zeta)\,|d\zeta|\\
& =-\frac 1{\pi}\int_\gamma u(\zeta)\Big(\frac{\partial u}{\partial\nu_+}(\zeta)+\frac{\partial u}{\partial\nu_-}(\zeta)\Big)\,|d\zeta|
\\
&=\frac 1{\pi}\int_\C|\nabla u_e|^2\,dxdy,
\end{align*}
where the last equality follows from Green's theorem applied to the domain $\C\smallsetminus\gamma$.\end{proof}

\section{Variational formulas and proof of Theorem~\ref{thm:Grunsky-Loewner2-intro}}\label{sec:var}

In this section we prove the variational formulas which will allow us to connect the Fredholm determinant involving the arc-Grunsky operator, see Proposition~\ref{lem:trace-class}, with the Loewner energy of the opened Jordan curve, see Section~\ref{sect:Loewner-energies}.

Recall that
\[ \psi: \C\smallsetminus \gamma \to \D^*, \quad h: D^*\to \D^* \]
are the unique conformal maps from the exterior of $\gamma$ and $\eta$ onto the exterior of the unit disc, which map $\infty$ to $\infty$ and have positive derivative at $\infty$. Set \[\varphi = \psi^{-1}, \quad g = h^{-1}.\] Recall the definition of $q_\pm$ from Section~\ref{sect:Loewner-energies}.

The function $\psi$ has limit $\psi_+$ and $\psi_-$ as we approach $\gamma$ from the left and right sides and 
 \[\psi_+ = h\circ q_+, \quad \psi_- = h \circ q_-.\]

\subsection{Variation of the Dirichlet integrals}
Let $(\eta_\alpha)_{\alpha \in [0,1]}$ be a family of analytic Jordan curves as in Section~\ref{sect:interpolation-of-curves}. We do not assume $\eta$ to necessarily be symmetric with respect to reflection in $\mathbb{T}$. We stress that we here restrict attention to real $\alpha$ and $\partial_\alpha$ means partial derivative with respect to $\alpha$. See \cites{WieZab, TakTeo} for related computations.

We often omit the $\alpha$-subscript when the meaning is clear from context. 

\begin{lemma}\label{lem:dirichlet-integral}
Suppose $f=u+iv$ is analytic in a neighborhood of $\mathbb{D}$. Then 
\begin{align}\label{eq:dirichlet-energy-interior}
 \mathcal{D}_\mathbb{D}(u)=\frac{1}{\pi}\int_0^{2\pi} u(e^{i\theta})\partial_\theta v(e^{i\theta}) d\theta.
\end{align}
Suppose $f=u+iv$ is analytic in a neighborhood of $\mathbb{D}^*$. Then 
\begin{align}\label{eq:dirichlet-energy-exterior}
 \mathcal{D}_\mathbb{D^*}(u)=-\frac{1}{\pi}\int_0^{2\pi} u(e^{i\theta})\partial_\theta v(e^{i\theta}) d\theta.
\end{align}
\begin{proof}
The formulas follow by performing an integration by parts and applying the Cauchy-Riemann equations.
\end{proof}

\end{lemma}
The following lemma is proved by direct computation.
\begin{lemma}\label{lem:curvature}
Let $\theta \mapsto \eta(\theta)$ be a $C^3$ Jordan curve. 
If $\kappa(z)$ denotes the curvature of $\eta$ at $z$, then
\[
\Im \frac{\eta''(\theta)}{\eta'(\theta)} = |\eta'(\theta)|\kappa(\eta(\theta)).
\]
In particular, if $\eta(\theta) = f(e^{i\theta})$ where $f: \mathbb{D} \to D$ is a conformal map, then
\[
\Re e^{i\theta}\frac{f''(e^{i\theta})}{f'(e^{i\theta})} + 1 = |f'(e^{i\theta})| \kappa(f(e^{i\theta})). 
\]
Moreover,
\begin{equation}\label{schwarzian-curvature}
-\Im(e^{2 i \theta} Sf(e^{i\theta})) = |f'(e^{i\theta})|^2 \frac{d}{ds}\kappa(s),
\end{equation}
where $Sf$ is the Schwarzian derivative. 
\end{lemma}

Recall that we write $f_\alpha:\mathbb{D} \to D_\alpha$, $g_\alpha : \D^* \to D_\alpha^*$ for the interior and exterior map, respectively. Define
\begin{align}\label{def:rho}
\rho_i(z) = \rho_{\alpha, i}(z) =\frac{1}{2\pi}\Re \frac{\partial_\alpha f_\alpha(z)}{zf_\alpha'(z)}, \qquad  \rho_e(z) = \rho_{\alpha, e}(z)  = \frac{1}{2\pi}\Re \frac{\partial_\alpha g_\alpha(z)}{zg_\alpha'(z)}.
\end{align}
The $1/2\pi$-factor is included to match the usual notation for the Loewner-Kufarev equation. Geometrically, $2\pi|f'(z)|\rho(z)$ gives the normal component (relative to the curve) of $\partial_\alpha f(z)$ (at a fixed $z$) and similarly for $g$. In general, we do not have $\rho > 0$ which means the deformation is not necessarily monotone as in the setting of Loewner-Kufarev evolution. 
\begin{lemma}\label{lem:conformal-radii}
    We have
    \[
    \frac{\d}{\d \alpha} \log|f_\alpha'(0)| = \int_{\mathbb{T}}\rho_i(z)|dz|, \qquad \frac{\d}{\d \alpha} \log|g_\alpha'(\infty)| = \int_{\mathbb{T}}\rho_e(z)|dz|.
    \]
\end{lemma}
\begin{proof}
Using integration by parts,
\begin{align*}
\int_{\mathbb{T}}\rho_i(z)|dz| & = \frac{1}{2\pi}\Re \int_{\mathbb{T}} \frac{\partial_\alpha f(z)}{zf'(z) }  \frac{dz}{iz} \\
& = \frac{1}{2\pi}\Re \int_{\mathbb{T}} \partial_z \left(\frac{\partial_\alpha f(z)}{f'(z) } \right)  \frac{dz}{iz} \\
& =  \frac{1}{2\pi}\Re \int_{\mathbb{T}} \left(\frac{\partial_\alpha f'(z)}{f'(z) }  - \frac{\partial_\alpha f(z)f''(z)}{f'(z)^2 }\right)  \frac{dz}{iz} \\
& =  \frac{\d}{\d \alpha} \log|f'(0)|,
\end{align*}
where the last step follows from a residue computation using the normalization of $f$ at $0$. The proof for $\log|g'(\infty)|$ is similar.
\end{proof}
\begin{prop}\label{prop:var-dirichlet}
Let $(\eta_\alpha)_{\alpha \in [0,1]}$ be a family of analytic Jordan curves as in Section~\ref{sect:interpolation-of-curves}. 
Then for all $\alpha \in (0,1)$,
\[
\frac{\d}{\d \alpha} \mathcal{D}_\mathbb{D}(\log|f_\alpha'|) = 4\int_{\mathbb{T}}\left( \Re z^2 S f_\alpha(z) + \left(|f_\alpha'(z)|\kappa(f_\alpha(z))\right)^2 -1\right) \rho_{\alpha,i}(z)|dz|
\]
and
\[
\frac{\d}{\d \alpha} \mathcal{D}(\log|g_\alpha'|) = -4\int_{\mathbb{T}}\left( \Re z^2 S g_\alpha(z) + \left(|g_\alpha'(z)|\kappa(g_\alpha(z))\right)^2 -1\right) \rho_{\alpha,e}(z)|dz|,
\]
where $S$ denotes the Schwarzian derivative and $\kappa$ curvature and $\rho_{\alpha, i}, \rho_{\alpha,e}$ are as in \eqref{def:rho}. 
\end{prop}

\begin{proof}
Write $U(z) = U_\alpha(z) = \log f_\alpha'(z)$ and $U = u+iv$. We drop the $\alpha$-subscript from now on. Let
\begin{align}\label{eq:H}
H_i(z) =  \frac{\partial_\alpha f(z)}{zf'(z)}.
\end{align}
Then $H_i$ is analytic in $\mathbb{D}$ and $2\pi\rho_i = \Re H_i$. 
We drop the $i$-index from now on. 
By \eqref{eq:dirichlet-energy-interior},
\[
\frac{\d}{\d \alpha} \mathcal{D}(\log|f'|)  = \frac{1}{\pi}\frac{\d}{\d \alpha}\int_0^{2\pi}u(e^{i\theta}) \partial_\theta v(e^{i\theta}) d\theta.
\]
Set 
\[
I:= \frac{\d}{\d \alpha} \int_0^{2\pi}u(e^{i\theta}) \partial_\theta v(e^{i\theta}) d\theta.
\]
It follows from \eqref{eq:H} that
\[
\partial_\alpha U = \frac{\partial_\alpha f'}{f'} = \frac{\partial_z(z f' H)}{f'} = zHU'+(zH)'.
\]
Hence,
\begin{align*}
I & =  - \Im \int_0^{2\pi} \partial_\alpha U(e^{i\theta})\partial_\theta \overline{U(e^{i\theta})} d\theta \\
& = \Im \int_{\mathbb{T}}U' \overline{\left(zHU'+(zH)' \right)} dz =: I_1+I_2,
\end{align*}
where
\[
I_1 = \Im \int_{\mathbb{T}} |U'|^2\overline{zH} dz. 
\]
Next, we compute
\begin{align*}
I_2 & = \Im \int_{\mathbb{T}} U' \, \overline{(zH)'}  dz\\
& = -\Im \int_0^{2\pi}   e^{2i\theta} U' \, \partial_\theta \overline{ (zH)}   d\theta \\
& = \Im \int_{\mathbb{T}}   (z^2U')'  \, \overline{zH}  dz \\
& = \Re \int_{\mathbb{T}} (z^2U')' \, \overline{H}  |dz|
\end{align*}
where we used integration by parts for the third step.
Note that $H(z)$ and $G(z):=(z^2U'(z))'$ are both analytic in $\mathbb{D}$ and $H(0)G(0)=0$. Therefore, a residue computation shows that
\[
I_2 = 2 \int_{\mathbb{T}} \Re (z^2U')' \, \Re H  |dz|.
\]
We can write 
\[
z^2U'' = z^2Sf(z) + \frac{1}{2}(zU')^2
\]
so
\[
(z^2U')' =z^2U'' +  2zU'  = z^2S f(z) + \frac{1}{2}(zU')^2 +  2zU'.
\]
It follows that
\[
2\Re (z^2U')' = \Re(2z^2Sf(z) + (zU')^2 +  4zU').
\]
We now consider $I= I_1+I_2$. Since $|U'|^2=|zU'|^2$, completing the square,
\begin{align*}
2\Re (z^2U')' + |U'|^2 &= \Re(2z^2Sf(z)) + 2(\Re zU')^2+ 4 \Re zU' \\
& = \Re(2z^2Sf(z)) + 2(\Re zU' + 1)^2-2.
\end{align*}
By Lemma~\ref{lem:curvature},
\[
\Re zU'+1 = |f'(z)| \kappa(f(z)),
\]
so using also that $\Re H = 2\pi \rho$ we get
\[
I = 4\pi \int_{\mathbb{T}} \left( \Re(z^2Sf(z)) + \left(|f'(z)| \kappa(f(z)) \right)^2 - 1 \right)\rho(z) |dz|,
\]
as claimed. The computation for $\mathcal{D}_{\mathbb{D}^*}(\log|g'|)$ is similar, starting from \eqref{eq:dirichlet-energy-exterior}.
\end{proof}
\subsection{Variation of the arc-Grunsky determinant and proof of Theorem~\ref{thm:Grunsky-Loewner2-intro}}
We now turn to computing the variation of the arc-Grunsky Fredholm determinant. For the classical Grunsky operator related computations go back at least to work of Schiffer, see, e.g., \cite{Schiffer-Hawley}. More recently, variations in the setting of Weil-Petersson quasicircles were obtained in \cite{TakTeo}. In contrast to those works, we do not have access to representations in terms of Fredholm eigenvalues or in terms of an operator acting on a space of analytic functions. We instead follow a different route based on Douglas' formula. 

In this section we will continue to omit the $\alpha$-subscript when it is clear from the context, and write $B=B_\alpha$, $\gamma = \gamma_\alpha$, and $h= h_\alpha$ etc. Recall that for a given Jordan arc $\gamma$ between $-1$ and $1$, $\eta$ denotes the  opened Jordan curve with complementary domains $ D,  D^*$ and $h = h_e:   D^* \to \mathbb{D}^*$ and $h_i :  D \to \mathbb{D}$ are the exterior and interior uniformizing maps, respectively.

\begin{prop}\label{prop:variation-of-Grunsky} Let $(\gamma_\alpha)_{\alpha \in [0,1]}$ be an analytic family of Jordan arcs between $-1$ and $1$ and let $(\eta_\alpha)_{\alpha \in [0,1]}$ be the corresponding family of opened Jordan curves as in Section~\ref{sect:interpolation-of-arcs}. Let $B_\alpha$ be the arc-Grunsky operator for $\gamma_\alpha$. Then for all $\alpha \in (0,1)$,
\begin{equation}
    \frac{\d}{\d\alpha} \log \det (I+B_\alpha) = \frac{1}{24}(I_e -I_i) + \frac{1}{4}  \frac{\d}{\d\alpha} \log|h'_\alpha(1)h'_\alpha(-1)|,
\end{equation}
where,
\[
I_i =  4\int_{\mathbb{T}} \left(\Re z^2 S f_\alpha(z) + \left(|f_\alpha'(z)|\kappa(f_\alpha(z))\right)^2 \right)\rho_{\alpha, i}(z)|dz|
\]
and
\[
I_e =  4\int_{\mathbb{T}} \left(\Re z^2 S g_\alpha(z) + \left(|g_\alpha'(z)|\kappa(g_\alpha(z))\right)^2 \right)\rho_{\alpha, e}(z)|dz|,
\]
where $S$ denotes the Schwarzian derivative and $\kappa$ curvature and $\rho_{\alpha, i}, \rho_{\alpha,e}$ are as in \eqref{def:rho}. 
\end{prop}
Given this proposition we can complete the proof of Theorem~\ref{thm:Grunsky-Loewner2-intro}.
\begin{proof}[Proof of Theorem~\ref{thm:Grunsky-Loewner2-intro} assuming Proposition~\ref{prop:variation-of-Grunsky}]
Let $\gamma$ be an analytic Jordan arc between $-1$ and $1$. We embed $\gamma$ in a family of analytic arcs $(\gamma_\alpha)_{\alpha \in [0,1]}$ so that $\gamma_0 = [-1,1]$ and $\gamma_1 = \gamma$, with associated opened Jordan curves $\eta_\alpha$. Let $B_\alpha$ be the arc-Grunsky operator for $\gamma_\alpha$. Then by Lemma~\ref{lem:conformal-radii} and Proposition~\ref{prop:var-dirichlet} 
\begin{align*}
\frac{\d}{\d\alpha}I^L(\eta_\alpha)  = & \frac{\d}{\d\alpha}\left(\mathcal{D}_{\D}(\log|f_\alpha'|)+(\mathcal{D}_{\D}(\log|g_\alpha'|) + 4 \log|f'_\alpha(0)/g'_\alpha(\infty)|\right) \\ 
 =  \, & 4\int_{\mathbb{T}} \left(\Re z^2 S f_\alpha(z) + \left(|f_\alpha'(z)|\kappa(f_\alpha(z))\right)^2 \right)\rho_i(z)|dz| \\ &- 4\int_{\mathbb{T}} \left(\Re z^2 S g_\alpha(z) + \left(|g_\alpha'(z)|\kappa(g_\alpha(z))\right)^2 \right)\rho_e(z)|dz|\\
 =: & \, I_i - I_e.
\end{align*}
On the other hand, by Proposition~\ref{prop:variation-of-Grunsky},
\[
-12\frac{\d}{\d\alpha} \log \det (I + B_\alpha) =\frac{1}{2}(I_i-I_e) - 3 \frac{\d}{\d\alpha}\log|h_\alpha'(1)h_\alpha'(-1)|.
\]
Recall that $J^A(\gamma_\alpha) = \frac{1}{2}I^L(\eta_\alpha) - 3 \log|h_\alpha'(1)h'_\alpha(-1)|$.
Hence, since $J^A([-1,1]) = 0 =  \log \det(I+B_0) $, by integration,
\[
-12 \log \det (I+B) = J^A(\gamma),
\]
and we are done.
\end{proof}

Let $\gamma_\alpha$ and $\eta_\alpha$ be as in Proposition~\ref{prop:variation-of-Grunsky}.
Let $\alpha \in [0,1]$ and define
\begin{equation}\label{def:lambda}
\lambda(w) = \lambda_\alpha(w)= \frac{h(w)}{h(1/w)}.
\end{equation}
Then $\lambda$ is holomorphic in a neighborhood of $\eta_\alpha$. The first step in the proof of Proposition~\ref{prop:variation-of-Grunsky} is the following lemma.
\begin{prop}\label{prop:var}
   Let $(\gamma_\alpha)_{\alpha \in [0,1]}$ be an analytic family of Jordan arcs between $-1$ and $1$ and let $(\eta_\alpha)_{\alpha \in [0,1]}$ be the corresponding family of opened Jordan curves as in Section~\ref{sect:interpolation-of-arcs}. Let $B_\alpha$ be the arc-Grunsky operator for $\gamma_\alpha$. Then for all $\alpha \in (0,1)$,
\begin{equation}\label{eq:V'}
    \frac{\d}{\d\alpha} \log \det (I+B_\alpha) = \Re \frac{1}{6\pi \i} \int_{\eta_\alpha} V(w)(Sh)(w) \d w + \frac{1}{4} \Re (V'(1)+V'(-1))
\end{equation}
where for $\lambda$ as in \eqref{def:lambda}
\[V(w) = \frac{\partial_\alpha \lambda(w)}{\lambda'(w)}\]
and $S$ denotes the Schwarzian derivative.
Moreover,
\begin{equation}\label{eq:V'2}
V'(\pm 1) = \frac{\partial_\alpha h'(\pm 1)}{h'(\pm 1)}-\frac{\partial_\alpha h(\pm 1)}{h(\pm 1)},    
\end{equation}
and
\begin{equation}\label{eq:intS}
\int_{\eta_\alpha} V(w) (Sh)(w) \d w = \int_{\gamma_\alpha} \partial_\alpha z_e (\tau_e(\zeta)) \big(S\psi_{-}-S\psi_{+} \big)(\zeta) \d \zeta.   
\end{equation}

\end{prop}

We begin by establishing several lemmas needed for the proof of Proposition~\ref{prop:var}. 

\begin{lemma}\label{lemma1section9}
    Let $T(w) \coloneqq \tau_e(u(w))$ for $w\in \eta$. Then,
    \[V(w) = \frac{\partial_\alpha T(w)}{T'(w)}, \quad \partial_\alpha z_e(T(w)) = \frac{1}{2} (\frac{1}{w^2}-1)V(w)\]
for all $w\in \eta$.
\end{lemma}

\begin{proof}
By definition of $u$ and $q_\pm$,
\begin{equation*}
    \psi_-(u(w)) = h(q_-\circ u(w)) = \begin{cases}
        &h(w)\ \mathrm{if}\ w\in \eta_- \\
        &h(1/w)\ \mathrm{if}\ w\in \eta_+
    \end{cases}
\end{equation*}
and \begin{equation*}
    \psi_+(u(w)) = h(q_+\circ u(w)) = \begin{cases}
        &h(1/w)\ \mathrm{if}\ w\in \eta_- \\
        &h(w)\ \mathrm{if}\ w\in \eta_+
    \end{cases}
\end{equation*}
so we see that
\begin{equation*}
    c(w) \coloneqq \frac{\psi_-(u(w))}{\psi_+(u(w))} = \begin{cases}
        &\frac{h(w)}{h(1/w)}\ \mathrm{if}\ w\in \eta_- \\
        &\frac{h(1/w)}{h(w)}\ \mathrm{if}\ w\in \eta_+.
    \end{cases}
\end{equation*}
Recall that $\tau_e$ denotes the inverse of the equilibrium parametrization and that in the beginning of Section \ref{sec:prel} we showed that
 \begin{equation}\label{eq:tesec9}
     \tau_e(\zeta) = - \cos \big( \frac{1}{2\i} \log \frac{\psi_-(\zeta)}{\psi_+(\zeta)} \big),\quad \zeta\in \gamma.
 \end{equation}
It follows from \eqref{eq:tesec9} that
\[T(w) = -\frac{1}{2}\big(c(w)^{1/2}+c(w)^{-1/2} \big)\]
which in turn gives
\begin{equation}\label{eq:Vformula}
\frac{\partial_\alpha T(w)}{T'(w)} = \frac{\partial_\alpha c(w)}{c'(w)} = \frac{\partial_\alpha h(w) h(1/w)-h(w)\partial_\alpha h(1/w)}{h'(w)h(1/w)+h(w)h'(1/w)w^{-2}} = V(w).
\end{equation}
For the second identity, note that
\[ z_e(T(w)) = z_e(\tau_e(u(w))) = u(w)\]
and differentiating gives
\[ \partial_\alpha z_e(T(w)) = -u'(w) \frac{\partial_\alpha T(w)}{T'(w)} = \frac{1}{2} (\frac{1}{w^2}-1)V(w). \]
\end{proof}

\begin{lemma}\label{lemma2section9}
For $\zeta \in \gamma\smallsetminus \{-1,1\}$,
\begin{equation}
    \lim_{z\to \zeta\pm} \frac{V(u^{-1}(z))}{(u^{-1})'(z)} = \frac{V(q_\pm(\zeta))}{q_\pm'(\zeta)} = -\partial_\alpha z_e(\tau_e(\zeta))
\end{equation}
where $z\to \zeta\pm$ denotes the limit to (the prime end) $\zeta\in \gamma$ from the $\pm$ sides of $\gamma$.
\end{lemma}

\begin{proof}
We have that
\[\lim_{z\to \zeta\pm} u^{-1}(z) = q_\pm(\zeta)\] 
where $q_\pm(\zeta) = \zeta\pm \sqrt{\zeta^2-1}$ and 
\[\lim_{z\to \zeta\pm} (u^{-1})'(z) = q_\pm'(\zeta)= \pm \frac{q_\pm (\zeta)}{\sqrt{\zeta^2-1}}.\]
By Lemma \ref{lemma1section9}
\[V(w) = - \frac{2w}{w-w^{-1}} \partial_\alpha z_e(\tau_e(u(w)))\]
and hence, for $\zeta \in \gamma\smallsetminus \{-1,1\}$,
\[ \lim_{z\to \zeta\pm} V(u^{-1}(z)) = -\frac{2q_\pm(\zeta)}{q_\pm(\zeta)-q_\pm^{-1}(\zeta)}\partial_\alpha z_e(\tau_e(\zeta)) = \mp \frac{q_\pm (\zeta)}{\sqrt{\zeta^2-1}} \partial_\alpha z_e(\tau_e(\zeta)) .  \]
\end{proof}

Thus, $\partial_\alpha z_e(\tau_e(\zeta))$ has an analytic extension to a neighborhood of $\gamma$, except possibly at $\pm 1$. The key formula used to prove Proposition~\ref{prop:var} is provided by the following lemma.

\begin{lemma}\label{lemma3section9}
Let 
\[f(\zeta_1,\zeta_2) =  \frac{\partial_\alpha z_e(\tau_e(\zeta_1))-\partial_\alpha z_e(\tau_e(\zeta_2)) }{\zeta_1-\zeta_2}\]
and
$F(\zeta_1,\zeta_2) = f(\zeta_1,\zeta_1)+f(\zeta_2,\zeta_2)-2f(\zeta_1, \zeta_2)$ for $\zeta_1, \zeta_2\in \gamma$. Then for $\phi = \psi^{-1} : \D^* \to \mathbb{C} \smallsetminus \gamma$,
\begin{equation}\label{eq:tr(i+b)0}
  \mathrm{Tr} \Big[(I+B)^{-1} \frac{\d B}{\d \alpha} \Big] = \Re \frac{1}{(2\pi \i)^2} \int_\T  \int_\T  \frac{F(\phi(z_1),\phi(z_2))}{(z_1-z_2)^2} \d z_1\d z_2.  
\end{equation}
\end{lemma}

\begin{proof}
From the arc-Grunsky expansion,
\[ \log \Big|\frac{z_e(s)-z_e(t)}{s-t}  \Big| = \log(2\mathrm{cap}(\gamma)) -2\sum_{k,l\geq 1} a_{kl} T_k(s) T_l(t) \]
and $b_{kl} = \sqrt{kl}a_{kl}$, we get
\[\frac{\d b_{kl}}{d \alpha} = -2 \Re \int_{-1}^1\int_{-1}^1 \frac{\partial_\alpha z_e(s)-\partial_\alpha z_e(t)}{z_e(s)-z_e(t)} \frac{\sqrt{k}T_k(s)}{\pi\sqrt{1-s^2}}\frac{\sqrt{l}T_l(t)}{\pi \sqrt{1-t^2}} \d s \d t. \]
For $0<r<1$ we define the column vector
\[\mathbf{a}_r(s) = \Big( \frac{\sqrt{k}T_k(s)}{\pi\sqrt{1-s^2}} r^k \Big)_{k\geq1}\]
so that 
\[\frac{\d B}{\d \alpha} = -\lim_{r \to 1} 2 \int_{-1}^1\int_{-1}^1 \mathbf{a}_r(s) \mathbf{a}_r(t)^t f(z_e(s), z_e(t)) \d s \d t.\]
Hence
\begin{align}\label{eq:tr(i+b)}
    \mathrm{Tr} \Big[(I+B)^{-1} \frac{\d B}{\d \alpha} \Big] &= -\lim_{r \to 1} 2 \int_{-1}^1\int_{-1}^1  \mathrm{Tr}(I+B)^{-1} \mathbf{a}_r(s) \mathbf{a}_r(t)^t f(z_e(s), z_e(t)) \d s \d t\nonumber \\
    &= -\lim_{r \to 1} 2 \int_{-1}^1\int_{-1}^1  \mathbf{a}_r(t)^t(I+B)^{-1} \mathbf{a}_r(s)  f(z_e(s), z_e(t)) \d s \d t.
\end{align}
For each $s\in[-1,1]$, we define the function $g_s: \gamma\mapsto \R$ by
\[ g_s(z_e(t)) = \sum_{k\geq 1} \frac{T_k(s)}{\pi\sqrt{1-s^2}}r^k T_k(t), \quad t\in[-1,1] \]
and let $\mathscr{G}_s$ be the bounded harmonic extension of $g_s$ to $\C\smallsetminus \gamma$. Then, by Proposition \ref{prop:Dir},
\[ \mathbf{a}_r(t)^t(I+B)^{-1} \mathbf{a}_r(s) = \frac{1}{\pi} \int_\C \nabla \mathscr{G}_s \cdot \nabla \mathscr{G}_t \d x \d y. \]
Let $\mathcal{G}_s$ be the harmonic extension of $g_s(\phi(z))$ to $\D^*$. By the conformal invariance of the Dirichlet energy,
\[\frac{1}{\pi} \int_\C \nabla \mathscr{G}_s \cdot \nabla \mathscr{G}_t \d x \d y = \frac{1}{\pi} \int_{\D^*} \nabla \mathcal{G}_s \cdot \nabla \mathcal{G}_t \d x \d y \]
and by Douglas' formula,
\begin{align*}
&\frac{1}{\pi} \int_{\D^*} \nabla \mathcal{G}_s \cdot \nabla \mathcal{G}_t \d x \d y \\ &= \frac{1}{8\pi^2} \int_0^{2\pi} \int_0^{2\pi} \frac{(g_s(\phi(e^{\i \theta}))-g_s(\phi(e^{\i\omega})))(g_t(\phi(e^{\i \theta}))-g_t(\phi(e^{\i\omega})))}{\sin^2\frac{\theta-\omega}{2}} \d\theta \d \omega.    
\end{align*}
Inserting this into \eqref{eq:tr(i+b)} and changing the order of integration gives integrals of the form
\begin{multline*}
\Re \int_{-1}^1\int_{-1}^1 f(z_e(s), z_e(t)) g_s(\phi(e^{\i x}))g_t(\phi(e^{\i y})) \d s \d t = \Re \int_{-1}^1\int_{-1}^1 f(z_e(s), z_e(t)) \\ \cdot \sum_{j,k\geq 1} r^{j+k} T_j(s)T_j(\tau_e(\phi(e^{\i x})) T_k(t)T_k(\tau_e(\phi(e^{\i y})) \frac{\d s}{\pi \sqrt{1-s^2}} \frac{\d t}{\pi \sqrt{1-t^2}}.
\end{multline*}
So by using that $\Re f(z_e(s),z_e(t)) = -2\sum_{l,m\geq 1} \frac{\d}{\d\alpha}a_{lm} T_l(s)T_m(t)+\frac{\d }{\d \alpha} \log (2\mathrm{cap}(\gamma))$ the integral becomes
\[-\frac{1}{2} \sum_{j,k\geq 1} \frac{\d a_{jk}}{\d \alpha} r^{j+k} T_j(\tau_e(\phi(e^{\i x})))T_k(\tau_e(\phi(e^{\i y}))) \]
which approaches 
\[ \frac{1}{4} \Re f(\phi(e^{\i x}), \phi(e^{\i y}))- \frac{1}{4}\frac{\d }{\d \alpha} \log (2\mathrm{cap}(\gamma))\]
as $r\to 1-$. Thus, \eqref{eq:tr(i+b)} becomes 
\begin{align*}
\mathrm{Tr} \Big[(I+B)^{-1} \frac{\d B}{\d \alpha} \Big] &= -\Re \frac{1}{16\pi^2} \int_0^{2\pi} \int_0^{2\pi}\frac{F(\phi(e^{\i \theta}),\phi(e^{\i\omega}))}{\sin^2\frac{\theta-\omega}{2}} \d\theta\d\omega \\ 
 &= -\Re \frac{1}{4\pi^2} \int_0^{2\pi} \int_0^{2\pi}\frac{F(\phi(e^{\i \theta}),\phi(e^{\i\omega}))}{|e^{\i \theta}-e^{\i \omega}|^2} \d\theta\d\omega
\end{align*}
which is equivalent to \eqref{eq:tr(i+b)0}.

\end{proof}

We are now in position to prove Proposition~\ref{prop:var}. We will use the two-point Schwarzian, which for a function $f$ holomorphic at $z_1,z_2$ is defined by
\[Sf(z_1,z_2) = \frac{f'(z_1)f'(z_2)}{(f(z_1)-f(z_2))^2}-\frac{1}{(z_1-z_2)^2}.\]
Note that $\lim_{z_2\to z_1} Sf(z_1,z_2) = \frac{1}{6} Sf(z_1)$.

\begin{proof}[Proof of Proposition~\ref{prop:var}]
In \eqref{eq:tr(i+b)0} we make the change of variables $z_1 = h(w_1)$, $z_2= h(w_2)$, which gives
\begin{equation}\label{eq:tr1}
\mathrm{Tr} \Big[ (I+B)^{-1}\frac{\d B}{\d \alpha} \Big] = \Re \frac{1}{(2\pi \i)^2} \int_{\eta} \int_{\eta}F(u(w_1),u(w_2)) \frac{h'(w_1)h'(w_2)}{(h(w_1)-h(w_2))^2} \d w_1\d w_2
\end{equation}
since $\phi\circ h = u$. From Lemma \ref{lemma1section9} we obtain
\begin{equation}\label{eq:fu}
   f(u(w_1), u(w_2)) = \frac{\partial_\alpha z_e(T(w_1))-\partial_\alpha z_e(T(w_2))}{u(w_1)-u(w_2)} = \frac{(\frac{1}{w_1^{2}}-1)V(w_1)-(\frac{1}{w_2^2}-1)V(w_2)}{(w_1-w_2)(1-\frac{1}{w_1w_2})}. 
\end{equation}
For $1-\epsilon < r_1 < r_2$, with $\epsilon >0$ small we let $\eta_{r_i} = g(C_{r_i})$, where $C_{r_i}$ is the circle $|z| = r_i$. Then $\eta_{r_1}$ is a Jordan curve strictly inside the Jordan curve $\eta_{r_2}$. Equation \eqref{eq:fu} shows that the integrand in \eqref{eq:tr1} is analytic in a neighborhood of $\eta\times \eta$ so we can deform the contours of integration in \eqref{eq:tr1} to obtain
\begin{align*}
\mathrm{Tr} \Big[ (I+B)^{-1}\frac{\d B}{\d \alpha} \Big] &= \Re \frac{1}{(2\pi \i)^2} \int_{\eta_{r_1}} \int_{\eta_{r_2}} F(u(w_1),u(w_2)) \frac{h'(w_1)h'(w_2)}{(h(w_1)-h(w_2))^2} \d w_1\d w_2. \\
&= I_1 + I_2 
\end{align*}
where 
\begin{align*}
&I_1 = \Re \frac{1}{(2\pi \i)^2} \int_{\eta_{r_1}} \int_{\eta_{r_2}} F(u(w_1),u(w_2)) S h(w_1,w_2) \d w_1\d w_2 \\
&I_2 = \Re \frac{1}{(2\pi \i)^2} \int_{\eta_{r_1}} \int_{\eta_{r_2}} \frac{F(u(w_1),u(w_2))}{(w_1-w_2)^2} \d w_1\d w_2 
\end{align*}
and $S$ denotes the two-point Schwarzian derivative. $Sh(w_1,w_2)$ is analytic as a function of $w_1$ and $w_2$ in $D^*$, the exterior of $\eta$, and has no pole at infinity so we see that
\[ I_1 = -2\Re \frac{1}{(2\pi \i)^2} \int_{\eta_{r_1}} \int_{\eta_{r_2}} \frac{(\frac{1}{w_1^{2}}-1)V(w_1)-(\frac{1}{w_2^2}-1)V(w_2)}{(w_1-w_2)(1-\frac{1}{w_1w_2})} S h(w_1,w_2) \d w_1\d w_2 \]
where we used \eqref{eq:fu}. Choose $r_1=1$, $r_2=r>1$. We see that
\[ \frac{1}{2\pi \i} \int_{\eta_r} \frac{w_2(\frac{1}{w_1^2}-1)V(w_1)}{(w_1-w_2)(w_2-\frac{1}{w_1})} S h(w_1,w_2) \d w_2 = 0 \]
since we have no pole in $w_2$ outside $\eta_r$. Note that $w_1$, $1/w_1 \in \eta$, which is inside $\eta_r$. Hence,
\[ I_1 = \Re \frac{2}{(2\pi \i)} \int_{\eta_r} \frac{w_2(\frac{1}{w_2^2}-1)V(w_2)}{w_2-\frac{1}{w_2}} \frac{1}{6} S h(w_2) \d w_2 = \Re \frac{1}{6\pi\i} \int_{\eta} V(w) S h(w) \d w. \]
Next we consider $I_2$. Since $\int_{\eta_{r_1}} (w_1-w_2)^{-2}\d w_1 =0$, we see that
\begin{align*}
I_2 &= - \Re \frac{2}{(2\pi\i)^2} \int_{\eta} \int_{\eta_r} \frac{(\frac{1}{w_1^{2}}-1)V(w_1)-(\frac{1}{w_2^2}-1)V(w_2)}{(w_1-w_2)^3(1-\frac{1}{w_1w_2})} \d w_1 \d w_2 \\
&= \Re \frac{2}{(2\pi\i)^2} \int_{\eta} \int_{\eta_r} \frac{w_1(\frac{1}{w_2^2}-1)V(w_2)}{(w_1-w_2)^3(w_1-\frac{1}{w_2})} \d w_1 \d w_2.
\end{align*}
In the last integral we can compute the $w_1$-integral using Cauchy's integral formula to see that
\[ \frac{1}{2\pi\i} \int_{\eta} \frac{w_1}{(w_1-w_2)^3(w_1-\frac{1}{w_2})} \d w_1 = -\frac{1}{2} \frac{\d^2}{\d w_1^2} \frac{w_1w_2}{1-w_1w_2} \Big|_{w_1=w_2} = - \frac{w_2^2}{(w_2^2-1)^3}, \]
since $\frac{1}{w^2}$ lies outside $\eta$. Hence, 
\begin{align*}
I_2 &= 2 \Re \frac{1}{2\pi\i } \int_{\eta_r} \frac{V(w)}{(w^2-1)^2} \d w = 2 \Re \frac{1}{2\pi\i } \int_{\eta_r'} \frac{V(1/w)}{(\frac{1}{w^2}-1)^2} \frac{\d w}{w^2} \\
&= -2 \Re \frac{1}{2\pi\i } \int_{\eta_r'} \frac{V(w)}{(w-1)^2(w+1)^2} \d w = -I_2+\frac{1}{2}\Re(V'(1)+V'(-1)),
\end{align*}
where $\eta_r'$ is the positively oriented image of $\eta_r$ under the map $w\mapsto 1/w$. In the last line we used that $V(1/w)/w^2 = -V(w)$ by \eqref{eq:Vformula}, and then used Cauchy's integral formula. Thus we see that
\[I_2 = \frac{1}{4} \Re (V'(1)+V'(-1)) \]
which gives \eqref{eq:V'}. To prove \eqref{eq:V'2} we write
\[V(w) = \frac{\partial_\alpha h(w) h(\frac{1}{w})-h(w)\partial_\alpha h(\frac{1}{w})}{h'(w)h(\frac{1}{w})+h(w)h'(\frac{1}{w})\frac{1}{w^2}} = \frac{p_1(w)}{p_2(w)} \]
and note that $p_1(\pm 1) = 0$ (and hence $V(\pm 1) =0$ since $p_2(\pm 1) \neq 0$). From this we see that
\[ V'(\pm 1) = \frac{p_1'(\pm 1)}{p_2(\pm 1)} = \frac{\partial_\alpha h'(\pm 1)}{h'(\pm 1)}-\frac{\partial_\alpha h(\pm 1)}{h(\pm 1)}. \]
To prove \eqref{eq:intS}, note that
\begin{align*}
&\int_{\eta} V(w) Sh(w) \d w = \int_{\eta_-}V(w) Sh(w) \d w-\int_{\eta_+}V(w) Sh(w) \d w \\
&= \int_{\gamma}V(q_-(\zeta)) q_-'(\zeta) Sh(q_-(\zeta)) \d\zeta-\int_{\gamma}V(q_+(\zeta)) q_+'(\zeta) Sh(q_+(\zeta)) \d\zeta \\
&= -\int_{\gamma} \partial_\alpha z_e(\tau_e(\zeta)) ( S\psi_-(\zeta)-Sq_-(\zeta)) \d\zeta  + \int_{\gamma} \partial_\alpha z_e(\tau_e(\zeta)) ( S\psi_+(\zeta)-Sq_+(\zeta)) \d\zeta,
\end{align*}
where we used Lemma \ref{lemma2section9} and the chain rule for the Schwarzian derivative. A computation shows that $Sq_+-Sq_- = 0$, which proves \eqref{eq:intS}.
\end{proof}
To prove Proposition~\ref{prop:variation-of-Grunsky} we will now rewrite the integral appearing in Proposition~\ref{prop:var}:
\[
\Re \frac{1}{6\pi \i} \int_{\eta_\alpha} V(w)(Sh)(w) \d w, \quad V(w) = \partial_\alpha \lambda(w)/\lambda'(w).
\]
The main step is to express the vector field $V(w)$ in terms of the interior/exterior conformal maps $ f$ and $g$. For this we will use the following lemma.  
\begin{lemma}\label{lem:normal-variation}
    Consider a family of Jordan curves or Jordan arcs $(\gamma_\alpha)$ as in Section~\ref{sect:interpolation-of-arcs} for real $\alpha$ in a neighborhood of $\alpha_0$.
 Suppose we are given two $C^2-$diffeomorphic local parametrizations by (a subinterval of) the unit circle, $f_\alpha(e^{it}), t \in (t_0 -\delta, t_0+\delta)$ and $g_\alpha(e^{is}), s \in (s_0 -\delta, s_0+\delta)$ so that $f_\alpha(e^{it_0}) = g_\alpha(e^{is_0})$. Suppose further that for each such $t$,  $\alpha \mapsto f_\alpha, f^{-1}_\alpha, g_\alpha, g_\alpha^{-1}$ are smooth in a neighborhood of $\alpha_0$. Write $w_0 = f_{\alpha_0}(e^{it_0}) = g_{\alpha_0}(e^{is_0})$. Then
\[
\Re \left(\overline{\nu (w_0)}\partial_\alpha f_{\alpha_0}(e^{it_0}) \right)=\Re \left(\overline{\nu (w_0)}\partial_\alpha g_{\alpha_0}(e^{is_0}) \right) 
\]
\end{lemma}
\begin{proof} There exists a real-valued function $h_\alpha$ with $h_\alpha(t_0) = s_0$ such that for each $\alpha$, $f_\alpha(e^{it}) = g_\alpha(e^{ih_\alpha(t)})$. By our assumptions on $f_\alpha, g_\alpha$ we see that $h_\alpha$ is $C^1$ with respect to $\alpha$. Differentiating  with respect to $\alpha$,
\[
\partial_\alpha f_\alpha(e^{it_0})-\partial_\alpha g_\alpha(e^{is_0}) = g'_\alpha(e^{is_0})ie^{is_0}\partial_\alpha h_\alpha(t_0).
\]
On the other hand $\nu(w_0) = e^{is_0}g'_\alpha(e^{is_0})/|g'_\alpha(e^{is_0})|$, which implies that
\[
\Re \overline{\nu(w_0)}\left(\partial_\alpha f_\alpha(e^{it_0} )-\partial_\alpha g_\alpha(e^{is_0})\right) = 0.
\]
\end{proof}

We are now ready to prove Proposition~\ref{prop:variation-of-Grunsky}.
\begin{proof}[Proof of Proposition~\ref{prop:variation-of-Grunsky}] Write $h_i=f^{-1}, h_e=h=g^{-1}$ and let 
\[
I := \Re \frac{1}{i} \int_{\eta} V(w)(Sh_e(w) - S h_i(w)) dw.
\]
We claim that 
\[
\frac{\d}{\d \alpha} \log \det(I+B_\alpha) = \frac{1}{12\pi}I + \frac{1}{4}\Re (V'(1) + V'(-1)).
\]
Indeed, this follows easily from \eqref{eq:V'} using the fact that $h_i=j \circ h_e \circ j$ combined with the transformation rule for the Schwarzian. Let $\nu$ be the unit (complex) normal vector in the outward-pointing direction. Then we may rewrite $I$ as
\[
I = \Re \int_{\eta} \overline{\nu(w)}V(w) \nu(w)^2(Sh_e(w) - S h_i(w))|dw|.
\]
By Lemma~\ref{lem:curvature},   
\[
\Im \nu(w)^2(Sh_e(w) - S h_i(w)) \equiv 0
\] on $\eta$, so 
\[
I = \int_{\eta} \Re (\overline{\nu(w)} V(w) )\Re \left( \nu(w)^2(Sh_e(w) - S h_i(w)) \right)|dw|.
\]
Next, let $\kappa(w)$ denote the curvature of $\eta$ at the point $w$ as in \eqref{lem:curvature}. If we define
\[
J_e := \int_{\eta}  \Re (\overline{\nu(w)} V(w) )\Re \left( \nu(w)^2(Sh_e - \kappa(w)^2) \right)|dw|.
\]
and $J_i$ in the same way, by replacing $h_e$ by $h_i$, we can write
\[
I=J_e - J_i.
\]
Recall that $V(w) = \partial_\alpha \lambda_\alpha(w)/\lambda_\alpha'(w)$, where $ \lambda_\alpha(w)=  h(w)/h(1/w)$ which maps $\eta = \eta_\alpha$ onto $\mathbb{T}$. We want to relate $\Re (\overline{\nu(w)}V(w))$ appearing in the definition of $J_e$ and $J_i$ to similar expressions involving the conformal maps $f$ and $g$. To do this we show that $\lambda_\alpha(w)$ locally has a holomorphic inverse which (locally) provides an analytic parametrization of $\eta_\alpha$ which is smooth in $\alpha$ and to which we can apply Lemma~\ref{lem:normal-variation}.

Fix $\alpha_0$ and consider $w_0 \in  \eta_{\alpha_0}$. 
We first claim that $\lambda_{\alpha_0}$ is holomorphic and $\lambda'_{\alpha_0}(w) \neq 0$ in a neighborhood $N$ of $w_0$. The holomorphicity is clear from the definition since $h_{\alpha}$ is a conformal map defined in a neighborhood of $\eta_{\alpha_0}$. 
Then since the equilibrium measure has a positive density everywhere on $\gamma$,
\[
\frac{\d}{\d t} \arg \lambda_{\alpha_0}(w(t)) = \Im \frac{\lambda_{\alpha_0}'(w(t))w'(t)}{\lambda_{\alpha_0}(w(t))}> 0. 
\]
It follows that $\lambda_{\alpha_0}'(w) \neq 0$ for all $w \in \eta_{\alpha_0}$, so the claim follows. By the argument principle, we may assume that $\lambda_{\alpha_0}$ has an inverse $\mu_{\alpha_0}(z)$ in $N$. Moreover, there exists $\delta >0$ such that $B(w_0, 2\delta) \subset N$ and
\begin{equation}\label{jan20}
\mu_{\alpha_0}(z) = \frac{1}{2\pi i}\int_{|w-w_0| = \delta}\frac{w\lambda'_{\alpha_0}(w)}{\lambda_{\alpha_0}(w)-z} dw
\end{equation}
from which we see that $\mu_{\alpha_0}$ is holomorphic in a neighborhood of $e^{it_0}=\lambda_{\alpha_0}(w_0) \in \mathbb{T}$.
Indeed, the integrand has a simple pole at $w=w(z)$ such that $\lambda_{\alpha_0}(w)=z$ and the residue equals $w(z)$. It follows that if $\epsilon >0$ is small enough, then $\mu_{\alpha_0}(e^{it}), t\in (t_0-\epsilon, t_0 +\epsilon),$ provides an analytic parametrization of a subarc of $\eta_{\alpha_0}$ containing $w_0$. Using the smoothness in $\alpha$, taking $\alpha$ in a small enough neighborhood of $\alpha_0$, we similarly obtain analytic parametrizations $\mu_\alpha(e^{it}), t\in (t_0-\epsilon, t_0 +\epsilon),$ of subarcs of $\eta_\alpha$ by replacing $\alpha_0$ by $\alpha$ in \eqref{jan20}. It also follows from \eqref{jan20} that $\alpha \mapsto \mu_\alpha(e^{it})$ is smooth in a neighborhood of $\alpha$. 

We now apply Lemma~\ref{lem:normal-variation} which shows that
\begin{align}\label{jan21.1}
\Re \left(\overline{\nu (w)}(\partial_\alpha \mu_{\alpha})(\lambda_\alpha(w)) \right) = \Re \left(\overline{\nu (w)}(\partial_\alpha f)(h_i(w)) \right) = \Re \left(\overline{\nu (w)}(\partial_\alpha g)(h_e(w)) \right).
\end{align}
On the other hand, differentiating $\mu \circ \lambda(w) = w$ gives
\[
 (\partial_\alpha \mu_\alpha)(\lambda_\alpha(w)) = -\frac{\partial_\alpha \lambda_\alpha(w)}{\lambda_\alpha'(w)}.
\]
Plugging this into \eqref{jan21.1} implies
\[
\Re\left( \overline{\nu(w)} \frac{\partial_\alpha \lambda(w)}{\lambda'(w)}\right) = - \Re \left(\overline{\nu (w)}(\partial_\alpha f)(h_i(w)) \right) = -\Re \left(\overline{\nu (w)}(\partial_\alpha g)(h_e(w)) \right).
\]
Combining this with the definition of $J_e, J_i$ and a change of coordinates (using the transformation rule for the Schwarzian), we get
\[
J_e = \int_{\mathbb{T}}\Re\left( \frac{\partial_\alpha  g(z)}{z  g'(z)}\right)\left(\Re(z^2 S   g(z)) + |  g'(z)|^2\kappa(  g(z))^2\right)|dz|
\]
and
\[
J_i= \int_{\mathbb{T}}\Re\left( \frac{\partial_\alpha   f(z)}{z  f'(z)}\right)\left(\Re(z^2 S f(z)) + | f'(z)|^2\kappa(  f(z))^2\right)|dz|.
\]
Recalling the definitions of $\rho_i, \rho_e$, this proves the proposition.
\end{proof}

\section{Proof of Theorem~\ref{thm:main-all-beta} and \ref{thm:main-linear-statistics}}\label{sec:mainthm}

Our main goal in this section will be to prove Theorem \ref{thm:main} below, from which Theorem~\ref{thm:main-all-beta} and \ref{thm:main-linear-statistics} are derived. The proof will be completed only at the end of the section.

\subsection{Set-up and statements}
Let $g:\gamma \rightarrow \C$ be a given function, which will serve as the test function in the linear statistic (note that in this section we are \emph{not} denoting the exterior conformal map by $g$). Recall that $z_e(t)$, $t\in[-1,1]$, is the equilibrium parametrization of $\gamma$. Let
\begin{equation}\label{gzt}
g_k = \frac{2}{\pi} \int_{-1}^1 g(z_e(t))T_k(t) \frac{\d t}{\sqrt{1-t^2}}, \quad k\geq 0,
\end{equation}
\begin{equation}\label{logzprime}
-d_k = \frac{2}{\pi} \int_{-1}^1 \log |z_e'(t)|T_k(t) \frac{\d t}{\sqrt{1-t^2}}, \quad k\geq 0
\end{equation}
be the Chebyshev coefficients of $g\circ z_e$ and $-\log |z_e'|$. Introduce the infinite column vectors
\begin{equation}\label{gvec}
\g = (\sqrt{k} g_k)_{k\geq1},
\end{equation}
\begin{equation}\label{dvec}
\dd = (\sqrt{k} d_k)_{k\geq 1},
\end{equation}
and
\begin{equation}\label{fvec}
\f = \big(\frac{1+(-1)^k}{\sqrt{k}} \big)_{k\geq 1}.
\end{equation} 
Observe that if $g$ and $\gamma$ are sufficiently regular, then $\g, \dd \in \ell^2(\Z_+)$, but $\f \notin \ell^2(\Z_+)$. Finally, write
\begin{equation}\label{gbetavec}
\g_\beta = \g + \big(\frac{\beta}{2}-1 \big)\dd.
\end{equation}

\begin{thm}\label{thm:main}
Let $\gamma$ be a $C^{21+\alpha}$ simple arc and $g$ a complex-valued $C^{4+\alpha}$ function on $\gamma$. Then, as $n\to\infty$,
\begin{multline}\label{Dnasy}
\bar{Z}_{n}^{\beta}(\gamma)[e^g] = \bar{Z}_{n}^{\beta}([-1,1])\det(I+B)^{-1/2}  \exp \Big( n g_0 + \frac{1}{4\beta}\g_\beta^t(I+B)^{-1}\g_\beta \\ + \frac{1}{2\beta} \big(\frac{\beta}{2}-1\big)\f^t(I+B)^{-1}\g_\beta 
-\frac{1}{4\beta} \big(\frac{\beta}{2}-1\big)^2 \f^t (I+B)^{-1} B\f +o(1) \Big).
\end{multline}

\end{thm}

Thus we see that as $n\to \infty$,
\begin{align*}
  &\log \frac{Z_{n}^\beta(\gamma)}{Z_{n}^\beta([-1,1])} =  \left(\frac{\beta}{2} n^2 \right)\log (2\mathrm{cap}(\gamma))+(\frac{\beta}{2}-1)d_0n/2 - \frac{1}{2}\log \det(I+B) \\
  &+ \frac{1}{4\beta} \big(\frac{\beta}{2}-1\big)^2 \left( \dd^t (1+B)^{-1} \dd + 2 \f^t (1+B)^{-1} \dd - \f^t(1+B)^{-1} B \f \right) + o(1).  
\end{align*}

Combined with Theorem \ref{thm:Grunsky-Loewner2-intro}, \eqref{d0ent} and Lemma \ref{lem:new_var_and_mean} below, this gives Theorem \ref{thm:main-all-beta}.

Next let $\E_{n,\gamma}^\beta[\cdot]$ denote expectation with respect to \eqref{fn}. Another consequence of Theorem \ref{thm:main} is the following limit theorem for the Laplace transform of a linear statistic (for arcs that are sufficiently smooth).

\begin{thm}\label{thm:relSz}
Let $\gamma$ be a $C^{21+\alpha}$ Jordan arc and $g$ a complex-valued $C^{4+\alpha}$ function on $\gamma$. Then, 
\begin{equation}\label{relSz}
\lim_{n\to\infty} \E_{n,\gamma}^\beta \big[ \exp \big( \sum_\mu g(z_\mu) -n \int_\gamma g(z) \d \nu_e(z)\big) \big] = \exp \big( A_\beta[g] \big)
\end{equation}
where
\begin{equation}\label{Ag}
A_\beta[g] = \frac{1}{4\beta} \g^t(I+B)^{-1}\g+\frac{1}{4}\big(1-\frac{2}{\beta}\big) (\dd+\f)^t(I+B)^{-1}\g.
\end{equation}

\end{thm}

The proof of Theorem \ref{thm:main} is based on a modified version of Theorem \ref{thm:relSz}. The proof of this modified version, which includes Theorem \ref{thm:relSz} as a special case, will be given in Section \ref{sec:proof}.

As a consequence of Theorem \ref{thm:relSz}, we see that the linear statistic $\sum_\mu g(z_\mu) -n \int_\gamma g(z) \d \nu_e(z)$ converges to a Gaussian random variable with mean $\frac{1}{4}\big(1-\frac{2}{\beta}\big) (\dd+\f)^t(I+B)^{-1}\g$ and variance $\frac{1}{2\beta} \g^t(I+B)^{-1}\g$. 

In the case of real test functions, we can obtain an alternative expression for the limiting mean and variance of the linear statistic, in terms of the Dirichlet integral of the harmonic extension of the test function. This is the content of the next lemma, which combined with Theorem \ref{thm:relSz} gives Theorem \ref{thm:main-linear-statistics}.

\begin{lemma}\label{lem:new_var_and_mean}
Let $\gamma$ be a $C^{9+\alpha}$ Jordan arc with endpoints $-1$ and $1$ and let $g : \gamma \to \mathbb{R}$ be a $C^{4+\alpha}$ function. Set 
\[
d(w) = \log|\tau_e'(w)|, \quad w \in \gamma
\]
and
\[
m(w) = - \frac{1}{2} \log \left|\frac{1-w^2}{1-\tau_e(w)^2} \right|, \quad w \in \gamma.
\]
   Define 
    \begin{align}
    A[g] &= \mathcal{D}_{\C \smallsetminus \gamma}(g) \\
    M[g] &= \mathcal{D}_{\C \smallsetminus \gamma}(d-m,g) +g(1) +g(-1) - g_0 \\
    J^F(\gamma) &=  \mathcal{D}_{\C \smallsetminus \gamma}(d-m) +  3 \log|h'(1)h'(-1)| - 2(d_0 + \log(2\textrm{cap}(\gamma)).
    \end{align}
    Then,
    \begin{align}
    A[g] &= \g^t (1+B)^{-1} \g \\
    M[g] &= \dd^t (1+B)^{-1} \g + \f^t (1+B)^{-1} \g \\
    J^F(\gamma) &= \dd^t (1+B)^{-1} \dd + 2 \f^t (1+B)^{-1} \dd - \f^t(1+B)^{-1} B \f,
    \end{align}
     where $\g, \dd,$ and $ \f$ are as in \eqref{gvec}, \eqref{dvec}, and \eqref{fvec}. Finally,
     \[
     |z_e'(\pm 1)| =  |h'(\pm 1)|^{-2}.
     \]
\end{lemma}
\begin{proof}
The identity for $A[g]$ follows directly from Proposition~\ref{prop:Dir}.
Consider the expression for $M[g]$ on the Chebyshev side. Let $w\in \gamma$. We can write $d(w) = -\log|z'_e(\tau_e(w))|$. Then \[d(z_e(t)) = -\log|z'_e(t)| = \frac{d_0}{2} + \sum_{k=1}^\infty d_k T_k(t).\] Using this, we can argue similarly as in Proposition~\ref{prop:Dir} to see that
\begin{align}\label{M1}
\dd^t (1+B)^{-1} \g = \mathcal{D}_{\mathbb{C} \smallsetminus \gamma }(d, g)
\end{align}
and
\begin{align}\label{M2}
\dd^t (1+B)^{-1} \dd = \mathcal{D}_{\mathbb{C} \smallsetminus \gamma }(d).
\end{align}
We cannot immediately apply the same reasoning to the terms involving $\f$ since this is not an element of $\ell^2$. However, we have
\begin{align}\label{M4}
\f^t(I+B)^{-1}\g = \f^t \g + \f^t((I+B)^{-1}-I) \g = \f^t \g - \f^t B (I+B)^{-1} \g.
\end{align}
Since $T_k(1)=1, T_k(-1)=(-1)^k$, we see that
\begin{align} \nonumber
\f^t \g & =  \sum_{k=1}^\infty(1+(-1)^k)g_k \\\nonumber
& = \frac{g_0}{2} + \sum_{k =1}^\infty T_k(1) g_k + \frac{g_0}{2} + \sum_{k =1}^\infty T_k(-1) g_k - g_0\\ 
& = g(1) + g(-1) - g_0. \label{M33}
\end{align}
We next claim that we can write 
\begin{align}\label{M5}
     B \f = {\bf m},
\end{align} 
where ${\bf m} = (\sqrt{k} m_k)_{ \ge 1}$ are the rescaled Chebyshev coefficients of the function
\[
m(z_e(t)) = \frac{m_0}{2} +\sum_{k = 1}^\infty m_kT_k(t).
\]
Indeed, since $z_e(\pm 1) = \pm 1$, and using the arc-Grunsky expansion we find
\begin{align*}
m(z_e(t)) &= - \frac{1}{2} \log \left|\frac{z_e(t) - z_e(1)}{t-1} \right|- \frac{1}{2} \log \left|\frac{z_e(t) - z_e(-1)}{t+1} \right| \\ 
& = -\log(2\textrm{cap}(\gamma)) + \sum_{k,l = 1}^\infty a_{kl}(T_k(t)T_l(1) + T_k(t)T_l(-1) ) \\
& =  -\log(2\textrm{cap}(\gamma)) + \sum_{k = 1}^\infty (\sum_{l = 1}^\infty a_{kl}(1+(-1)^l)T_k(t)  \\ & = -\log(2\textrm{cap}(\gamma))  + \sum_{k =1}^\infty( \sum_{l =1}^\infty \sqrt{l} a_{kl} f_l)T_k(t),
\end{align*}
which proves the claim. Combining this with \eqref{M33} it therefore follows from \eqref{M4} that
\begin{align}\label{M6}
\f^t(I+B)^{-1}\g = g(1) + g(-1) - g_0 - \mathcal{D}_{\mathbb{C} \smallsetminus \gamma}(m,g).
\end{align}
Similarly, we see that
\begin{align}\label{M7}
\f^t(I+B)^{-1}\dd = d(1) + d(-1) - d_0 - \mathcal{D}_{\mathbb{C} \smallsetminus \gamma}(m,d)
\end{align}
and
\begin{align}\label{M8}
\f^t(I+B)^{-1}(B \f) = \f^t(I+B)^{-1}{\bf m} =  m(1) + m(-1) - m_0 - \mathcal{D}_{\mathbb{C} \smallsetminus \gamma}(m).
\end{align}
On the other hand, from the arc-Grunsky expansion, we have $m_0 = - 2 \log(2\textrm{cap}(\gamma))$ and from the definition,
\[
m(1) + m(-1) = -\frac{1}{2}\log|z'_e(-1)z'_e(1)| = \frac{1}{2}(d(1) + d(-1)).
\]
We conclude that
\begin{align*}
M[g] & = \mathcal{D}_{\mathbb{C} \smallsetminus \gamma}(d,g) - \mathcal{D}_{\mathbb{C} \smallsetminus \gamma}(m,g) + g(1) + g(-1) - g_0 \\ 
& = \mathcal{D}_{\mathbb{C} \smallsetminus \gamma}(d-m,g)+ g(1) + g(-1) - g_0,
\end{align*}
and 
\begin{align*}
J^F(\gamma)  = &\mathcal{D}_{\mathbb{C} \smallsetminus \gamma}(d) +  2(d(1) + d(-1) - d_0)- 2 \mathcal{D}_{\mathbb{C} \smallsetminus \gamma}(m,d) + \mathcal{D}_{\mathbb{C} \smallsetminus \gamma}(m) \\ &- (m(1) + m(-1) - m_0) \\ 
 = &\mathcal{D}_{\mathbb{C} \smallsetminus \gamma}(d-m)-\frac{3}{2}\log|z'_e(-1)z'_e(1)| -2(d_0 + \log (2\textrm{cap}(\gamma)).
\end{align*}
It remains to prove that 
  \[
     |z_e'(\pm 1)| =  |h'(\pm 1)|^{-2}.
     \]
     We write $u(w) = (w+1/w)/2$. For $w \in \eta$, let $\lambda(w) = h(1/w)/h(w)$ and define
     \begin{align}\label{eq:T}
     T(w) = \tau_e(u(w))= -\cos \left(\frac{1}{2i} \log \lambda(w)\right).
     \end{align}
     Then tracing along $\eta_+$ from $-1$ to $1$, $T$ traces $[-1,1]$ from $-1$ to $1$. Note that $z'_e(T(w)) = u'(w)/T'(w)$ if $w \in \eta_+, w \neq \pm 1$. Using the formula for $u$, taking limits along $\eta$,
     \[
     \lim_{w \to 1} | z'_e(T(w))|^2 = \lim_{w \to 1} \left|\frac{w-1}{T'(w)}\right|^2.
     \]
     By \eqref{eq:T} we have
     \[
     T'(w)^2 = -\frac{\lambda'(w)^2}{4\lambda(w)^2} \frac{1 - \cos(\frac{1}{i} \log \lambda(w))}{2} = \frac{\lambda'(w)^2}{16 \lambda(w)^3}(\lambda(w) - 1)^2.
     \]
     Hence,
     \[\lim_{w \to 1} | z'_e(T(w))|^2 = \frac{16}{|\lambda'(1)|^4}.\]
     On the other hand, we may compute explicitly that $|\lambda'(1)| = 2|h'(1)|$ which gives
     \[
     |z'_e(1)| = |h'(1)|^{-2},
     \]
     as claimed.
     The derivative at $-1$ is handled in a similar manner.
     \end{proof}

\subsection{Outline of the proof}
Given the equilibrium parametrization of the arc, and the general facts that we derived about the arc-Grunsky operator in Section \ref{sec:Grunsky}, the general strategy for the proof is similar to that of \cites{Joh1, CouJoh}. 

The first step to analyze the asymptotics of $Z_{n}^{\beta}(\gamma) [e^g]$ is to introduce the equilibrium parametrization $z_e$ of the Jordan arc $\gamma$ by making the change of variables $z_\mu = z_e(\cos(\theta_\mu))$ in \eqref{Pn}. This gives a new probability measure $d\P_{n,\gamma}^{\beta}$ on $[0,2\pi]^n$ that describes the gas on the Jordan arc. We then define a family of probability measures $\P_{n,\gamma}^{\beta,s}$ on $[0,\pi]^n$, for $s\in[0,1]$, that interpolate between $\P_{n,\gamma}^{\beta} = \P_{n,\gamma}^{\beta,1}$ and $\P_{n,\gamma}^{\beta, 0}$, the measure which corresponds to the gas on $[-1,1]$. We let $Z_{n}^{\beta,s}(\gamma)$ denote the corresponding partition functions. 

To obtain the asymptotic partition function for the Jordan arc we can now compare it to the one for $[-1,1]$ by writing
\[ \log\ Z_{n}^{\beta}(\gamma)/ Z_{n}^{\beta,0}(\gamma) = \int_0^1 B_n'(s) \d s, \quad B_n(s) = \log\ Z_{n}^{\beta,s}(\gamma). \]
After some computations,
\[ B_n'(s) = -\beta \sum_{k,l \geq 1} a_{kl} \E_n^s[X_k(\theta)X_l(\theta)] -(1-\beta/2) \sum_{k\geq 1} \frac{1}{\sqrt{k}}d_k \E_n^s[X_k(\theta)] \]
where $(a_{kl})$ are the Grunsky coefficients, $(d_k)$ are the Chebyshev coefficients of $\log |z_e'|$ defined in \eqref{logzprime}, and $X_k(\theta) = \sum_{\mu=1}^n \cos k\theta_\mu$. The limit in $n$ of $Z_{n}^{\beta}(\gamma)$ will then follow from the limit of the mean and covariances of the linear statistics $X_k(\theta)$, for any $s\in[0,1]$ and $k\in \N$, together with some uniform bounds. These will follow in turn from the more general formula for the asymptotic Laplace transform of linear statistics of $\P_{n,\gamma}^{\beta,s}$ given in Theorem \ref{thm:bound}.

To study the Laplace transform we make a change of variables $\omega_\mu = \theta_\mu +\frac{1}{n}H_s(\theta_\mu)$ where $H_s$ solves the specific integral equation \eqref{Intekv1}. For $s=1$, we can write $H_1(\theta) = h(z_e(\cos(\theta))$ where $h$ is the solution to 
\[ g(z) = \frac{\beta}{\pi} \Re p.v. \int_\gamma \frac{h(\zeta)}{\zeta-z} \d\zeta.\]
After the change of variables, the linear statistic becomes
\[ \sum_{\mu} G(\omega_\mu) = \sum_\mu G(\theta_\mu) + \frac{1}{n} \sum_\mu G'(\theta_\mu)H_s(\theta_\mu) + O(\frac{1}{n})  \]
where $G(\theta) = g(z_e(\cos(\theta))$. In Section \ref{sec:domain}, Lemma \ref{lem:Ensk}, we define a set $E_{n,s,K}$ whose probability tends to one exponentially fast and on which the normalized counting measure $\frac{1}{n}\sum_\mu \delta_{\cos \theta_\mu}$ converges weakly to the equilibrium measure on $[-1,1]$. Therefore
\[ \sum_\mu G(\theta_\mu) \approx \frac{n}{\pi} \int_{-1}^1 g(z_e(t)) \frac{\d t}{\sqrt{1-t^2}} = ng_0 \]
and 
\[ \frac{1}{n} \sum_\mu G'(\theta_\mu)H_s(\theta_\mu) \approx \frac{1}{\pi} \int_{-1}^1 g'(z_e(t))h_s(z_e(t)) \frac{\d t}{\sqrt{1-t^2}} =  \frac{1}{2\beta} \mathbf{g}^t(I+sB)^{-1}\mathbf{g} \] 
which gives the asymptotic variance. Taylor expanding up to first or second order the other terms in the exponential in $\E_{n,\gamma}^{\beta,s}[e^g]$ after the change of variables gives the asymptotic mean minus half the variance, and other terms that cancel each other out. In order to show this cancellation we need to know more precisely how close $\frac{1}{n}\sum_\mu \delta_{\cos \theta_\mu}$ is to the equilibrium measure. We obtain bounds in $l_2$ on the deviations from the Fekete points for any sequence $\theta_\mu \in E_{n,s,K}$ in Lemma \ref{lem:sumtmu}. We first obtain an upper bound on the remainder term, and thus on the Laplace transform (see \eqref{expbound}). We then use this upper bound to bootstrap the argument and show that the remainder term actually goes to zero.

The fact that the integral equation has a solution $H_s$ is not obvious. But we can rewrite the integral equation by using the arc-Grunsky operator so that in terms of Chebyshev coefficients it takes the form
\begin{equation}
-\frac{1}{\beta} \g = (I+sB)\h^s
\end{equation}
where $\h^s = (\sqrt{k} h_k^s)_{k\geq1}$ and $h_k^s$ is the $k$th sine Fourier coefficient of $H_s$. In Section \ref{sec:Grunsky} we showed that if the curve is sufficiently smooth then $B\geq -\kappa I$ so that $I+sB$ is invertible for all $s\in[0,1]$. This gives a solution to the integral equation in terms of its Chebyshev coefficients.

\subsection{Interpolation of the partition function}\label{sec:interpolation-partition-function}
From now on we assume without loss of generality that $g$ is real valued, see Lemma 2.2 in \cite{CouJoh}. 
In order to analyze the asymptotics of $Z_{n}^{\beta}(\gamma) [e^g]$, we introduce the equilibrium parametrization $z_\mu = z_e(x_\mu)$, $x_\mu\in[-1,1]$ into the integral \eqref{Dn}. It follows from \eqref{Znnorm} and \eqref{d0ent} that
\begin{multline*}
\bar{Z}_{n}^{\beta}(\gamma)[e^g]\\
= \frac{1}{n!} \int_{[-1,1]^n} \exp \Big[ \frac{\beta}{2} \sum_{\mu\neq \nu} \log |x_\mu-x_\nu| + \frac{\beta}{2} \sum_{\mu,\nu} \Big(\log \big| \frac{z_e(x_\mu)-z_e(x_\nu)}{x_\mu-x_\nu} \big| -\log(2\text{cap}(\gamma)\Big)\Big]\\
 \cdot\exp\Big( \sum_\mu g(z_e(x_\mu)) + (1-\beta/2)\sum_\mu \big(\log |z_e'(x_\mu)|+d_0/2\big) \Big) \prod_\mu \d x_\mu.
\end{multline*}

We modify this expression by introducing an extra parameter $s\in [0,1]$. Let
\begin{multline}\label{Dn2}
 \bar{Z}_{n}^{\beta,s}(\gamma) [e^g] \\
 = \frac{1}{n!} \int_{[-1,1]^n} \exp \Big[ \frac{\beta}{2} \sum_{\mu\neq \nu} \log |x_\mu-x_\nu| + \frac{\beta s}{2} \sum_{\mu,\nu}\Big( \log \big| \frac{z_e(x_\mu)-z_e(x_\nu)}{x_\mu-x_\nu} \big|-\log(2\text{cap}(\gamma)\Big)\Big] \\
\cdot \exp\Big( \sum_\mu g(z_e(x_\mu)) + (1-\beta/2)s\sum_\mu \big(\log |z_e'(x_\mu)|+d_0/2\big) \Big) \prod_\mu \d x_\mu.
\end{multline}
Note that $\bar{Z}_{n}^{\beta,1}(\gamma)[e^g] = \bar{Z}_{n}^{\beta}(\gamma)[e^g]$ and $\bar{Z}_{n}^{\beta,0}(\gamma)[e^g] = \bar{Z}_{n}^{\beta}([-1,1])[e^g]$ so the expression in \eqref{Dn2} interpolates between the case of the general Jordan arc $\gamma$ and the case of the interval $[-1,1]$. We rewrite \eqref{Dn2} further by making the change of variables $x_\mu = \cos \theta_\mu$, $\theta_\mu \in [0,\pi]$. This gives
\begin{equation}\label{Dn3}
\bar{Z}_{n}^{\beta,s}(\gamma)[e^g] = \int_{H_n} \exp \Big( \frac{\beta}{2} F_n^s(\theta) + \sum_\mu (G(\theta_\mu)+ (1-\beta/2) s L(\theta_\mu)+\log (\sin(\theta_\mu)) ) \Big) \prod_\mu \d\theta_\mu
\end{equation}
where $H_n = \{ \theta \in [0,\pi]^n: 0\leq \theta_1<\theta_2<\cdots <\theta_n \leq \pi \}$, and
\begin{equation}\label{Fns}
F_n^s(\theta) = \sum_{\mu\neq \nu} \log |\cos \theta_\mu-\cos \theta_\nu| + s \sum_{\mu,\nu}\Big( \log \big| \frac{z_e(\cos \theta_\mu)-z_e(\cos \theta_\nu)}{\cos \theta_\mu-\cos \theta_\nu} \big|-\log(2\text{cap}(\gamma)\Big),
\end{equation}
\begin{equation}\label{L}
L(\theta) = \log | z_e'(\cos \theta) |+d_0/2,
\end{equation}
and
\begin{equation}\label{G}
G(\theta) = g(z_e(\cos \theta)).
\end{equation}
Note that $\frac 1{\pi}\int_0^\pi L(\theta)\d\theta=0$, and we can assume that $\frac 1{\pi}\int_0^\pi G(\theta)\d\theta=0$.

In \eqref{Dn3} we will make the change of variables $\theta_\mu\rightarrow \theta_\mu+\frac{1}{n}H_s(\theta_\mu)$ with an appropriate function $H_s(\theta)$. It turns out that the correct choice is to let $H_s$ be the solution to the integral equation discussed in Section~\ref{sec:prel}:
\begin{equation}\label{Intekv1}
G(\omega) = \frac{\beta s}{\pi} \Re  p.v.\int_0^\pi \frac{H_s(\theta)z_e'(\cos \theta)\sin \theta}{z_e(\cos\theta)-z_e(\cos \omega)} \d\theta + p.v.\frac{\beta(1-s)}{\pi} \int_0^\pi \frac{H_s(\theta)\sin\theta}{\cos\theta-\cos\omega}\d\theta
\end{equation}
 on $[0,\pi]$. 

Recall \eqref{Intekv4:prel}, 
\begin{equation}\label{Intekv4}
G(\omega) = -\beta s \sum_{k,l\geq 1} ka_{kl}h_k^s\cos(l\omega) + \beta\widetilde{H}_s(\omega),
\end{equation}
where $\widetilde{H}_s$ is the conjugate function.
Next, we recall the infinite column vectors
\begin{equation}\label{xtheta}
\x_\theta = \big( \frac{1}{\sqrt{k}}\cos(k\theta) \big)_{k\geq 1},\quad \y_\theta = \big( \frac{1}{\sqrt{k}}\sin(k\theta) \big)_{k\geq 1}.
\end{equation}
Then
\[ G(\theta) = \x_\theta^t\g, \quad \widetilde{H}_s(\theta) = -\x_\theta^t\h^s,\]
where $\h^s = (\sqrt{k}h_k^s)_{k\geq 1}$. We see that \eqref{Intekv4} can be written 
\[ \x_\theta^t\g = -\beta s \x_\theta^t B \h^s -\beta \x_\theta^t\h^s\]
where $B$ is the arc-Grunsky operator and $\g$ is defined in \eqref{gvec}. Thus, in terms of Fourier coefficients, the integral equation takes the form
\begin{equation}\label{Intekv5}
-\frac{1}{\beta} \g = (I+sB)\h^s.
\end{equation}
This equation can now be solved for $\h^s$ since by the strengthened Grunsky inequality \eqref{stGrineq}, $I+sB$ is invertible for $s\in[0,1]$. This in turn gives us the solution $H_s$ to the integral equation \eqref{Intekv1}. The next lemma shows that if $G$ is $C^{4+\alpha}$, then so is $H_s$. We define the $(4,\alpha)$-norm of functions $F$ on $[0,\pi]$ by
\[ \|F\|_{4,\alpha} = \sup_{0\leq \theta_1,\theta_2 \leq \pi} \frac{|F^{(4)}(\theta_1)-F^{(4)}(\theta_2)|}{|\theta_1-\theta_2|^\alpha}+\sum_{j=0}^4 \big\| F^{(j)} \big\|_\infty. \]

\begin{lemma}\label{lem:Hreg}
Define $G(\theta)$ by \eqref{G}, $g_k$ by \eqref{gzt}, and $\g$ by \eqref{gvec}. Let
\begin{equation}\label{hsol}
\h^s = -\frac{1}{\beta} (I+sB)^{-1}\g,
\end{equation}
and 
\begin{equation}\label{Hsol}
H_s(\theta) = \y_\theta^t\h^s.
\end{equation}
Then $\h^s$ and $H_s$ satisfy \eqref{Intekv4}. Moreover, if $g$ is $C^{l+\alpha}$ on the $C^{2l+10+\alpha}$ curve $\gamma$, $l\geq 1$, $\alpha>0$, then $H_s$ is $C^{l+\alpha}$ on $[0,\pi]$ and there is a constant $C$ such that
\begin{equation}\label{Hreg}
\| H_s\|_{4,\alpha} \leq C (A(1-\kappa)^{-1}+1)\|G\|_{4,\alpha},
\end{equation}
where $\kappa<1$ is the constant in \eqref{stGrineq} and $A$ is the constant in \eqref{Grunskydecay}. 

\end{lemma}

\begin{proof}
The first statement follows from the computations above.
From \eqref{Grunskydecay} we have the following estimate for a $C^{m+\alpha}$ arc,
\begin{equation}\label{bklest}
|b_{kl}| \leq \frac{A}{k^{p-1/2+\epsilon}l^{q-1/2+\epsilon}}
\end{equation}
if $p+q<m/2$, $\epsilon=\alpha/4$ and $k,l\geq 1$. It follows from \eqref{stGrineq} and \eqref{hsol} that 
\begin{equation}\label{hgest}
\| \h^s\|_2 \leq \frac{1}{\beta} (1-\kappa)^{-1} \| \g \|_2
\end{equation}
for $0\leq s\leq 1$. On the other hand,
\[ \h^s = -\frac{1}{\beta}\g -sB\h^s \]
so by \eqref{Hsol},
\[ H_s(\theta) = -\frac{1}{\beta} \widetilde{G}(\theta) -s\y_\theta^t B\h^s. \]
By Privalov's theorem there is a constant $C$ such that the conjugate function $\| \widetilde{G} \|_{l,\alpha} \leq C \| G \|_{l,\alpha}$, and from \eqref{hgest} it follows that
\[ \sqrt{k} |h_k^s | \leq \| \h^s \|_2 \leq \frac{1}{\beta} (1-\kappa)^{-1} \|\g \|_2  \leq C(1-\kappa)^{-1} \| G \|_{l,\alpha}, \]
since $|g_k| \leq \| G\|_{l,\alpha} k^{-l-\alpha}$, $k\geq 1$. We can now use \eqref{bklest} to get the estimate
\[ |(B\h^s)_j | = \big| \sum_{k\geq 1} b_{jk} \sqrt{k}h_k^s \big| \leq C(1-\kappa)^{-1} \| G \|_{l,\alpha} \frac{A}{j^{p-1/2+\epsilon}} \sum_{k\geq 1} \frac{1}{k^{q-1/2+\epsilon}}. \]
Take $q=2$, $p=l+2$, i.e. $m/2 \geq l+5$. It follows that $\y_\theta^t B\h^s$ is a $C^{l+\alpha}$ function and 
\[ \|\y_\theta^tBh^s \|_{l,\alpha} \leq CA(1-\kappa)^{-1} \| G \|_{l,\alpha},  \]
which proves \eqref{Hreg}. 

\end{proof}

We will now formulate a theorem that we will use to prove Theorem~\ref{thm:main}. Recall the formula \eqref{Dn3} for $Z_{n}^{\beta,s}(\gamma)[e^g]$. Let $\E_{n,\gamma}^{\beta,s}$ denote expectation with respect to the probability density proportional to 
\[ \exp \Big( \frac{\beta}{2}F_n^s(\theta) + \sum_\mu \big( (1-\beta/2)sL(\theta_\mu)+\log(\sin\theta_\mu) \big) \Big) \]
on $H_n = \{ \theta\in [0,\pi]^n: 0\leq \theta_1 < \cdots < \theta_n \leq \pi \}$. In analogy with \eqref{Ag} we define
\begin{equation}\label{Asg}
A_s[g] = \frac{1}{4\beta} \g^t (I+sB)^{-1}\g +\frac{1}{4} \big(1-\frac{2}{\beta}\big) (s\dd+\f)^t(I+sB)^{-1}\g
\end{equation}
so that $A[g] = A_1[g]$.
In the following $C$ will denote a constant that is independent of the Jordan arc $\gamma$ and the function $g$ defined on $\gamma$, but it can depend on $\beta$. A constant that depends on $\gamma$, $\beta$ and the norm $\|G\|_{4,\alpha}$ of $G$ will be denoted by $C(\gamma,G)$. The precise value of these constants can vary.

We can now state the theorem that we will need.

\begin{thm}\label{thm:bound}
Assume that $\gamma$ is a $C^{21+\alpha}$ Jordan arc, $g\in C^{4+\alpha}$ for some $\alpha>0$, and
$\int_\gamma g\d\nu_e =0$, where $\nu_e$ is the equilibrium measure on $\gamma$. There is a constant $C(\gamma,G)$ such that
\begin{equation}\label{expbound}
\E_{n,\gamma}^{\beta,s}[e^{\sum_\mu G(\theta_\mu)}] \leq C(\gamma,G)
\end{equation}
for $n\geq 1$, $s\in[0,1]$, where $G(\theta) = g(z_e(\cos\theta))$. Also, for each $s\in[0,1]$,
\begin{equation}\label{relSz2}
\lim_{n\to\infty} \E_{n,\gamma}^{\beta,s}[e^{\sum_\mu G(\theta_\mu)}] = e^{A_s[g]}
\end{equation}

\end{thm} 

Note that this theorem directly gives Theorem~\ref{thm:relSz}. It will be proved in Section \ref{sec:proof} based on preliminary estimates developed in Section \ref{sec:domain}. We can assume that $g_0 = 0$ in \eqref{Dnasy}. Then, by \eqref{Dn3},
\begin{equation}\label{Dnfor}
\frac{\bar{Z}_{n}^{\beta}(\gamma)[e^g]}{\bar{Z}_{n}^{\beta}([-1,1])} = \frac{\bar{Z}_{n}^{\beta}(\gamma)}{\bar{Z}_{n}^{\beta}([-1,1])} \E_{n,\gamma}^{\beta,1} [e^{\sum_\mu G(\theta_\mu)}]
\end{equation}
We will prove that
\begin{align}\label{partquot}
&\lim_{n\to\infty} \log \frac{\bar{Z}_{n}^{\beta}(\gamma)}{\bar{Z}_{n}^{\beta}([-1,1])} \notag\\
&= -\frac{1}{2} \log \det(I+B) + \frac{1}{4\beta}\big(\frac{\beta}{2}-1\big)^2\Big(\dd^t(I+B)^{-1}\dd + 2\f(I+B)^{-1}\dd - \f^t(I+B)^{-1}B\f \Big)
\end{align}
provided $\gamma$ is sufficiently smooth. Combining this with \eqref{relSz2} with $s=1$ proves \eqref{Dnasy}. 

It follows from \eqref{Grunskydecay} that if $\gamma\in C^{m+\alpha}$, with $m\geq 21$, then
\[ \Big| \sum_{\max(k,l)>n} a_{kl} \big(\sum_\mu \cos(k\theta_\mu) \big) \big(\sum_\nu \cos(l\theta_\nu) \big) \Big| \leq 2n^2 \sum_{k>n} \sum_{l\geq1} |a_{kl}| \leq C(\gamma)n^{-\alpha/2} \]
for all $\theta$. Hence, using \eqref{Grunskyexp2:prel}, we see that, without any consequences for our results, we can replace $F_n^s(\theta)$ in \eqref{Fns} with
\begin{equation}\label{Fns2}
F_n^s(\theta) = \sum_{\mu\neq\nu} \log |\cos\theta_\mu-\cos\theta_\nu|-2s \sum_{k,l=1}^n a_{kl} \big( \sum_\mu \cos(k\theta_\mu) \big) \big( \sum_\nu \cos(l\theta_\nu) \big)
\end{equation}
Let \[ \bar{Z}_{n}^{\beta,s}(\gamma) =  \int_{H_n} \exp \Big( \frac{\beta}{2} F_n^s(\theta) +\sum_\mu (1-\beta/2)sL(\theta_\mu)+\log(\sin\theta_\mu) \Big) d\theta \]
so that $\bar{Z}_{n}^{\beta,1}(\gamma) = \bar{Z}_{n}^{\beta}(\gamma)$ and $\bar{Z}_{n}^{\beta,0}(\gamma) = \bar{Z}_{n}^{\beta}([-1,1])$. 
For simplicity we have kept the same notation although we have modified the model.
Define the function

\begin{equation}\label{Bns}
B_n(s) = \log\ (\bar{Z}_{n}^{\beta,s}(\gamma)/ \bar{Z}_{n}^{\beta,0}(\gamma)), \quad 0\leq s\leq 1.
\end{equation}
To simplify notation we will often drop the dependence on $\gamma$ and $\beta$, e.g.\ write $\E_n^s$ instead of $\E_{n,\gamma}^{\beta,s}$. We turn now to the proof of \eqref{partquot}. By \eqref{Bns}, \eqref{L}, and \eqref{Fns2},
\begin{align*}
B_n'(s) &= \frac{1}{Z_{n}^{\beta,s}(\gamma)} \int_{H_n} \Big( \frac{\beta}{2}\frac{\d}{\d s}F_n^s(\theta) +(1-\beta/2)\sum_\mu L(\theta_\mu) \Big) \\
& \qquad \qquad \cdot \exp \Big( \frac{\beta}{2} F_n^s(\theta) +\sum_\mu (1-\beta/2)sL(\theta_\mu)+\log(\sin\theta_\mu) \Big) d\theta\\
&= \E_n^s \Big[ -\beta \sum_{k,l=1}^n a_{kl} \big( \sum_\mu \cos(k\theta_\mu) \big) \big( \sum_\nu \cos(l\theta_\nu) \big) - \big(1-\frac{\beta}{2} \big)\sum_{k\geq 1} d_k \big( \sum_\mu \cos(k\theta_\mu) \big) \Big].
\end{align*}
Write 
\[ X_k(\theta) = \sum_\mu \cos(k\theta_\mu).\]
Then, by Fubini's theorem,
\begin{equation}\label{Bnprime}
B_n'(s) = -\beta \sum_{k,l= 1}^n a_{kl} \E_n^s[X_k(\theta)X_l(\theta)] -(1-\beta/2) \sum_{k\geq 1} d_k \E_n^s[X_k(\theta)].
\end{equation}
Fix $s\in[0,1]$ and pick $J=J(n,s)$ such that
\begin{equation}\label{supj}
\sup_{j\geq 1} \E_n^s\big[ \frac{1}{j^{8+2\epsilon}} X_j(\theta)^2 \big] = \E_n^s \big[ \frac{1}{J^{8+2\epsilon}}X_J(\theta)^2 \big]
\end{equation}
for some $\epsilon>0$. Such a $J$ exists since the expectation on the left side goes to $0$ as $j\to\infty$ because $|X_j(\theta)|\leq n$, for all $j\geq 1$. Define
\[ G_J(\theta) = \frac{\cos(J\theta)}{J^{4+\epsilon}}. \]
Then $G_J\in C^{4+\epsilon}([-\pi,\pi])$ and $\| G_J \|_{4,\epsilon} \leq C$ for some constant $C$, which does not depend on $J$ (and hence on $n$ or $s$). It follows from Theorem \ref{thm:bound} that there is a constant $C(\gamma)$ such that
\[ \E_n^s[e^{\pm \sum_\mu G_J(\theta_\mu)}] \leq C(\gamma).  \]
If we let
\[ Q_n(z) = \E_n^s\big[ e^{z\sum_\mu G_J(\theta_\mu)} \big], \]
then, for $|z|\leq 1$,
\[ |Q_n(z)| \leq \E_n^s\big[e^{\Re z \sum_\mu G_J(\theta_\mu)} \big] \leq C(\gamma). \]
Consequently, by Cauchy's integral formula, $|Q_n''(0)| \leq C(\gamma)$, which gives
\[ \E_n^s\big[ \frac{1}{J^{8+2\epsilon}}X_J(\theta)^2 \big] = \E_n^s\big[ \big(\sum_\mu G_J(\theta_\mu)\big)^2\big] = Q_n''(0) \leq C(\gamma), \]
and we have proved
\begin{equation}\label{supjbound}
\sup_{j\geq 1} \E_n^s\big[ \frac{1}{j^{8+2\epsilon}}X_j(\theta)^2 \big] \leq C(\gamma).
\end{equation}
The Cauchy-Schwarz inequality gives
\begin{equation}\label{XkXlbound}
\big| \E_n^s[X_k(\theta)X_l(\theta)] \big| = (kl)^{4+\epsilon} \big| \E_n^s \big[ \frac{X_k(\theta)}{k^{4+\epsilon}}\frac{X_l(\theta)}{l^{4+\epsilon}} \big] \big| \leq C(\gamma) (kl)^{4+\epsilon},
\end{equation}
and
\begin{equation}\label{Xkbound}
\big| \E_n^s[X_k(\theta)] \big| = k^{4+\epsilon} \big| \E_n^s \big[ \frac{X_k(\theta)}{k^{4+\epsilon}} \big] \big| \leq C(\gamma) k^{4+\epsilon}.
\end{equation}
It follows from \eqref{Grunskydecay} that for $p+q\leq 10$,
\[ |a_{kl}| \leq \frac{A}{k^{p+2\epsilon}l^{q+2\epsilon}} \]
if $\epsilon = \alpha/8$. This estimate shows that
\begin{equation}\label{aklest}
\sum_{k,l\geq 1} (kl)^{4+\epsilon} |a_{kl}| < \infty, \quad \sum_{k\geq 1} k^{4+\epsilon} |d_k| < \infty.
\end{equation}
It follows from the estimates \eqref{XkXlbound}, \eqref{Xkbound}, \eqref{aklest} that we can apply the dominated convergence theorem to the sums in \eqref{Bnprime} to see that
\begin{equation}\label{Bnprimelimit}
\lim_{n\to\infty} B_n'(s) = -\beta \sum_{k,l\geq 1} a_{kl}\lim_{n\to\infty} \E_n^s[X_k(\theta)X_l(\theta)] -(1-\beta/2) \sum_{k\geq 1} d_k \lim_{n\to\infty} \E_n^s [X_k(\theta)]
\end{equation}
Let $\mathbf{e}_k$ be the infinite column vector with $1$ at position $k$ and $0$ otherwise and let
\[ \g_\lambda = \lambda_1\mathbf{e}_k+\lambda_2\mathbf{e}_l \]
where $\lambda_1,\lambda_2\in \C$. Theorem \ref{thm:bound} and a normal family argument shows that
\begin{align*}
&\lim_{n\to\infty} \E_n^s[X_k(\theta)X_l(\theta)] = \sqrt{kl} \lim_{n\to\infty} \frac{\partial^2}{\partial\lambda_1\partial\lambda_2} \Big|_{\lambda_1=\lambda_2=0} \E_n^s\big[e^{\frac{\lambda_1}{\sqrt{k}}X_k(\theta)+\frac{\lambda_2}{\sqrt{l}}X_l(\theta)} \big] \\
&= \sqrt{kl}  \frac{\partial^2}{\partial\lambda_1\partial\lambda_2} \Big|_{\lambda_1=\lambda_2=0} \exp \Big(\frac{1}{4\beta} \g_\lambda^t(I+sB)^{-1}\g_\lambda + \frac{1}{4}\big(1-\frac{2}{\beta})(s\dd+\f)^t(I+sB)^{-1}\g_\lambda \Big).
\end{align*}
Similarly,
\begin{align*}
&\lim_{n\to\infty} \E_n^s[X_k(\theta)] = \sqrt{k} \lim_{n\to\infty} \frac{\partial}{\partial\lambda_1} \Big|_{\lambda_1=0} \E_n^s\big[e^{\frac{\lambda_1}{\sqrt{k}}X_k(\theta)} \big] \\
&= \sqrt{k}  \frac{\partial}{\partial\lambda_1} \Big|_{\lambda_1=0} \exp \Big(\frac{\lambda_1^2}{4\beta} \mathbf{e}_k^t(I+sB)^{-1}\mathbf{e}_k + \frac{\lambda_1}{4}\big(1-\frac{2}{\beta})(s\dd+\f)^t(I+sB)^{-1}\mathbf{e}_k \Big).
\end{align*}
Computing the derivatives gives
\begin{multline}
\lim_{n\to\infty} \E_n^s[X_k(\theta)X_l(\theta)] = \frac{ \sqrt{kl}}{2\beta}(I+sB)^{-1}_{kl} \\
+\frac{\sqrt{kl}}{16}\big(1-\frac{2}{\beta}\big)^2 \big(\sum_{m_1\geq 1}(s\dd+\f)_{m_1}(I+sB)_{m_1k}^{-1} \big)  \big(\sum_{m_2\geq 1}(s\dd+\f)_{m_2}(I+sB)_{m_2l}^{-1} \big)  
\end{multline}
and 
\[\lim_{n\to\infty} \E_n^s[X_k(\theta)] = \frac{\sqrt{k}}{4}\big(1-\frac{2}{\beta} \big)\sum_{m\geq 1} (s\dd+\f)_m(I+sB)_{mk}^{-1}. \]
We now insert these limits into \eqref{Bnprimelimit} to obtain
\begin{align*}
\lim_{n\to\infty} B_n'(s) &= -\beta \sum_{k,l\geq 1} b_{kl} \Big( \frac{1}{2\beta} (I+sB)_{kl}^{-1} +\frac{1}{16}\big(1-\frac{2}{\beta}\big)^2\cdot  \\
&\sum_{m_1,m_2\geq 1}(s\dd+\f)_{m_1}(I+sB)_{m_1k}^{-1} (s\dd+\f)_{m_2}(I+sB)_{m_2l}^{-1} \Big) \\
&+\frac{1}{2\beta}\big(\frac{\beta}{2}-1\big)^2 \sum_{m,k\geq 1} \sqrt{k} d_k  (s\dd+\f)_m(I+sB)_{mk}^{-1}\\
&= - \frac{1}{2}\mathrm{Tr} \Big((I+sB)^{-1}B \Big) -\frac{1}{4\beta}\big(\frac{\beta}{2}-1\big)^2(s\dd+\f)^t(I+sB)^{-2}B(s\dd+\f)\\
&+\frac{1}{2\beta}\big(\frac{\beta}{2}-1\big)^2(s\dd+\f)^t(I+sB)^{-1}\dd.
\end{align*}
Note that
\[ \frac{\d}{\d s} \log \det (I+sB) = \mathrm{Tr} \Big((I+sB)^{-1}B \Big) \]
and
\begin{multline}
\frac{\d}{\d s}\Big( s^2\dd^t(I+sB)^{-1}\dd+2s\f^t(I+sB)^{-1}\dd - s\f^t(I+sB)^{-1}B\f \Big) \\= 2(s\dd+\f)^t(I+sB)^{-1}\dd-(s\dd+\f)^t(I+sB)^{-2}B(s\dd+\f)
\end{multline}
Consequently,

\begin{multline}\label{limBnprime}
\lim_{n\to\infty} B_n'(s) = \frac{\d}{\d s} \Big(-\frac{1}{2}\log \det (I+sB)+\frac{1}{4\beta}\big( \frac{\beta}{2}-1\big)^2\big(s^2\dd^t(I+sB)^{-1}\dd\\
+2s\f^t(I+sB)^{-1}\dd - s\f^t(I+sB)^{-1}B\f \big) \Big).
\end{multline}
It follows from \eqref{Bnprime} and the estimates \eqref{XkXlbound} and \eqref{Xkbound} that
\[ |B_n'(s)| \leq  C(\gamma) \Big( \sum_{k,l\geq 1} (kl)^{4+\epsilon} |a_{kl}| +\sum_{k\geq1} k^{7/2+\epsilon}|d_k| \Big) \leq C(\gamma) \]
by \eqref{aklest}, for all $s\in[0,1]$. We can thus apply the dominated convergence theorem to see that 
\[ \lim_{n\to\infty} B_n(1) = \lim_{n\to\infty} \int_0^1 B_n'(s) \d s = \int_0^1 \lim_{n\to\infty} B_n'(s) \d s. \]
Here we can insert the formula \eqref{limBnprime} and this gives \eqref{partquot}, by the definition \eqref{Bns} of the function $B_n(s)$.

\subsection{A smaller domain of integration}\label{sec:domain}

In this section we show that when computing the asymptotics of $\E_n^s[e^{X_n}]$ for any function $X_n$ growing at most linearly with $n$, we can replace the domain of integration $H_n$ with a smaller set $E_{n,s,K}$, see the definition \eqref{Ensk}. This is the content of Lemma \ref{lem:Ensk}. Moreover, for any sequence $\theta^{(n)} = (\theta_1^{(n)}, \cdots, \theta_n^{(n)}) \in E_{n,s,K}$ the normalized counting measure $\lambda_n = \frac{1}{n} \sum_{j=1}^{n} \delta_{\cos \theta_j^{(n)}}$ converges weakly to the equilibrium measure on $[-1,1]$. This is a consequence of Lemma \ref{lem:tmutozero}. For the proof of Theorem \ref{thm:relSz} we will also need a quantitative result, Lemma \ref{lem:sumtmu}, which gives a bound on the sum of the squares of the deviations of $\theta_\mu$ from 
\begin{equation}\label{alphamu}
\alpha_\mu = \frac{\pi (\mu-1)}{n-1},\quad  1\leq \mu\leq n
\end{equation}
where we assume $n\geq 2$ (as we will do from now on). Note that $\cos\alpha_\mu$, $1\leq \mu\leq n$ are good approximations to the Fekete points for the interval $[-1,1]$.
We begin by collecting some formulas that we will need in the proof of the lemmas.
\begin{lemma}\label{lem:formulas}
With $\alpha_\mu$ defined by \eqref{alphamu}, we have the formulas
\begin{equation}\label{sumcos}
\sum_{\mu=1}^n \cos(k\alpha_\mu) = \frac{1+(-1)^k}{2}, \quad k\in\N,\quad k/(2(n-1)) \notin \N,
\end{equation}
\begin{equation}\label{sumlogcos}
\sum_{1\leq \mu\neq\nu\leq n} \log |\cos \alpha_\mu-\cos\alpha_\nu| = n\log(n-1)-(n^2-2n-2)\log 2,
\end{equation}
\begin{equation}\label{Fnsalpha}
F_n^s(\alpha) = n\log n -(n^2-2n)\log 2+d_{n,s}
\end{equation}
where $|d_{n,s}|\leq C(\gamma)$ for all $n\geq 2$, $0\leq s\leq 1$. We also note the identity
\begin{equation}\label{sumcot}
\sum_{\mu=2}^{n-1} \cot \alpha_\mu =0.
\end{equation}
\end{lemma}

\begin{proof}
If $k$ is not a multiple of $2(n-1)$,
\[ \sum_{\mu=1}^n \cos(k\alpha_\mu) = \Re \frac{1-e^{\i k\pi n/(n-1)}}{1-e^{\i k\pi/(n-1)}} = \cos(k\pi/2) \frac{\sin(k\pi n/(2(n-1)))}{\sin(k\pi/(2(n-1))} = \cos^2(k\pi/2) \]
which gives \eqref{sumcos}.
For the second equality, \eqref{sumlogcos}, observe that
\begin{align*}
&\sum_{\mu\neq\nu} \log |\cos \alpha_\mu-\cos\alpha_\nu| = \sum_{\mu=0}^{n-1}\sum_{\substack{\nu=0\\ \nu\neq\mu}}^{n-1} \log|1-e^{\i\pi(\mu-\nu)/(n-1)}|\cdot |1-e^{\i\pi(\mu+\nu)/(n-1)}| \\ & - (n^2-n)\log 2 
= \sum_{\mu=0}^{n-1} \sum_{\substack{\nu=-(n-1)\\ \nu\neq \mu}}^{n-2}\log |1-e^{\i\pi(\mu+\nu)/(n-1)}| + \sum_{\mu=1}^{n-1} \log|1-e^{\i\pi\mu/(n-1)}| \\
&+ \sum_{\mu=0}^{n-2} \log|1-e^{\i\pi(\mu+n-1)/(n-1)}|  - \sum_{\mu=1}^{n-2}\log|1-e^{\i 2\pi\mu/(n-1)}| -(n^2-n)\log 2 \\
&= (n+1) \sum_{\nu=1}^{2(n-1)-1}\log |1-e^{\i\pi\nu/(n-1)}| - \sum_{\mu=1}^{n-2}\log|1-e^{\i 2\pi\mu/(n-1)}| -(n^2-n-1)\log 2.
\end{align*}
Now we use that for any $m\in \N$,
\begin{equation}\label{prod}
\prod_{k=1}^{m-1} (1-e^{\i2\pi k/m}) = \lim_{z\to 1}\prod_{k=1}^{m-1} (z-e^{\i2\pi k/m})=\lim_{z\to 1} \frac{z^m-1}{z-1} =m.
\end{equation}
We obtain
\begin{align*}
\sum_{\mu\neq\nu} \log |\cos \alpha_\mu-\cos\alpha_\nu| &= (n+1)\log(2(n-1))-\log(n-1)-(n^2-n-1)\log 2 \\
&= n\log(n-1)-(n^2-2n-2)\log 2.
\end{align*}
Next, by definition \eqref{Fns2} of $F_n^s$,
\begin{align*}
F_n^s(\alpha) &= \sum_{\mu\neq\nu} \log |\cos\alpha_\mu-\cos\alpha_\nu|-2s \sum_{k,l=1}^n a_{kl} \big( \sum_\mu \cos(k\alpha_\mu) \big) \big( \sum_\nu \cos(l\alpha_\nu) \big) \\
&= n\log(n-1)-(n^2-2n-2)\log 2 -2s\sum_{k,l=1}^{\floor{n/2}} a_{2k,2l}
\end{align*}
where we used \eqref{sumcos} and \eqref{sumlogcos}. This gives \eqref{Fnsalpha} because the sum is bounded thanks to \eqref{Grunskydecay}.
The last identity, \eqref{sumcot}, follows from the fact that $\cot\frac{\pi k}{n-1} = -\cot\frac{\pi(n-1-k)}{n-1}$ and $\cot \pi/2=0$.
\end{proof}

Fix $\theta\in H_n = \{ \theta\in [0,\pi]^n: 0\leq \theta_1 < \cdots < \theta_n \leq \pi\}$ and $t = (t_1, \cdots, t_n)\in \R^n$ and define the function
\begin{equation}\label{fnthetatau}
f_{n,\theta}(\tau) = F_n^s(\theta+\tau t)
\end{equation}
for $\tau\in[0,1]$. We will compute the derivatives with respect to $\tau$. Write $\omega_\mu(\tau) = \theta_\mu+\tau t_\mu$ and note that $\omega_\mu' = t_\mu$, $\omega_\mu'' =0$. By \eqref{Fns2},
\begin{multline}
f_{n,\theta} (\tau) = n(n-1)\log 2 + \sum_{\mu\neq\nu} \Big( \log \sin \big| \frac{\omega_\mu-\omega_\nu}{2}\big| + \log \sin \big| \frac{\omega_\mu+\omega_\nu}{2} \big| \Big)  \\ -2s \sum_{k,l=1}^n a_{kl} \big(\sum_\mu \cos (k\omega_\mu) \big) \big( \sum_\nu \cos(l\omega_\nu) \big).
\end{multline}
Differentiating gives
\begin{multline}\label{fnprime}
f_{n,\theta}'(\tau) = \sum_{\mu\neq \nu} \Big( \cot\big(\frac{\omega_\mu-\omega_\nu}{2}\big)\frac{t_\mu-t_\nu}{2}+\cot\big(\frac{\omega_\mu+\omega_\nu}{2}\big)\frac{t_\mu+t_\nu}{2}  \Big) \\
+4s \sum_{k,l=1}^n ka_{kl} \big(\sum_\mu t_\mu\sin (k\omega_\mu) \big) \big( \sum_\nu \cos(l\omega_\nu) \big).
\end{multline}
where we used that $a_{kl}=a_{lk}$. Furthermore
\begin{multline}\label{fnbis}
f_{n,\theta}''(\tau) = -\sum_{\mu\neq\nu} \Big(\frac{(t_\mu-t_\nu)^2}{4\sin^2\frac{\omega_\mu-\omega_\nu}{2}}+\frac{(t_\mu+t_\nu)^2}{4\sin^2\frac{\omega_\mu+\omega_\nu}{2}} \Big) + 4s \sum_{k,l=1}^n k^2a_{kl} \big(\sum_\mu t_\mu^2 \cos(k\omega_\mu) \big) \times \\ \big( \sum_\nu \cos(l\omega_\nu) \big)
- 4s \sum_{k,l=1}^n kla_{kl} \big(\sum_\mu t_\mu \sin(k\omega_\mu) \big) \big( \sum_\nu t_\nu \sin(l\omega_\nu) \big).
\end{multline}
We want to define a domain of integration smaller than $H_n$ which still captures the essential contribution. This is provided by the next lemma.

\begin{lemma}\label{lem:Ensk}
Given a constant $K$, define
\begin{equation}\label{Ensk}
E_{n,s,K} = \{ \theta\in H_n: F_n^s(\theta) > F_n^s(\alpha) -Kn \}.
\end{equation}
Then, there is a constant $C(\gamma)$ so that
\begin{equation}\label{domainest}
\P_n^s[E_{n,s,K}^c] \leq e^{(C(\gamma)-\beta K/2)n}.
\end{equation}
\end{lemma}

\begin{proof}
In order to prove \eqref{domainest} we need an estimate of $Z_n^s$ from below. We will prove the following estimate
\begin{equation}\label{Znest}
Z_n^s \geq n^{-n} \exp \big( \frac{\beta}{2}F_n^s(\alpha)-C(\gamma)n \big)
\end{equation}
for some constant $C(\gamma)$. Assuming \eqref{Znest} we can prove \eqref{domainest}. In fact,
\begin{align*}
&\P_n^s[E_{n,s,K}^c] = \frac{1}{Z_n^s} \int_{E_{n,s,K}^c} \exp \big(\frac{\beta}{2}F_n^s(\theta) + (1-\frac{\beta}{2}) \sum_\mu L(\theta_\mu)  \big) \prod_\mu \sin\theta_\mu \d\theta_\mu \\
&\leq e^{C(\gamma)n} n^n \int_{E_{n,s,K}^c} \exp \big( \frac{\beta}{2}(F_n^s(\theta) -F_n^s(\alpha) ) \big) \prod_\mu \sin\theta_\mu \d\theta_\mu \\
&\leq e^{(C(\gamma)-\beta K/2)n} n^n \int_{H_n}  \prod_\mu \sin\theta_\mu \d\theta_\mu \\
&= e^{(C(\gamma)-\beta K/2)n} \frac{n^n}{n!} \big( \int_0^\pi  \sin\theta \d\theta \big)^n \leq  e^{(C(\gamma)-\beta K/2)n}. \\
\end{align*}
We turn now to the proof of \eqref{Znest}. By \eqref{Grunskyexp}, \eqref{Grunskydecay}, and \eqref{L},
\[ L(\theta) = |\log z'(\cos\theta)| \leq 2 \sum_{k,l=1}^n |a_{kl}| \leq C(\gamma) \]
for all $\theta\in H_n$, since we assume that $\gamma\in C^{17+\alpha}$. Let
\[\Omega_n = \{t\in \R^n: 0\leq t_1 \leq \frac{1}{2n}, -\frac{1}{2n} \leq t_n \leq 0, | t_\mu | \leq \frac{1}{2n}, 1<\mu <n \}.\]
Then,
\begin{align}\label{Znsbelow}
Z_n^s &\geq \int_{H_n}  \exp \big( \frac{\beta}{2}F_n^s(\theta)-C(\gamma)n \big) \prod_\mu \sin\theta_\mu \d\theta_\mu \nonumber\\
&\geq \exp (-C(\gamma) n) \int_{\Omega_n} \exp \big( \frac{\beta}{2}F_n^s(\alpha + t) \big) \prod_\mu \sin(\alpha_\mu + t_\mu) \d t_\mu \nonumber \\
&=  \exp \big(\frac{\beta}{2}F_n^s(\alpha)-C(\gamma) n\big) \int_{\Omega_n} \exp \big( \frac{\beta}{2}(f_{n,\alpha}(1)-f_{n,\alpha}(0)) \big) \prod_\mu \sin(\alpha_\mu + t_\mu) \d t_\mu
\end{align}
where $f_{n,\alpha}(\tau) = F_n^s(\alpha+\tau t)$.
We write
\begin{equation}\label{taylor1}
f_{n,\alpha}(1)-f_{n,\alpha}(0) = f_{n,\alpha}'(0) + \int_0^1(1-\tau) f_{n,\alpha}''(\tau) \d\tau.
\end{equation}
It follows from \eqref{fnprime} that
\begin{multline}\label{fnprimezero}
f_{n,\alpha}'(0) =  \sum_{\mu\neq\nu} \cot\big(\frac{\alpha_\mu-\alpha_\nu}{2} \big)\frac{t_\mu-t_\nu}{2}+\cot\big(\frac{\alpha_\mu+\alpha_\nu}{2} \big)\frac{t_\mu+t_\nu}{2}\\ +4s\sum_{k,l=1}^n k a_{kl} \big(\sum_\mu t_\mu \sin(k\alpha_\mu) \big)\big(\sum_\mu \cos(k\alpha_\mu) \big).
\end{multline}
To estimate $f_{n,\alpha}'(0)$ we will use the following identity
\begin{equation}\label{cotid}
\sum_{\mu\neq\nu} \cot\big(\frac{\alpha_\mu-\alpha_\nu}{2} \big)\frac{t_\mu-t_\nu}{2}+\cot\big(\frac{\alpha_\mu+\alpha_\nu}{2} \big)\frac{t_\mu+t_\nu}{2} = \sum_{\mu=2}^{n-1} t_\mu \cot \frac{\pi(\mu-1)}{n-1}.
\end{equation}
To see this notice that we can write
\[
\sum_{\mu\neq\nu} \cot\big(\frac{\alpha_\mu-\alpha_\nu}{2} \big)\frac{t_\mu-t_\nu}{2}+\cot\big(\frac{\alpha_\mu+\alpha_\nu}{2} \big)\frac{t_\mu+t_\nu}{2}=\sum_{\mu=2}^{n-1}c_\mu t_\mu,
\]
where
\[
c_\mu=\sum_{\nu=1,\nu\neq\mu}^{n}\cot\frac{\pi(\mu-\nu)}N+\cot\frac{\pi(\mu+\nu-2)}N,
\]
and $N=2(n-1)$. If $1<\mu<n$ some manipulations of the sum and using the identity
\begin{equation}\label{sumcotN}
\sum_{k=1}^{N-1}\cot\frac{\pi k}N=0,
\end{equation}
gives
\[
c_\mu=\cot\Bigg(\frac{\pi(\mu-1)}N+\frac{\pi}2\Bigg)+\cot\frac{\pi(\mu-1)}N-\cot\frac{2\pi(\mu-1)}N=\cot\frac{\pi(\mu-1)}{n-1}.
\]
It is easy to see that $c_1=0$, and $c_n=0$ follows from \eqref{sumcotN}.
The right side of \eqref{cotid} can be estimated as follows.
If $t\in \Omega_n$, then
\begin{equation}\label{cottmuest}
 \Big| \sum_{\mu=2}^{n-1} t_\mu \cot \frac{\pi(\mu-1)}{n-1} \Big| \leq \frac{1}{2n} \sum_{\mu=2}^{n-1} \big| \cot\frac{\pi(\mu-1)}{n-1} \big| \leq C\log n.
\end{equation}
Also, if $t\in\Omega_n$, then using \eqref{sumcos}, we see that
\[ \big| 4s\sum_{k,l=1}^n k a_{kl} \big(\sum_\mu t_\mu \sin(k\alpha_\mu) \big)\big(\sum_\mu \cos(k\alpha_\mu) \big) \big| \leq 4 \big( \max_{1\leq \mu\leq n} |t_\mu| \big)n \sum_{k,l\geq 1} k|a_{kl}| \leq C(\gamma). \]
This estimate, together with \eqref{fnprimezero}, \eqref{cotid}, and \eqref{cottmuest}, shows that
\begin{equation}\label{fnprimezerotest}
|f_{n,\alpha}'(0)| \leq C(\gamma)\big( \max_{1\leq \mu\leq n} |t_\mu| \big)n\log n \leq C(\gamma) \log n.
\end{equation}
To bound the second derivative in \eqref{taylor1} we will use the following inequality. If $t\in \Omega_n$, then,
\begin{equation}\label{invsinest}
\frac{1}{n^2} \sum_{\mu\neq \nu} \Big( \frac{1}{\sin^2\frac{\omega_\mu-\omega_\nu}{2}}+\frac{1}{\sin^2\frac{\omega_\mu+\omega_\nu}{2}} \Big) \leq Cn,
\end{equation}
where $\omega_\mu = \alpha_\mu+\tau t_\mu$ and $\tau\in [0,1]$. To prove this inequality, note that if $t\in \Omega_n$, there is a numerical constant $C$ such that
\begin{equation}\label{invsinest2}
 \sum_{\mu\neq \nu} \Big( \frac{1}{\sin^2\frac{\omega_\mu-\omega_\nu}{2}}+\frac{1}{\sin^2\frac{\omega_\mu+\omega_\nu}{2}} \Big)\le
C\sum_{\mu\neq \nu} \Big( \frac{1}{\sin^2\frac{\alpha_\mu-\alpha_\nu}{2}}+\frac{1}{\sin^2\frac{\alpha_\mu+\alpha_\nu}{2}} \Big).
\end{equation}
Manipulations of the sum gives, with $N=2(n-1)$ as before,
\[
\sum_{\mu,\nu=1,\mu\neq\nu}^n\frac 1{\sin^2\frac{\alpha_\mu+\alpha_\nu}2}=\sum_{\mu=1}^{n-1}\sum_{\nu = n+1}^{2n-2}\frac 1{\sin^2\frac{\pi(\mu-\nu)}N},
\]
so the right side of \eqref{invsinest2} is bounded by
\begin{align*}
&C\Bigg(\sum_{\mu,\nu=1,\mu\neq\nu}^n+\sum_{\mu=1}^{n-1}\sum_{\nu = n+1}^{2n-2}\Bigg)\frac 1{\sin^2\frac{\pi(\mu-\nu)}N}\le 2C \sum_{\mu,\nu=1,\mu\neq\nu}^N
\frac 1{\sin^2\frac{\pi(\mu-\nu)}N}\\
&=2C(N-1)\sum_{\nu=1}^{N-1}\frac 1{\sin^2\frac{\pi\nu}N}\le 8Cn\sum_{\nu=1}^{n-1}\frac 1{\sin^2\frac{\pi\nu}{2(n-1)}}\le 8Cn\sum_{\nu=1}^{n-1}\frac{(n-1)^2}{\nu^2}\le C'n^3,
\end{align*}
for some numerical constant $C'$.
It follows from \eqref{invsinest}
\[ |f_{n,\alpha}''(\tau) | \leq \frac{1}{16n^2} \sum_{\mu\neq\nu}\Big( \frac{1}{\sin^2\frac{\omega_\mu-\omega_\nu}{2}}+\frac{1}{\sin^2\frac{\omega_\mu+\omega_\nu}{2}} \Big) + 4 \sum_{k,l\geq 1} (k^2+kl) |a_{kl}| \leq C(\gamma) n.  \]
Combining this estimate with \eqref{fnprimezerotest}, we see that \eqref{taylor1} gives
\[ |f_{n,\alpha}(1)-f_{n,\alpha}(0) |\leq C(\gamma) n. \]
It now follows from \eqref{Znsbelow} that
\[ Z_n^s \geq e^{\frac{\beta}{2}F_n^s(\alpha)-C(\gamma)n}\int_{\Omega_n} \prod_\mu\sin(\alpha_\nu+t_\mu) \d t_\mu. \]
By the definition of $\Omega_n$
\[\int_{\Omega_n} \prod_\mu \sin(\alpha_\mu+t_\mu) \d t_\mu = \big(1-\cos\frac{1}{2n}\big)^2 \big(\sin\frac{1}{2n}\big)^{n-2} \prod_{\mu=2}^{n-1} 2\sin\frac{\pi(\mu-1)}{n-1} \geq e^{-Cn}n^{-n} \]
since $\prod_{\mu=2}^{n-1} 2 \sin\frac{\pi (\mu-1)}{n-1} = \prod_{\mu=1}^{n-2} |1-e^{\i 2\pi\mu/(n-1)}| = n-1$ as in \eqref{prod}. We have proved \eqref{Znest}.
\end{proof}

Given $\theta \in H_n$ we define
\begin{equation}\label{tmudef}
t_\mu = t_\mu(\theta) = \theta_\mu-\alpha_\mu, \quad 1\leq \mu \leq n.
\end{equation}
Let \begin{equation}\label{epsilonn}
\epsilon_n = \sup_{0<s<1} \sup_{\theta \in E_{n,s,K}} \max_{1\leq \mu \leq n} |t_\mu(\theta)|.
\end{equation}
We want to show that if $\theta$ is in $E_{n,s,K}$ then the deviations $t_\mu$ are small. The next lemma is our first result in this direction.

\begin{lemma}\label{lem:tmutozero}
Fix $K$ independent of $s$ in the definition of $E_{n,s,K}$. Then $\epsilon_n\to 0$ as $n\to\infty$.
\end{lemma}

\begin{proof}
Suppose that $\limsup_{n\to\infty} \epsilon_n = \epsilon>0$, so that there is a subsequence $\{ n_k \}_{k\geq 1}$ such that $\epsilon_{n_k} \to\epsilon$ as $k\to \infty$. Then there exists $s_k\in[0,1]$ and $\theta^{(k)} \in E_{n_k,s_k,K}$ such that
\begin{equation}\label{maxtmu}
\max_{1\leq \mu \leq n} |t_\mu(\theta^{(k)})| \geq \epsilon/2
\end{equation}
for all sufficiently large $k$. Define the probability measures
\[ \lambda_k = \frac{1}{n_k} \sum_{j=1}^{n_k} \delta_{\cos \theta_j^{(k)}}, \quad \mu_k= \frac{1}{n_k} \sum_{j=1}^{n_k} \delta_{z(\cos(\theta_j^{(k)}))} \]
on $[-1,1]$ and $\gamma$, respectively. By picking a further subsequence we can assume that $\lambda_k$ converges weakly to $\lambda$ and $\mu_k$ converges weakly to $\mu$ as $k\to\infty$, where $\lambda$ and $\mu$ are two probability measures. It follows from \eqref{maxtmu}, and the fact that the $\theta_\mu$ are ordered, that $\lambda$ is not the equilibrium measure on $[-1,1]$, and consequently, $\mu$ is not the equilibrium measure on $\gamma$. By \eqref{Fns}, \eqref{Fnsalpha} and \eqref{Ensk},
\begin{multline*}
(1-s_k) \sum_{\mu\neq\nu} \log |\cos\theta_\mu-\cos\theta_\nu|^{-1} + s_k \sum_{\mu\neq\nu} \log |z_e(\cos\theta_\mu)-z_e(\cos\theta_\nu)|^{-1} \\
= -F_{n_k}^{s_k}(\theta^{(k)})+s_k\sum_{\mu} \big(\log|z_e'(\cos(\theta_\mu)|-\log(2\text{cap}(\gamma))\big)-s_kn_k^2\log(2\text{cap}(\gamma))
\\
\le -F_{n_k}^{s_k}(\alpha)+Kn_k -s_kn_k^2\log(2\text{cap}(\gamma))+C(\gamma)n_k\\
\le -n_k\log n_k+(1-s_k)n_k^2\log 2-s_kn_k^2\log\text{cap}(\gamma)+C(\gamma)n_k
\end{multline*}
Fix $M>0$ and let $\log_M(x) = \min(\log x, M)$. Then, the above estimate gives
\begin{multline*}
(1-s_k)\int_{-1}^1\int_{-1}^1 \log_M|x-y|^{-1} \d\lambda_k(x)\d\lambda_k(y) + s_k\int_\gamma\int_\gamma \log_M|z-w|^{-1} \d\mu_k(z)\d\mu_k(w) \\
\leq \frac{1}{n_k^2}(-n_k\log n_k + (1-s_k)n_k^2\log 2-s_kn_k^2\log\text{cap}(\gamma)+C(\gamma)n_k +M n_k).
\end{multline*}
If we let $k\to\infty$ we obtain the inequality
\begin{multline*}
(1-s)\int_{-1}^1\int_{-1}^1 \log_M|x-y|^{-1} \d\lambda(x)\d\lambda(y) + s\int_\gamma\int_\gamma \log_M |z-w|^{-1} \d\mu(z)\d\mu(w)\\
\leq -(1-s)\log\text{cap}([-1,1])-s\log\text{cap}(\gamma).
\end{multline*}
We can now let $M\to\infty$ and see that
\begin{equation}\label{senergyest}
(1-s)I[\lambda]+sI[\mu] \leq -(1-s)\log\text{cap}([-1,1])-s\log\text{cap}(\gamma).
\end{equation}
Since neither $\lambda$ nor $\mu$ equals the equilibrium measure, $I[\lambda]>\log 2$ and $I[\mu] >\log\text{cap}(\gamma)$. This contradicts \eqref{senergyest}. Hence, $\epsilon_n\to 0$ as $n\to\infty$ (and both $\lambda$ and $\mu$ are the equilibrium measures on $[-1,1]$ and $\gamma$).

\end{proof}

As a consequence of Lemma \ref{lem:tmutozero}, for any sequence $\theta^{(n)} = (\theta_1^{(n)}, \cdots, \theta_n^{(n)}) \in E_{n,s,K}$ the measure $\lambda_n = \frac{1}{n} \sum_{j=1}^{n} \delta_{\cos \theta_j^{(n)}}$ converges weakly to the equilibrium measure on $[-1,1]$. Indeed for any continuous function $f$ on the unit circle, and any $\epsilon>0$,
\[ \big| \frac{1}{n}\sum_\mu f(\cos \theta_\mu^{(n)}) -\frac{1}{n}\sum_\mu f(\cos \alpha_\mu) \big| \leq \sup_\mu |f(\cos \theta_\mu)-f(\cos \alpha_\mu)| < \epsilon \]
if $\sup_\mu |t_\mu|$ is small enough, by uniform continuity. By the above lemma, the latter condition holds for large enough $n$, so $\lambda_n - \frac{1}{n} \sum_{j=1}^{n} \delta_{\cos \alpha_j}$ converges weakly to zero. But $\frac{1}{n} \sum_{j=1}^{n} \delta_{\cos \alpha_j}$ converges weakly to the equilibrium measure on $[-1,1]$ so the same holds for $\lambda_n$.

The next lemma gives us further control of the size of the deviations $t_\mu$ for $\theta = (\theta_1,\cdots,\theta_n) \in  E_{n,s,K}$.

\begin{lemma}\label{lem:sumtmu}
There is a constant $C_K(\gamma)$ such that
\begin{equation}\label{sumtmu}
\sum_\mu t_\mu^2 \leq C_K(\gamma)
\end{equation}
for all $\theta\in E_{n,s,K}$, where $C_K(\gamma)$ depends only on $K$ and $\gamma$ but not on $s$.
\end{lemma}

\begin{proof}
Fix $\theta\in E_{n,s,K}$ and let $\epsilon_n = \max_\mu |t_\mu|$. It follows from \eqref{fnprimezerotest} that
\[ f_{n,\alpha}'(0) \leq C(\gamma) \epsilon_n n\log n. \]
Let 
\begin{align*}
&A(\tau) = 4s \sum_{k,l=1}^nk^2 a_{kl} \big(\sum_\mu t_\mu^2 \cos (k\omega_\mu) \big) \big( \sum_\nu \cos(l\omega_\nu) \big) \\
&B(\tau) = 4s \sum_{k,l=1}^nkl a_{kl} \big(\sum_\mu t_\mu \sin (k\omega_\mu) \big) \big( \sum_\nu t_\nu \sin(l\omega_\nu) \big)
\end{align*}
where $\omega_\mu = \omega_\mu (\tau) = \alpha_\mu + \tau t_\mu$ as above. Then by \eqref{fnbis},
\[ f_{n,\alpha}''(\tau) = -\frac{1}{4} \sum_{\mu\neq\nu} \Big( \frac{(t_\mu-t_\nu)^2}{\sin^2\frac{\omega_\mu-\omega_\nu}{2}} +  \frac{(t_\mu+t_\nu)^2}{\sin^2\frac{\omega_\mu+\omega_\nu}{2}}\Big) + A(\tau)-B(\tau). \]
Using this identity and
\[ F_n^s(\theta)-F_n^s(\alpha) = f_{n,\alpha}'(0)+\int_0^1 (1-\tau)f_{n,\alpha}''(\tau)\d\tau\]

we obtain, for $\theta\in E_{n,s,K}$,
\begin{align}\label{Est}
&\frac{1}{4} \int_0^1 \sum_{\mu\neq\nu} \Big( \frac{(t_\mu-t_\nu)^2}{\sin^2\frac{\omega_\mu-\omega_\nu}{2}} +  \frac{(t_\mu+t_\nu)^2}{\sin^2\frac{\omega_\mu+\omega_\nu}{2}} \Big) (1-\tau)\d\tau \nonumber \\
&= F_n^s(\alpha) -F_n^s(\theta) +f_{n,\alpha}'(0)+\int_0^1(A(\tau)-B(\tau))(1-\tau)\d\tau \nonumber\\
&\leq Kn + C(\gamma)\epsilon_n n\log n+ \int_0^1(A(\tau)-B(\tau))(1-\tau)\d\tau
\end{align}
We want to estimate $A(\tau)$ and $B(\tau)$. Using \eqref{sumcos} and the definition of $\epsilon_n$, we see that
\begin{equation}\label{sumcosest}
\big| \sum_\mu \cos(k\omega_\mu) \big| \leq \big| \sum_\mu (\cos(k\omega_\mu)-\cos(k\alpha_\mu)) \big| + \big| \sum_\mu \cos(k\alpha_\mu) \big| \leq kn\epsilon_n+1.
\end{equation}
For a $C^{6+\alpha}$ curve, \eqref{Grunskydecay} gives the estimate
\[ \sum_{k,l=1}^n lk^2 |a_{kl}| \leq \sum_{k,l\geq 1} lk^2\frac{A}{k^{3+\alpha/2}l^{2+\alpha/2}} \leq C(\gamma). \]
Hence,
\begin{equation}\label{Aest}
|A(\tau)| \leq 4 \sum_{k,l=1}^n k^2|a_{kl}| \big( \sum_\mu t_\mu^2 \big) (1+ln\epsilon_n) \leq C(\gamma)(1+n\epsilon_n) \sum_\mu t_\mu^2.
\end{equation}
Fix an integer $q$, which will be specified below, and let
\begin{align*}
&B_1(\tau) = 4 \sum_{k,l=1}^q kla_{kl} \big(\sum_\mu t_\mu \sin(k\omega_\mu) \big) \big( \sum_\nu t_\nu \sin(l\omega_\nu) \big) \\
&B_2(\tau) = 8 \sum_{k\geq1} \sum_{l\geq q+1} kl a_{kl} \big(\sum_\mu t_\mu \sin(k\omega_\mu) \big) \big(\sum_\nu t_\nu \sin(l\omega_\nu) \big).
\end{align*}
It follows from the strengthened arc-Grunsky inequality \eqref{stGrineq} that
\[ B_1(\tau)\geq -4\kappa \sum_{k=1}^q k \big( \sum_\mu t_\mu \sin(k\omega_\mu) \big)^2. \] 
Let $\rho = \kappa^{1/(2 q)}$ so that $\rho^{2k} = \kappa^{k/q}\geq \kappa$ for $1\leq k\leq q$. Then
\begin{align}\label{B1est}
&-B_1(\tau) \leq 4 \sum_{k\geq 1} k\rho^{2k} \big( \sum_\mu t_\mu \sin(k\omega_\mu) \big)^2 = 4\sum_{ \mu,\nu} t_\mu t_\nu \sum_{k\geq 1} k\rho^{2k} \sin(k\omega_\mu)\sin(k\omega_\nu)\nonumber \\
&= 2 \sum_{\mu,\nu} t_\mu t_\nu \sum_{k\geq 1} k\rho^{2k} (\cos(k(\omega_\mu-\omega_\nu))-\cos(k(\omega_\mu+\omega_\nu))) \nonumber \\
& = -\sum_{\mu,\nu} (t_\mu-t_\nu)^2 \sum_{k\geq 1} k \rho^{2k} \cos(k(\omega_\mu-\omega_\nu))\nonumber -\sum_{\mu,\nu} (t_\mu+t_\nu)^2 \sum_{k\geq 1} k \rho^{2k} \cos(k(\omega_\mu+\omega_\nu)) \\
&+4\sum_{k\geq 1} k \rho^{2k} \big( \sum_\mu t_\mu^2\cos(k\omega_\mu) \big) \big( \sum_\nu \cos(k\omega_\nu) \big).
\end{align}
Taking the real part of the identity 
\[ \sum_{k\geq 1} k\rho^{2k}e^{\i kx} = \frac{\rho^2e^{\i x}}{(1-\rho^2e^{\i x})^2} \]
gives, after some computations
\[ \sum_{k\geq 1} k\rho^{2k} \cos(kx) = \frac{a\cos x -4 \sin^2\frac{x}{2}}{(a+4\sin^2\frac{x}{2})^2}, \]
where $a=(1/\rho-\rho)^2$. Thus,
\begin{align*}
&\big| \sum_{\mu,\nu} (t_\mu-t_\nu)^2 \sum_{k\geq1} k\rho^{2k} \cos(k(\omega_\mu-\omega_\nu)) \big| \\
&\leq \sum_{\mu,\nu} (t_\mu-t_\nu)^2 \frac{a|\cos(\omega_\mu-\omega_\nu)|+4\sin^2\frac{\omega_\mu-\omega_\nu}{2}}{(a+4\sin^2\frac{\omega_\mu-\omega_\nu}{2})^2}\\
&\leq \sum_{\mu,\nu} \frac{(t_\mu-t_\nu)^2}{a+4\sin^2 \frac{\omega_\mu-\omega_\nu}{2}} \leq \frac{1}{a+4} \sum_{\mu\neq \nu} \frac{(t_\mu-t_\nu)^2}{\sin^2\frac{\omega_\mu-\omega_\nu}{2}}.
\end{align*}
Similarly,
\begin{multline*}
 \Big| \sum_{\mu,\nu}(t_\mu+t_\nu)^2 \sum_{k\geq1} k\rho^{2k} \cos(k(\omega_\mu+\omega_\nu)) \Big| \leq \sum_{\mu,\nu} \frac{(t_\mu+t_\nu)^2}{a+4\sin^2\frac{\omega_\mu+\omega_\nu}{2}} \\ \leq \frac{1}{a+4}\sum_{\mu\neq \nu} \frac{(t_\mu+t_\nu)^2}{\sin^2\frac{\omega_\mu+\omega_\nu}{2}} + \frac{4}{a}\sum_\mu t_\mu^2.    
\end{multline*}
Furthermore, by \eqref{sumcosest},
\begin{multline*}
\Big| \sum_{k\geq 1} k\rho^{2k} \big(\sum_\mu t_\mu^2\cos(k\omega_\mu) \big) \big( \sum_\nu \cos(k\omega_\nu) \big) \Big| \leq \big(\sum_{k\geq1} k\rho^{2k} \big) \big(\sum_\mu t_\mu^2 \big) (kn\epsilon_n+1) \\ \leq \frac{1}{a}\big(\frac{2n\epsilon_n}{1-\rho^2}+1 \big)\big(\sum_\mu t_\mu^2 \big).     
\end{multline*}
We can use these estimates in \eqref{B1est} to see that
\begin{equation}\label{B1est2}
-B_1(\tau) \leq \frac{1}{a+4} \sum_{\mu\neq \nu} \Big(\frac{(t_\mu-t_\nu)^2}{\sin^2\frac{\omega_\mu-\omega_\nu}{2}}+\frac{(t_\mu+t_\nu)^2}{\sin^2\frac{\omega_\mu+\omega_\nu}{2}} \Big) + \frac{1}{a}(\frac{2n\epsilon_n}{1-\rho^2}+5) \big(\sum_\mu t_\mu^2 \big).
\end{equation}
Consider now $B_2(\tau)$. It follows from the inequality
\[ |\sum_\mu t_\mu\sin(k\omega_\mu) | \leq \sum_\mu |t_\mu| \leq \sqrt{n} \big(\sum_\mu t_\mu^2 \big)^{1/2}, \]
that
\[ |B_2(\tau)| \leq 8  n \big(\sum_\mu t_\mu^2 \big) \sum_{k\geq 1}\sum_{l\geq q+1} kl |a_{kl}| \]
For a $C^{7+\alpha}$ curve, \eqref{Grunskydecay} decay gives the estimate
\[ |a_{kl}| \leq \frac{A}{k^{2+\alpha/2} l^{4+\alpha/2}}. \]
Hence,
\begin{equation}\label{B2est}
|B_2(\tau)| \leq \frac{C(\gamma)}{q^{2+\epsilon}} n \sum_\mu t_\mu^2.
\end{equation}

We can now use \eqref{Aest}, \eqref{B1est2}, and \eqref{B2est} in \eqref{Est} to get the inequality
\begin{align}\label{Est2}
\big( \frac{1}{4}-\frac{1}{4+a} \big) \int_0^1 \sum_{\mu\neq\nu} \Big(\frac{(t_\mu-t_\nu)^2}{\sin^2\frac{\omega_\mu-\omega_\nu}{2}}+\frac{(t_\mu+t_\nu)^2}{\sin^2\frac{\omega_\mu+\omega_\nu}{2}} \Big) (1-\tau)\d\tau \nonumber\\
\leq Kn + C(\gamma)\epsilon_n n\log n + C(\gamma)n\big( \frac{1}{q^{2+\epsilon}}+ \frac{\epsilon_n}{a(1-\rho^2)}+\frac{1}{an}\big)\big(\sum_\mu t_\mu^2 \big)
\end{align}
Note that
\[ \sum_{\mu\neq\nu} \Big(\frac{(t_\mu-t_\nu)^2}{\sin^2\frac{\omega_\mu-\omega_\nu}{2}}+\frac{(t_\mu+t_\nu)^2}{\sin^2\frac{\omega_\mu+\omega_\nu}{2}} \Big) \geq \sum_{\mu\neq \nu} 2(t_\mu^2+t_\nu^2) \geq 2n \sum_\mu t_\mu^2 \]
since $n\geq 2$. Thus \eqref{Est2} gives the inequality
\begin{equation}\label{Est3}
\Big( \frac{a}{4(4+a)}-\frac{C(\gamma)}{q^{2+\epsilon}} \Big) \big( \sum_\mu t_\mu^2 \big) \leq K+ C(\gamma)\epsilon_n\log n + C(\gamma) \big(  \frac{\epsilon_n}{a(1-\rho^2)}+\frac{1}{an}\big)\big(\sum_\mu t_\mu^2 \big).
\end{equation}
Now, since $a=(1/\rho-\rho)^2$ and $\rho=\kappa^{1/(2q)}<1$, we see that
\begin{multline*}
 \frac{a}{4(4+a)} = \frac{1}{4}\Big(\frac{1-\kappa^{1/q}}{1+\kappa^{1/q}} \Big)^2 \geq \frac{1}{16}(1-\kappa^{1/q})^2 = \frac{1}{16}\Big(\frac{1}{q}\log(\frac{1}{\kappa})+O(\frac{1}{q^2}) \Big)^2 \\ = \frac{1}{16q^2} \big( (\log\frac{1}{\kappa})^2+O(\frac{1}{q}) \big).   
\end{multline*}
Since $\kappa$ is fixed, we can pick $q$ so large that there is a constant $c_0(\gamma)>0$ such that
\[ \frac{a}{4(4+a)}-\frac{C(\gamma)}{q^{2+\epsilon}} \geq c_0(\gamma). \]
The inequality \eqref{Est3} then gives the estimate
\[ \Big(c_0(\gamma)-C(\gamma)\big( \frac{\epsilon_n}{a(1-\rho^2)}+\frac{1}{an}\big)\Big)\sum_\mu t_\mu^2 \leq K+C(\gamma)\epsilon_n\log n. \]
By Lemma \eqref{lem:tmutozero} $\epsilon_n\to 0$ as $n\to\infty$. Thus, since $\rho,a$ are now fixed, we can pick $n$ so large that
\[ C(\gamma) \Big(\frac{\epsilon_n}{a(1-\rho^2)}+\frac{1}{an} \Big) \leq \frac{1}{2}c_0(\gamma), \]
and we have proved the inequality
\begin{equation}\label{Est4}
\sum_\mu t_\mu^2 \leq \frac{2}{c_0(\gamma)}\big(K+C(\gamma)\epsilon_n\log n \big).
\end{equation}
The final step in the proof is to show that $\epsilon_n = \max_\mu |t_\mu| = |t_{\mu_0}|$ decays faster than $\log n$. Consider the case when $t_{\mu_0}>0$. If $\nu>\mu_0$, then $t_\nu-t_{\mu_0}\geq -\frac{\pi(\nu-\mu_0)}{n-1}$, and consequently
$t_\nu \geq t_{\mu_0}-\frac{\pi(\nu-\mu_0)}{n-1} \geq t_{\mu_0}/2$ if $\mu_0 < \nu \leq \mu_0 + \floor{\frac{t_{\mu_0}(n-1)}{2\pi}}$. Suppose $\mu_0+\lfloor \frac{t_{\mu_0}}{2\pi}(n-1) \rfloor >n$. Then, $\pi \geq \theta_{\mu_0} = \frac{\pi(\mu_0-1)}{n-1}+t_{\mu_0}$ gives
\[ t_{\mu_0} \leq \frac{\pi(n-\mu_0)}{n-1} < \frac{\pi}{n-1}\floor{\frac{t_{\mu_0}}{2\pi}(n-1)} \leq \frac{t_{\mu_0}}{2}, \]
which is impossible. Hence, for $\mu_0<\nu \leq \mu_0+ \floor{\frac{t_{\mu_0}}{2\pi}(n-1)} \leq n$ we have $t_{\nu} \geq t_{\mu_0}/2$ and thus
\[ \sum_\mu t_\mu^2 \geq \floor{\frac{t_{\mu_0}}{2\pi}(n-1)}\frac{t_{\mu_0}^2}{4} \geq cnt_{\mu_0}^3.\]
It follows from \eqref{Est4} that
\[ cnt_{\mu_0}^3 \leq C_k(\gamma) \log n \]
so $t_{\mu_0} \leq C_k(\gamma) \big(\frac{\log n}{n} \big)^{1/3}$. The case $t_{\mu_0}<0$ is treated analogously. We conclude that $\epsilon_n \leq C_k(\gamma) \big(\frac{\log n}{n} \big)^{1/3}$, which inserted in \eqref{Est4} gives $\sum_\mu t_\mu^2 \leq C_K(\gamma)$.
\end{proof}

\subsection{Proof of Theorem \ref{thm:bound}}\label{sec:proof}

In this section we prove Theorem \ref{thm:bound} which gives the asymptotic formula for the Laplace functional of the empirical measure. 
By \eqref{Dn3}, we have
\begin{equation}\label{EnexpG}
\E_n^s\big[e^{\sum_\mu G(\omega_\mu)}\big] = \frac{1}{Z_n^s} \int_{H_n} \exp \big( \frac{\beta}{2}F_n^s(\omega)+\sum_\mu(G(\omega_\mu)+(1-\beta/2)sL(\omega_\mu)+\log(\sin\omega_\mu)) \big) \d\omega.
\end{equation}
Let $H_s(\theta)$ be given by Lemma \ref{lem:Hreg}. In the integral in the right side of \eqref{EnexpG} we make the change of variables $\omega_\mu = \theta_\mu +\frac{1}{n}H_s(\theta_\mu)$. Note that, by definition, $H_s(0) = H_s(\pi) =0$. If $n$ is sufficiently large this maps the set $H_n$ to itself. For $0\leq \tau\leq 1$, we let $\omega_\mu(\tau) = \theta_\mu+\frac{\tau}{n}H_s(\theta_\mu)$. Let 
\[ M(\omega) = (1-\beta/2)sL(\omega)+\log(\sin\omega)+G(\omega). \]
Using the notation \eqref{fnthetatau} with $t_\mu = \frac{1}{n} H_s(\theta_\mu)$,
\begin{equation}\label{EnexpG2}
\E_n^s\big[e^{\sum_\mu G(\omega_\mu)}\big] = \frac{1}{Z_n^s}\int_{H_n} \exp\Big(\frac{\beta}{2}f_{n,\theta}(1)+\sum_\mu \big(M(\omega_\mu)+\log (1+\frac{1}{n}H_s'(\omega_\mu))\big) \Big)\d\omega.
\end{equation}
We make the Taylor expansions
\[f_{n,\theta}(1) = f_{n,\theta}(0) + f_{n,\theta}'(0)+\frac{1}{2}f_{n,\theta}''(0)+\frac{1}{2}\int_0^1(1-\tau)^2 f_{n,\theta}^{(3)}(\tau)\d\tau \]
and
\[ M(\omega_\mu) = M(\theta_\mu)+ \frac{1}{n}M'(\theta_\mu)H_s(\theta_\mu) + \frac{H_s(\theta_\mu)^2}{n^2} \int_0^1 (1-\tau)M''(\omega_\mu(\tau))\d\tau.\]
Recall the expansion
\[ H_s(\theta) = \sum_{k\geq 1} h_k^s \sin(k\theta)\]
of $H_s$. We now use \eqref{fnprime} and \eqref{fnbis} to compute the derivatives in the Taylor expansions. We write
\begin{align}\label{Rns}
R_n^s(\theta) = \frac{\beta}{4n}\sum_{\mu,\nu} \big( \cot \frac{\theta_\mu-\theta_\nu}{2}(H_s(\theta_\mu)-H_s(\theta_\nu))+\cot \frac{\theta_\mu+\theta_\nu}{2}(H_s(\theta_\mu)+H_s(\theta_\nu))\big)\nonumber \\
+ \beta s\sum_{k,l=1}^n k a_{kl} h_k \big(\sum_\nu\cos(l\theta_\nu) \big)+\sum_\mu G(\theta_\mu),
\end{align}
\begin{equation}\label{Tns}
T_n^s(\theta) = -\beta s \sum_{k,l=1}^n k a_{kl}\big( h_k -\frac{2}{n}\sum_\mu H_s(\theta_\mu) \sin(k\theta_\mu) \big) \big( \sum_\nu \cos(l\theta_\nu) \big),
\end{equation}
\begin{align}\label{Uns}
U_n^s(\theta) &= \frac{1}{n}\sum_\mu \big( G'(\theta_\mu)+(1-\beta/2)sL'(\theta_\mu) \big) H_s(\theta_\mu) \nonumber
\\ &- \frac{\beta}{16n^2}\sum_{\mu,\nu} \Big(\frac{(H_s(\theta_\mu)-H_s(\theta_\nu))^2}{\sin^2\frac{\theta_\mu-\theta_\nu}{2}}+\frac{(H_s(\theta_\mu)+H_s(\theta_\nu))^2}{\sin^2\frac{\theta_\mu+\theta_\nu}{2}} \Big) \nonumber\\
&+ \frac{\beta s}{n^2} \sum_{k,l=1}^n k^2a_{kl} \big(\sum_\mu H_s(\theta_\mu)^2\cos(k\theta_\mu) \big) \big( \sum_\nu \cos(l\theta_\nu) \big) \nonumber \\ 
&- \frac{\beta s}{n^2} \sum_{k,l=1}^n kla_{kl} \big(\sum_\mu H_s(\theta_\mu)\sin(k\theta_\mu) \big) \big( \sum_\nu H_s(\theta_\nu)\sin(l\theta_\nu) \big) \nonumber \\ 
&+(1-\beta/2)\frac{1}{n}\sum_\mu (H'(\theta_\mu)+H_s(\theta_\mu)\cot\theta_\mu), 
\end{align}
and
\begin{multline}\label{Vns}
V_n^s(\theta) = \frac{1}{n^2} \sum_\mu H_s(\theta_\mu)^2 \int_0^1 (1-\tau)\big(M''(\omega_\mu(\tau))+\frac{\d^2}{\d\tau^2}\log \big(1+\frac{\tau}{n}H'(\theta_\mu) \big) \big)\d\tau \\
+ \frac{1}{2}\int_0^1(1-\tau)^2f_{n,\theta}^{(3)}(\tau) \d\tau.
\end{multline}
With this notation \eqref{EnexpG2} can be written
\begin{equation}\label{EnexpG3}
\E_n^s\big[e^{\sum_\mu G(\theta_\mu)}\big] = \E_n^s\big[e^{R_n^s(\theta)+T_n^s(\theta)+U_n^s(\theta)+V_n^s(\theta)} \big]
\end{equation}
In order to prove \eqref{expbound} in Theorem \ref{thm:bound}, we have to bound the terms in the exponent in the right side of \eqref{EnexpG3}. We will first obtain bounds for any $\theta\in H_n$ in order to use Lemma \ref{Ensk} and restrict the domain of integration to $E_{n,s,K}$ with a suitable $K$, and then obtain better bounds for any $\theta\in E_{n,s,K}$. By \eqref{Grunskyexp} and \eqref{L},
\[ L(\theta) = -2\sum_{k,l\geq 1} a_{kl}\cos(k\theta)\cos(l\theta). \]
It follows from \eqref{Grunskydecay}, under our assumption on the regularity of $\gamma$, that
\[ \| L^{(r)} \|_\infty \leq C(\gamma), \]
for $0\leq r\leq 2$. Note that since $H_s(0) = H_s(\pi) =0$, 
\[ |H_s(\theta) \cot\theta | \leq C\| H_s' \|_\infty, \quad 0\leq \theta\leq \pi. \]
Furthermore, if we write $\theta_\nu' = 2\pi-\theta_\nu$, $\theta_\nu \in[0,\pi]$, then
\[ \frac{1}{n^2} \sum_{\mu,\nu} \frac{(H_s(\theta_\mu)+H_s(\theta_\nu))^2}{4\sin^2\frac{\theta_\mu+\theta_\nu}{2}} = \frac{1}{n^2} \sum_{\mu,\nu} \frac{(H_s(\theta_\mu)-H_s(\theta_\nu'))^2}{4\sin^2\frac{\theta_\mu-\theta_\nu'}{2}} \leq C \|H_s'\|_\infty^2 \]
since $H_s(2\pi-\theta)=-H_s(\theta)$, From these inequalities, \eqref{Hreg} and \eqref{Grunskydecay}, we see that
\begin{equation}\label{Unsbound}
\| U_n^s\|_\infty \leq C \big( \|G'\|_\infty+\|L'\|_\infty \big) \|H_s \|_\infty + C\|H'_s\|_\infty + C\|H_s'\|_\infty^2+ C(\gamma)\|H_s\|_\infty^2 \leq C(\gamma, G).
\end{equation}
Differentiating one more time in \eqref{fnbis} gives
\begin{align*}
 f_{n,\theta}^{(3)}(\tau) = &-\frac{1}{n^3} \sum_{\mu\neq\nu} \Big( \frac{\cos\frac{\omega_\mu-\omega_\nu}{2}}{4\sin^3\frac{\omega_\mu-\omega_\nu}{2}}(H_s(\theta_\mu)-H_s(\theta_\nu))^3+ \frac{\cos\frac{\omega_\mu+\omega_\nu}{2}}{4\sin^3\frac{\omega_\mu+\omega_\nu}{2}}(H_s(\theta_\mu)+H_s(\theta_\nu))^3 \Big) \\
& -\frac{4s}{n^3} \sum_{k,l=1}^n k^3a_{kl} \big(\sum_\mu H_s(\theta_\mu)^3\sin(k\omega_\mu) \big) \big( \sum_\nu \cos(l\omega_\nu) \big) \\
&-\frac{12s}{n^3} \sum_{k,l=1}^n k^2la_{kl} \big(\sum_\mu H_s(\theta_\mu)^2\cos(k\omega_\mu) \big) \big( \sum_\nu H_s(\theta_\nu) \sin(l\omega_\nu) \big).
\end{align*}
Also,
\[ M''(\omega) = \big(1-\frac{\beta}{2} \big)sL''(\omega)-\frac{1}{\sin^2\omega}+G''(\omega), \]
and
\[\frac{\d^2}{\d\tau^2}\log \big(1+\frac{\tau}{n}H'(\theta) \big) = -\frac{1}{n^2}\Big(\frac{H'(\theta)}{1+\frac{\tau}{n}H'(\theta)} \Big)^2. \]
We can use these formulas in \eqref{Vns} to see that we have a bound
\begin{equation}\label{Vnsbound}
\| V_n^s\|_\infty \leq \frac{C(\gamma,G)}{n}.
\end{equation}
Similarly, from \eqref{Rns} and \eqref{Tns} we get the estimates
\[\|R_n^s\|_\infty \leq C(\gamma,G) n, \quad \| T_n^s\|_\infty \leq C(\gamma,G) n. \]
Write
\[ X_n^s(\theta) = R_n^s(\theta)+T_n^s(\theta)+U_n^s(\theta)+V_n^s(\theta). \]
Then the bounds above show that $ \|X_n^s\|_\infty \leq C_1(\gamma,G)n $. Take $\beta K/2=C(\gamma)+C_1(\gamma,G)+1$ in the definition \eqref{Ensk} of $E_{n,s,K}$, where $C(\gamma)$ is the constant in \eqref{domainest}. Then by \eqref{domainest} and the bound on $\|X_n^s\|$,
\begin{equation}\label{Endiff}
\Big| \E_n^s[e^{\sum_\mu G(\theta_\mu)}]-\E_n^s[e^{\sum_\mu G(\theta_\mu)}\mathbbm{1}_{E_{n,s,K}}] \Big| \leq e^{-n}. 
\end{equation}
Below, when we write $E_{n,s,K}$, we will always mean this choice of $K$. 
To prove \eqref{expbound} we need better estimates of $R_n^s$ and $T_n^s$ on $E_{n,s,K}$. If we define $t_\mu$ as in \eqref{tmudef} then by Lemma \ref{lem:sumtmu}, we have that
\begin{equation}\label{sumtmuest}
\sum_\mu t_\mu^2 \leq C(\gamma, G)
\end{equation} 
if $\theta\in E_{n,s,K}$. Using \eqref{sumcos} and \eqref{sumtmuest} we get the estimate
\begin{align}\label{sumcostheta}
\big| \sum_\mu \cos(k\theta_\mu) \big| &\leq 1 + \big| \sum_\mu \cos(k\theta_\mu)-\cos(k\alpha_\mu) \big| \leq 1+k\sum_\mu |t_\mu| \nonumber \\
&\leq 1+k\sqrt{n}\sqrt{\sum_\mu t_\mu^2} \leq C(\gamma,G) k \sqrt{n},
\end{align}
for some constant $C(\gamma,G)$.
We first estimate $T_n^s$ in $E_{n,s,K}$.

\begin{lemma}\label{lem:Tnsest}
There is a constant $C(\gamma,G)$ such that
\begin{equation}\label{Tnest2}
|T_n^s(\theta)| \leq C(\gamma,G),
\end{equation}
for all $n\geq 1$, $s\in[0,1]$ and $\theta \in E_{n,s,K}$.
\end{lemma}

\begin{proof}
A Riemann sum estimate gives
\begin{align*}
\big| h_k-\frac{2}{n}\sum_\mu H_s(\alpha_\mu) \sin(k\alpha_\mu) \big| &= \big| \frac{2}{\pi} \int_0^\pi H_s(\theta) \sin(k\theta)\d\theta-\frac{2}{n}\sum_\mu H_s(\alpha_\mu)\sin(k\alpha_\mu) \big| \\
&\leq \frac{C}{n} \big\| \frac{\d}{\d\theta}\big( H_s(\theta) \sin(k\theta) \big) \big\|_\infty \leq C(\gamma,G) \frac{k}{n}.
\end{align*}
Also, for $\theta\in E_{n,s,K}$,
\begin{align*}
\frac{2}{n} \big| \sum_\mu H_s(\theta_\mu) \sin(k\theta_\mu) - \sum_\mu H_s(\alpha_\mu)\sin(k\alpha_\mu) \big| &\leq \frac{C}{n}(\|H\|_\infty k +\|H'\|_\infty) \sum_\mu |t_\mu| \\
&\leq C(\gamma,G) \frac{k}{\sqrt{n}}.    
\end{align*}
Hence, by \eqref{Tns} and \eqref{sumcostheta},
\[ |T_n^s(\theta)| \leq C(\gamma,G) \sum_{k,l\geq 1} k |a_{kl}| \frac{k}{\sqrt{n}}l\sqrt{n} \leq C(\gamma,G) \]
where we used \eqref{Grunskydecay}.
\end{proof}

Next, we estimate $R_n^s$ in $E_{n,s,K}$.

\begin{lemma}\label{lem:Rnsest}
There is a constant $C(\gamma,G)$ such that
\begin{equation}\label{Rnest2}
|R_n^s(\theta)| \leq C(\gamma,G),
\end{equation}
for all $n\geq 1$, $s\in[0,1]$ and $\theta\in E_{n,s,K}$.
\end{lemma}

\begin{proof}
From \eqref{Intekv4} we have that
\[ \sum_\mu G(\theta_\mu) = -\beta s \sum_{k,l\geq 1} ka_{kl}h_k^s\big(\sum_\mu \cos(k\theta_\mu) \big)+\beta\sum_\mu \widetilde{H}_s(\theta_\mu). \]
Thus by \eqref{Rns}
\begin{multline}\label{Rnsformula}
R_n^s(\theta) = \frac{\beta}{4n} \sum_{\mu,\nu} \big(\cot\frac{\theta_\mu-\theta_\nu}{2}(H_s(\theta_\mu)-H_s(\theta_\nu))+\cot\frac{\theta_\mu+\theta_\nu}{2}(H_s(\theta_\mu)+H_s(\theta_\nu)) \big) \\
+\beta \sum_\mu \widetilde{H}_s(\theta_\mu)-\beta s\sum_{\max{k,l}> n} k a_{kl} h_k^s\sum_\mu \cos(l\theta_\mu).
\end{multline}
By \eqref{Grunskydecay},
\begin{equation}\label{Rnsextra}
\big| \beta s \sum_{\max{k,l}>n} k a_{kl} h_k^s\sum_\mu \cos(l\theta_\mu) \big| \leq C(\gamma)  \|H\|_\infty n\sum_{k\geq 1}\sum_{l\geq n+1} \frac{1}{k^{1+\alpha/2}l^{2+\alpha/2}} \leq \frac{C(\gamma,G)}{n^{\alpha/2}}.
\end{equation}
We have that 
\begin{equation}\label{Hstilde}
\sum_\mu \widetilde{H}_s(\theta_\mu) = - \sum_{k\geq 1} h_k^s \sum_\mu \cos(k\theta_\mu),
\end{equation}
\[ \cot\frac{\theta_\mu-\theta_\nu}{2}(H_s(\theta_\mu)-H_s(\theta_\nu)) = Im \sum_{k\geq 1} h_k^s \cot\frac{\theta_\mu-\theta_\nu}{2} (e^{\i k\theta_\mu}-e^{\i k\theta_\nu}). \]
Now, for $k\geq 1$,
\[ \cot\frac{x-y}{2}(e^{\i kx}-e^{\i ky}) = i(e^{\i kx}+e^{\i k y}+2\sum_{j=1}^{k-1} e^{\i(jx+(k-j)y)}) \]
and hence
\[ \cot\frac{\theta_\mu-\theta_\nu}{2}(H_s(\theta_\mu)-H_s(\theta_\nu)) =\sum_{k\geq 1} h_k^s \big( \cos\theta_\mu +\cos\theta_\nu + 2 \sum_{j=1}^{k-1}\cos(j\theta_\mu+(k-j)\theta_\nu) \big). \]
Since $H_s$ is odd, we see that, by replacing $\theta_\nu$ with $-\theta_\nu$,
\[ \cot\frac{\theta_\mu+\theta_\nu}{2}(H_s(\theta_\mu)+H_s(\theta_\nu)) =\sum_{k\geq 1} h_k^s \big( \cos\theta_\mu +\cos\theta_\nu + 2 \sum_{j=1}^{k-1}\cos(j\theta_\mu-(k-j)\theta_\nu) \big). \]
Adding these two equations gives
\begin{align*}
&\frac{1}{4n} \sum_{\mu,\nu} \Big( \cot\frac{\theta_\mu-\theta_\nu}{2}(H_s(\theta_\mu)-H_s(\theta_\nu))+\cot\frac{\theta_\mu+\theta_\nu}{2}(H_s(\theta_\mu)+H_s(\theta_\nu)) \Big) \\
&= \sum_{k\geq 1} h_k^s\sum_\mu \cos(k\theta_\mu) + \frac{1}{n}\sum_{k\geq 1}h_k^s\sum_{j=1}^{k-1} \sum_\mu \cos(j\theta_\mu) \sum_\nu \cos((k-j)\theta_\nu).
\end{align*}
Thus, it follows from \eqref{sumcostheta}, \eqref{Rnsformula}, \eqref{Rnsextra} and \eqref{Hstilde} that
\begin{align}\label{Rnsbound}
|R_n^s(\theta)| &\leq \frac{\beta}{n} \sum_{k\geq1} h_k^s\sum_{j=1}^{k-1} \big|\sum_\mu \cos(j\theta_\mu) \big| \big| \sum_\nu \cos((k-j)\theta_\nu) \big| +\frac{C(\gamma,G)}{n^{\alpha/2}} \nonumber \\
&\leq C(\gamma,G) \sum_{k\geq 2} \sum_{j=1}^{k-1} |h_k^s| j(k-j) \leq C(\gamma,G)
\end{align}
where we used that $|h_k^s| \leq C(\gamma, G) k^{-4-\alpha}$ by Lemma \ref{lem:Hreg}. 
\end{proof}

Combining \eqref{Unsbound}, \eqref{Vnsbound}, \eqref{Tnest2} and \eqref{Rnest2} we get a bound
\begin{equation}\label{Xnsbound}
\big| X_n^s(\theta) \mathbbm{1}_{E_{n,s,K}}(\theta) \big| \leq C(\gamma,G).
\end{equation}
Combining this with \eqref{Endiff} we have proved \eqref{expbound} in Theorem \ref{thm:bound}, so it remains to prove \eqref{relSz2}. This will follow from the next lemma.

\begin{lemma}\label{lem:termstozero}
For each fixed $s\in[0,1]$ we have the limits
\begin{equation}\label{Unslim}
\lim_{n\to\infty} \E_n^s \big[ |U_n^s(\theta)-A_s[g]| \mathbbm{1}_{E_{n,s,K}} \big] =0,
\end{equation}
\begin{equation}\label{RnsTnslim}
\lim_{n\to\infty} \E_n^s \big[ |R_n^s(\theta)+T_n^s(\theta)| \mathbbm{1}_{E_{n,s,K}} \big] =0,
\end{equation}
\end{lemma}

If we assume the lemma we can prove \eqref{relSz2}. It follows from \eqref{Vnsbound}, \eqref{Unslim}, and \eqref{RnsTnslim} that
\begin{equation}\label{Xnslim}
\lim_{n\to\infty} \E\big[|X_n^s(\theta)-A_s[g]|\mathbbm{1}_{E_n} \big] =0
\end{equation}
where $E_n= E_{n,s,K}$. Now, since $\P_n^s[E_n^c] \to 0$ as $n\to \infty$.
\[\lim_{n\to\infty} \big|\E_n^s\big[e^{X_n^s(\theta)-A_s[g]}\mathbbm{1}_{E_n} \big]-1 \big| = \lim_{n\to\infty} \big|\E_n^s\big[(e^{X_n^s(\theta)-A_s[g]}-1)\mathbbm{1}_{E_n} \big] \big|.\]
We can now use \eqref{Xnsbound} to see that
\[  \big|\E_n^s\big[(e^{X_n^s(\theta)-A_s[g]}-1)\mathbbm{1}_{E_n} \big] \big| \leq  \E_n^s\big[|X_n^s(\theta)-A_s[g]|e^{C(\gamma,G)+|A_s[g]|}\mathbbm{1}_{E_n} \big] \]
which converges to zero as $n\to\infty$ by \eqref{Xnslim}. Combining this with \eqref{Endiff} completes the proof of \eqref{relSz2}.

\begin{proof}[Proof of Lemma \ref{lem:termstozero}]
It follows from \eqref{expbound} that there is a positive, strictly increasing function $f(k)$, which does not depend on $n$, such that
\[ \E_n^s\big[ \big(\sum_\mu \cos(k\theta_\mu) \big)^2\big] \leq f(k)^2, \]
for $k\geq 1$. Also, we can take $G=0$ and then \eqref{sumcostheta} gives
\[ \E_n^s\big[ \big(\sum_\mu \cos(k\theta_\mu) \big)^2\big] \leq C(\gamma)k^2n. \]
Fix some $p\geq 2$. It follows from \eqref{Rnsbound} that
\begin{multline}\label{Rnsest3}
\E_n^s[|R_n^s(\theta)|\mathbbm{1}_{E_n}] \leq  \frac{\beta}{n} \sum_{k\geq 1} |h_k^s| \sum_{j=1}^{k-1} \E_n^s\big[ \big( \sum_\mu \cos(j\theta_\mu) \big)^2 \big]^{1/2}\E_n^s\big[ \big( \sum_\nu \cos((k-j)\theta_\nu) \big)^2 \big]^{1/2} \\ +\frac{C(\gamma,G)}{n^{\alpha/2}} 
\leq \frac{\beta}{n} \sum_{k=1}^p |h_k^s| \sum_{j=1}^{k-1} f(j)f(k-j) + C(\gamma)\sum_{k\geq p+1} 
k^3|h_k^s| + \frac{C(\gamma,G)}{n^{\alpha/2}}.
\end{multline}
We have that
\[ |h_k^s| \leq C \frac{\|H\|_{4,\alpha}}{k^{4+\alpha}} \leq C \frac{\| G \|_{4,\alpha}}{k^{4+\alpha}}, \]
and thus
\[ C(\gamma) \sum_{k\geq p+1} k^3 |h_k^s| \leq \frac{C(\gamma,G)}{p^{\alpha}}. \]
Let $F(k) = \sum_{j=1}^{k-1} f(j)f(k-j)$. Then $F$ is strictly increasing. Take $\delta \in (0,1)$ and let $p = \lfloor F^{-1}(\lfloor n^\delta \rfloor ) \rfloor$. If $k\leq p$ then $F(k) \leq F(p) \leq n^{\delta}$ and
\[ \frac{\beta}{n}\sum_{k=1}^p |h_k^s| \sum_{j=1}^{k-1} f(j)f(k-j) \leq \frac{C(\gamma,G)}{n^{1-\delta}}. \]
We can use these estimates in \eqref{Rnsest3} to see that $\lim_{n\to\infty} \E_n^s[R_n^s(\theta)\mathbbm{1}_{E_n}] =0$. The proof that $\lim_{n\to\infty} \E_n^s[T_n^s(\theta)\mathbbm{1}_{E_n}] =0$ is similar, and gives \eqref{RnsTnslim}.

We turn to the proof of \eqref{Unslim}. It follows from Lemma \eqref{lem:tmutozero} and \eqref{Uns} that 
\begin{align}\label{Unslimit}
&\lim_{n\to\infty} \E_n^s[U_n^s(\theta)] = \frac{1}{\pi} \int_0^\pi (G'(\theta)+(1-\beta/2)sL'(\theta))H_s(\theta) \d\theta \\
&+ (1-\beta/2)\frac{1}{\pi}\int_0^\pi (H_s'(\theta)+H_s(\theta) \cot\theta) \d\theta \nonumber \\
& -\frac{\beta}{16\pi^2} \int_0^\pi \int_0^\pi  \frac{(H_s(\theta)-H_s(\omega))^2}{\sin^2\frac{\theta-\omega}{2}}+\frac{(H_s(\theta)+H_s(\omega))^2}{\sin^2\frac{\theta+\omega}{2}} \d\theta\d\omega -\frac{\beta s}{4} \sum_{k,l\geq 1}kl a_{kl} h_k^sh_l^s.
\end{align}
We have to show that the expression in the right side of \eqref{Unslimit} agrees with $A_s[g]$ as defined by \eqref{Asg}. Note that since $H_s(0) = H_s(\pi) = 0$, the integral of $H_s'$ is equal to zero. If we use the integral
\begin{equation}
\frac{1}{\pi} \int_0^\pi \frac{\sin((k+1)\theta)}{\sin\theta} \d\theta = \begin{cases} 0\ \mathrm{if}\ k\ \mathrm{odd} \\ 1\ \mathrm{if}\ k\ \mathrm{even}.  \end{cases} 
\end{equation}
We obtain
\begin{multline*}
\frac{1}{\pi}\int_0^\pi H_s(\theta)\cot\theta \d\theta =  \sum_{k\geq 1} h_k\frac{1}{\pi}\int_0^\pi \frac{\sin(k\theta)\cos\theta}{\sin\theta} \d\theta \\
= \sum_{k\geq 1} h_k\frac{1}{\pi}\int_0^\pi \frac{\sin((k+1)\theta)-\sin\theta\cos(k\theta)}{\sin\theta} \d\theta = \sum_{n\geq 1} h_{2n}. 
\end{multline*}
Observe that
\begin{align}\label{hkhl}
-\frac{\beta s}{4} \sum_{k,l\geq 1} kl a_{kl} h_k^s h_l^s &= \frac{\beta s}{2\pi} \sum_{k,l\geq 1} k a_{kl} h_k^s \int_0^\pi H_s(\omega) \frac{\d}{\d \omega} \cos(l\omega) \d\omega \nonumber \\
&= -\frac{\beta s}{2\pi} \int_0^\pi \big( \sum_{k,l\geq 1} ka_{kl}h_k^s\cos(l\omega) \big) H_s'(\omega) \d\omega.
\end{align}
Next,
\begin{align}\label{Hquad}
&-\frac{\beta}{16\pi^2} \int_0^\pi \int_0^\pi \frac{(H_s(\theta)-H_s(\omega))^2}{\sin^2\frac{\theta-\omega}{2}}+\frac{(H_s(\theta)+H_s(\omega))^2}{\sin^2\frac{\theta+\omega}{2}} \d\theta\d\omega \nonumber \\
&= \frac{\beta}{8\pi^2} \int_0^\pi \d\omega \int_0^\pi \d\theta \Big( \big( H_s(\theta)-H_s(\omega)\big)^2 \frac{\partial}{\partial\theta} \cot \frac{\theta-\omega}{2}+ \big( H_s(\theta)+H_s(\omega)\big)^2 \frac{\partial}{\partial\theta} \cot \frac{\theta+\omega}{2} \Big) \nonumber \\
&= -\frac{\beta}{4\pi^2} \int_0^\pi \d\omega \int_0^\pi \d\theta \Big( \big( H_s(\theta)-H_s(\omega)\big)\cot \frac{\theta-\omega}{2}+\big( H_s(\theta)+H_s(\omega)\big)\cot \frac{\theta+\omega}{2} \Big)H'(\theta) \nonumber \\
&= -\frac{\beta}{4\pi^2} \int_0^\pi H_s'(\theta) \Big( \int_0^\pi H_s(\omega)\big( \cot\frac{\omega-\theta}{2}+\cot\frac{\omega+\theta}{2} \big) \d\omega \Big) \d\theta
\end{align}
by an integration by parts and the fact that $\int_{-\pi}^\pi \cot\frac{\theta-\omega}{2} \d\omega =0$. Combining \eqref{Intekv3:prel} and \eqref{Conjcomp} we see that
\[ G(\theta) = -\frac{\beta}{2\pi} \int_0^\pi H_s(\omega) \big(\cot\frac{\omega-\theta}{2}+\cot\frac{\omega+\theta}{2} \big) \d\omega - \beta s \sum_{k,l\geq 1} k a_{kl} h_k \cos(l\theta). \]
Hence, the sum of \eqref{hkhl} and \eqref{Hquad} equals
\[ \frac{1}{2\pi} \int_0^\pi H_s'(\theta)G(\theta) \d\theta. \]
From these computations we see that \eqref{Unslimit} can be written
\begin{equation}\label{Unslimit2}
\lim_{n\to\infty} \E_n^s[U_n^s(\theta)] = -\frac{1}{2\pi} \int_0^\pi G(\theta)H_s'(\theta) \d\theta - \big( 1-\frac{\beta}{2}\big) \frac{s}{\pi} \int_0^\pi L(\theta)H_s'(\theta) \d\theta + \big( 1-\frac{\beta}{2}\big)\sum_{n\geq 1} h_{2n}.
\end{equation}
Recall the notation \eqref{xtheta} which allow us to write
\[G(\theta) = \mathbf{x}_\theta^t\mathbf{g}, \quad H_s(\theta) = \mathbf{y}_\theta^t\mathbf{h}^s. \]
Also, by \eqref{logzprime}, \eqref{dvec} and \eqref{L},
\[ L(\theta) = -\mathbf{x}_\theta^t \mathbf{d}. \]
Let $J = (k\delta_{jk})_{j,k\geq 1}$. Then,
\[ H_s'(\theta) = \mathbf{x}_\theta^t J \mathbf{h}^s. \]
By \eqref{hsol},
\[ \mathbf{h}^s = -\frac{1}{\beta} (I+sB)^{-1}\mathbf{g}, \]
so
\[ H_s'(\theta) = -\frac{1}{\beta} \mathbf{x}_\theta^tJ(I+sB)^{-1}\mathbf{g}. \]

Thus,
\begin{align*}
&-\frac{1}{2\pi} \int_0^\pi G(\theta)H_s'(\theta) \d\theta - \big(1-\frac{\beta}{2}\big)\frac{s}{\pi} \int_0^\pi L(\theta)H_s'(\theta) \d\theta \\
&= \frac{1}{2\pi\beta} \int_0^\pi \mathbf{g}^t\mathbf{x}_\theta\mathbf{x}_\theta^tJ(I+sB)^{-1}\mathbf{g} \d\theta - \big(1-\frac{\beta}{2}\big) \frac{s}{\pi\beta} \int_0^\pi \mathbf{d}^t\mathbf{x}_\theta\mathbf{x}_\theta^tJ(I+sB)^{-1} \mathbf{g} \d\theta \\
&= \frac{1}{4\beta} \mathbf{g}^t(I+sB)^{-1}\mathbf{g}+\frac{1}{4}\big(1-\frac{2}{\beta}\big) s\mathbf{d}^t(I+sB)^{-1}\mathbf{g},
\end{align*}
since $\frac{2}{\pi}\int_0^\pi \mathbf{x}_\theta\mathbf{x}_\theta^t \d\theta = J^{-1}$.
With $\mathbf{f}$ defined as in \eqref{fvec}, we obtain
\[ (1-\frac{\beta}{2})\sum_{n\geq 1} h_{2n} = (1-\frac{\beta}{2})\frac{1}{2}\mathbf{f}^t\mathbf{h}^s = \frac{1}{4}\big(1-\frac{2}{\beta}\big) \mathbf{f}^t(I+sB)^{-1}\mathbf{g}. \]

Inserting these last two computations into \eqref{Unslimit2}, we see that the right side equals $A_s[g]$ and we proved \eqref{relSz2}.

\end{proof}

\appendix
\section{Proof of Proposition~\ref{prop:ellen}}\label{app:Ellen}
The result and proof are due to Ellen Krusell, see also \cite{krusell-PA}.
 
 \begin{proof}[Proof of Proposition~\ref{prop:ellen}]We start by changing coordinates using $m(z)=(z+1)/(z-1)$. We denote by $\hat \gamma=m(\gamma)$ and $\hat \eta=m(\eta)$. Then $\hat \eta$ is a Jordan curve through $\infty$ and it is the image of $\hat \gamma$ under (the two-valued function) $z\mapsto \sqrt{z}$. Indeed, we have that
    $$m^{-1}(\sqrt{m(z)})=z+\sqrt{z^2-1}.$$
    Moreover, if we let $\tilde h=m^{-1}\circ F \circ m$ then 
    $$|\tilde h'(0)|=|F'(-1)|\quad\text{and}\quad|\tilde h'(\infty)|=|F'(1)|^{-1}$$
    where we define $|\tilde h'(\infty)|$ by $\lim_{|z|\to\infty}|\tilde h'(z)|$.  We can now rephrase the statement we want to prove as
    $$2I^A(\hat \gamma)=I^L(\hat \eta)+6\log\bigg|\frac{\tilde h'(0)}{\tilde h'(\infty)}\bigg|,$$
    or equivalently
    \begin{equation}2\int_{\C\setminus\R^+}|\nabla\log|\hat{h}'||^2d^2z=\int_{\C\smallsetminus\R}|\nabla\log|\tilde {h}'||^2d^2z+6\pi\log\bigg|\frac{\tilde h'(0)}{\tilde h'(\infty)}\bigg|,\label{eq:identity}\end{equation}
    where $\hat h=(\tilde h(\sqrt{z}))^2$ is a conformal map from $\C\smallsetminus\R^+$ onto $\C\smallsetminus\hat \gamma$. (Note that $\hat{h}$ maps $\mathbb{R}_-$ onto the hyperbolic geodesic connecting $0$ with $\infty$ in $\mathbb{C} \smallsetminus \hat \gamma$.) 

    We use the notation $\sigma_f(z):=\log|f'(z)|$ for holomorphic $f$. Set $s(z)=z^2$. Then since $\hat{h}'(z^2) = \tilde{h}(z) \cdot \tilde{h}'(z)/z$, we have
    $$\sigma_{\hat{h}}\circ s=\sigma_s \circ \tilde h +\sigma_{\tilde h}-\sigma_s$$
    on $\mathbb{H}$. Let $A_{r} = \{r<|z|<1/r\} \smallsetminus \R$ for $r\in(0,1)$. By conformal invariance of the Dirichlet inner product we have
    \begin{equation}\int_{(A_r)^2}|\nabla\sigma_h|^2d^2z = \int_{A_r}|\nabla\sigma_{\tilde h}|^2d^2z+\int_{A_r}\nabla(\sigma_s \circ \tilde h-\sigma_s)\cdot\nabla(\sigma_s \circ \tilde h+2\sigma_{\tilde h}-\sigma_s) d^2z.\label{eq:annulusidentity}\end{equation}
    (Here we regard $(A_r)^2$ as two copies of $\{r^2<|z|<1/r^2\}\smallsetminus\R^+$.)
    Since $\sigma_s \circ \tilde h$, $\sigma_{\tilde h}$ and $\sigma_s$ are smooth on $A_{r}$ we may expand the integrand of the second integral on the right, and apply Stokes' theorem to each term separately. We also use the formula
    \begin{equation}
        k_sd\ell_s = (\partial_n\sigma_s+k)d\ell\label{eq:confchangecurv}
    \end{equation}
    where $k$, $\partial_n$, and $d\ell$ denote the geodesic curvature, outer unit normal derivative, and arc length measure respectively with respect to $d^2z$, and $k_s$ and $d\ell_s$ correspond to the metric $e^{2\sigma_s}d^2z$ (i.e., the pull-back metric $s^*d^2z$).

    Using Stokes' theorem and \eqref{eq:confchangecurv}, we obtain,
    \begin{align*}\int_{A_r}&\nabla(\sigma_s \circ \tilde h-\sigma_s)\cdot\nabla(\sigma_s \circ \tilde h+2\sigma_{\tilde h}-\sigma_s) d^2z\\ =& -2\int_{\partial A_r}\sigma_{\hat h}\circ s \, k_s \, d\ell_s + 2\int_{\partial A_r}\sigma_{\tilde h}kd\ell \\&+ \bigg(\int_{\tilde h(\partial A_r)}-\int_{\partial A_r}\bigg)\sigma_sk_sd\ell_s+\bigg(\int_{\tilde h(\partial A_r)}-\int_{\partial A_r}\bigg)\sigma_skd\ell.\end{align*}

   We now analyze the terms on the right in the last display, as $r\to 0+$. We claim that
    \begin{align}\label{eq:r-to-0}\lim_{r\to 0+}\int_{A_r}\nabla(\sigma_s \circ \tilde h-\sigma_s)\cdot\nabla(\sigma_s \circ \tilde h+2\sigma_{\tilde h}-\sigma_s) d^2z = 6\pi(\sigma_{\tilde h}(0)-\sigma_{\tilde h}(\infty)).\end{align}
    Given this, by taking the limit $r\to 0+$ on both sides of~\eqref{eq:annulusidentity}, $$2\int_{\C\smallsetminus\R^+}|\nabla\sigma_h|^2d^2z=\int_{\C\smallsetminus\R}|\nabla\sigma_{\tilde h}|^2d^2z+6\pi(\sigma_{\tilde h}(0)-\sigma_{\tilde h}(\infty)),$$
    which is equivalent to~\eqref{eq:identity}. It remains to prove \eqref{eq:r-to-0}.
   We start by noting that 
    $$\lim_{|z|\to 0}\sigma_{\tilde h}(z)=\sigma_{\tilde h}(0),$$
    $$\lim_{|z|\to \infty}\sigma_{\tilde h}(z)=\sigma_{\tilde h}(\infty),$$
    by the smoothness assumption on $\hat \eta$. Next, the boundary of $A_r$ consists of $r\partial\D$ (clockwise oriented), $\frac{1}{r}\partial\D$ (counterclockwise oriented), and line segments. As $k=0$ on the line segments, we have that
    $$\int_{\partial A_r}\sigma_{\tilde h}kd\ell = \int_0^{2\pi}\sigma_{\tilde h}(r^{-1}e^{i\t})rr^{-1}d\t +\int_0^{2\pi}\sigma_{\tilde h}(re^{i\t})(-r^{-1})rd\t \to 2\pi(\sigma_{\tilde h}(\infty)-\sigma_{\tilde h}(0)). $$
  Similarly, 
    $$\lim_{|z|\to 0}\sigma_h(s(z))=\lim_{|z|\to 0}\log\bigg|\frac{\tilde h(z)}{z}\tilde h'(z)\bigg|=2\sigma_{\tilde h}(0)$$
    $$\lim_{|z|\to \infty}\sigma_h(s(z))=\lim_{|z|\to \infty}\log\bigg|\frac{\tilde h(z)}{z}\tilde h'(z)\bigg|=2\sigma_{\tilde h}(\infty).$$
    We have,
    $$\int_{\partial A_r}\sigma_{\hat h}\circ s k_sd\ell_s = \int_0^{4\pi} \sigma_h(r^{-2}e^{i\t})r^2r^{-2}d\t +\int_0^{4\pi}\sigma(r^2 e^{i\t})(-r^{-2})r^2d\t \to 8\pi(\sigma_{\tilde h}(\infty)-\sigma_{\tilde h}(0)),$$
    since in $(\partial A_r)^2$ the circle $r^2\partial \D$ is traversed twice clockwise and $r^{-2}\partial\D$ is traversed twice counterclockwise.
 Next, we consider 
    $$\lim_{r \to 0+}\bigg(\int_{\tilde h(\partial A_r)}-\int_{\partial A_r}\bigg)\sigma_s kd\ell.$$
    We first observe that we may restrict to integration along $r\partial \mathbb{D}$, $r^{-1}\partial \mathbb{D}$ and their images under $\tilde h$: Let $L_r=\partial A_r\cap \R$, interpreted as the union of four line segments, that is, $[r,r^{-1}]$ and $[-r^{-1},-r]$ traversed once from right to left (as a subset of the boundary of $\mathbb{H}$) and traversed once from left to right (as boundary of $\mathbb{H}^*$). Consider $\tilde h(L_r)$. From the parts of $\tilde h(L_r)$ where there is overlap from the upper and lower components (i.e., the parts of $\hat \eta$ which are traversed in both directions from both bottom and top) the contributions to the integrals cancel. We are left with the integral over the parts of $\hat \eta$ which are traversed only from one side (top/bottom). Close to $0$ these parts of $\hat \eta$ have a length of order $O(r)$, the curvature is of order $O(1)$, and $\sigma_s$ evaluated along these parts of $\hat \eta$ is of order $O(\log r)$. Similarly, close to $\infty$, the length of the part of $\hat \eta$ under consideration is of order $O(r^{-1})$, the curvature is of order $O(r^{-2})$, and $\sigma_s$ is of order $O(\log r)$ (this can be seen by changing coordinates by $z\mapsto 1/z$). This gives
    $$\int_{\tilde h(L_r)}\sigma_s k d\ell \to 0$$
    as $r\to 0+$. We have,
    $$\bigg(\int_{\tilde h(r\partial \D)}-\int_{r\partial \D}\bigg)\sigma_s kd\ell = \int_{r\partial \D}(\sigma_s \circ \tilde hk_{\tilde h}d\ell_{\tilde h}-\sigma_s k d\ell) = \int_{r\partial \D}(\sigma_s \circ \tilde h(\partial_n\sigma_{\tilde h}+k)-\sigma_s k)d\ell.$$
    Using that
    $$\sigma_s(\tilde h(re^{i\t}))=\log2|\tilde h(re^{i\t})|=\log2(|\tilde h'(0)|r+O(r^2))=\log2r +\sigma_{\tilde h}(0) +O(r),$$
    and
    $$|\nabla\sigma_{\tilde h}(re^{i\t})|=O(1),$$
    we obtain
    $$\bigg(\int_{\tilde h(r\partial \D)}-\int_{r\partial \D}\bigg)\sigma_s kd\ell\to -2\pi\sigma_{\tilde h}(0).$$
    Similarly, we find
    $$\bigg(\int_{\tilde h(r^{-1}\partial \D)}-\int_{r^{-1}\partial \D}\bigg)\sigma_s kd\ell\to 2\pi\sigma_{\tilde h}(\infty),$$
    using 
    $$\sigma_s(\tilde h(r^{-1}e^{i\t}))=\log2(|\tilde h'(\infty)|r^{-1}+O(1))=\log(2/r)+\sigma_{\tilde h}(\infty)+O(r),$$
    $$|\nabla \sigma_{\tilde h}(r^{-1}e^{i\t})|=O(r^2).$$
    (For the last estimate we used the regularity of $1/\tilde h(1/z)$ at $0$). Combining the above shows
    $$\bigg(\int_{\tilde h(\partial A_r)}-\int_{\partial A_r}\bigg)\sigma_s kd\ell\to 2\pi(\sigma_{\tilde h}(\infty)-\sigma_{\tilde h}(0)).$$
    Finally, we consider 
    $$\bigg(\int_{\tilde h(\partial A_r)}-\int_{\partial A_r}\bigg)\sigma_s k_sd\ell_s$$
    as $r \to 0$,
    which is treated in a similar manner. We first note that, for the same reason as above, there is no contribution to the limit from integration along $\tilde h(L_r)$ and $L_r$ (here we again use $k_sd\ell_s = (\partial_n\sigma_s + k)d\ell$ and that $\partial_n\sigma_s(re^{i\t})=O(1)$ along $\hat \eta$ and similarly at $\infty$). Moreover, arguing as above,
    $$\bigg(\int_{\tilde h(r\partial \D)}-\int_{r\partial \D}\bigg)\sigma_s k_sd\ell_s\to -4\pi\sigma_{\tilde h}(0),$$
    $$\bigg(\int_{\tilde h(r^{-1}\partial \D)}-\int_{r^{-1}\partial \D}\bigg)\sigma_s k_sd\ell_s\to 4\pi\sigma_{\tilde h}(\infty),$$
    and this completes the proof.
\end{proof}

\bibliography{ref}
\bibliographystyle{plain}

\end{document}